\documentclass[a4paper,12pt]{article}
\usepackage{tikz}
\usetikzlibrary{cd}
\usepackage{graphicx,amsmath,amssymb,amsthm,amscd,mathrsfs}
\usepackage[margin=2.5cm]{geometry}

\DeclareMathOperator{\pr}{pr}

\newtheorem{thm}{Theorem}[section]
\newtheorem{dfn}[thm]{Definition}
\newtheorem{prp}[thm]{Proposition}
\newtheorem{lmm}[thm]{Lemma}
\newtheorem{crl}[thm]{Corollary}
\newtheorem{rem}[thm]{Remark}

\title{Long-Moody construction of braid representations and Katz middle convolution}          

\date{}
\author{Kazuki Hiroe\footnote{
	K. H. is supported byJSPS KAKENHI Grant Number 20K03648. Department of Mathematics and Informatics, Chiba University 
	1-33, Yayoi-cho, Inage-ku, Chiba-shi, Chiba, 263-8522 JAPAN 
	email: {\tt kazuki@math.s.chiba-u.ac.jp}} and Haru Negami \footnote{H. N. is supported by the ANRI Fellowship. Department of Mathematics and Informatics, Chiba University 
	1-33, Yayoi-cho, Inage-ku, Chiba-shi, Chiba, 263-8522 JAPAN 
	email: {\tt 21wm0102@student.gs.chiba-u.jp}}
} 

\begin{document}
\maketitle
\begin{abstract}      
	The Long-Moody construction is a method 
	to obtain representations of braid groups
	introduced by Long and Moody.
	Also the Katz middle convolution is 
	known to be a  
	method to construct local systems on 
	$\mathbb{C}\backslash\{n\text{-points}\}$
	introduced by Katz.
	In this paper, we explain that 
	these two methods 
	are naturally  unified and 
	define a new functor which we call the Katz-Long-Moody functor.
	This functor extends the framework of  Katz algorithm
	to categories of local systems on various topological spaces,
	for example, $B_{n}$-bundles 
	associated with simple Weierstrass polynomials, 
	complements of hyperplane arrangements of fiber-type,
	link complements in the solid torus, and so on.
\end{abstract}
\tableofcontents      
\section*{Introduction}
Katz introduced in his famous book \cite{Katz}
a machinery constructing  
local systems on 
$\mathbb{C}\backslash\{n\text{-points}\}$,
and this 
efficient machinery and 
the algorithm for constructing local systems
are called Katz middle convolution 
and Katz algorithm.
This machinery plays 
important roles 
not only for number theory,
but also for 
various areas in mathematics.

At the same time,
Long and Moody introduced 
a method to construct representations of 
braid groups in \cite{Long} as 
a natural generalization of Burau representation.
This Long-Moody construction 
is known to be an efficient 
way to obtain 
braid group representations. For example,
Burau representation, Gassner representation,
Jones representations,
and etc., are obtained through 
the Long-Moody construction, see \cite{Long}, \cite{Big}.
Also it has many applications for 
knot invariants, see \cite{Con} and \cite{Tak},
and homological stability, see \cite{Sou},
and so on.

This paper explains that 
these two seemingly different 
machineries by Katz and Long-Moody 
are naturally unified 
through the twisted homology theory.
This enable us to 
establish the Katz algorithm not only for 
local systems on $\mathbb{C}\backslash\{n\text{-points}\}$
but also for those on many other topological spaces
whose fundamental groups are related to braid groups.

Let us explain our results more precisely.
Let us take a pair $(G,\alpha)$
of a group $G$ and a homomorphism 
$\alpha\colon G\rightarrow \mathrm{Aut}_{\mathbf{Grp}}(F_{n})$,
which factors through  the Artin representation $\theta\colon B_{n}\hookrightarrow \mathrm{Aut}_{\mathbf{Grp}}(F_{n})$.
Then we can  
define the semi-direct product $F_{n}\rtimes_{\alpha}G$.
In particular, when $G$ is a free group  $F_{s}$ of rank $s$,
the map $\alpha\colon F_{s}\rightarrow B_{n}$
is called the braid monodromy map 
associated with $\rho:=\theta\circ \alpha\colon F_{s}\rightarrow \mathrm{Aut}_{\mathbf{Grp}}(F_{n})$,
which is introduce by Moishezon \cite{Moi} to obtain explicit descriptions 
of fundamental groups of plane algebraic curves, see also 
\cite{CS} by Cohen and Suciu.
More generally, when $G$ is a fundamental group 
of a topological space with the homotopy type of a CW-complex,
the group $F_{n}\rtimes_{\alpha}G$ appears as the 
fundamental group of the complement of a polynomial covering 
in the sense of Hansen \cite{Han}.

Then we shall construct the following functor as a generalization of the
Long-Moody construction of braid group representations.
\begin{thm}[Twisted Long-Moody functor]
	Let $k$ be a field and let us take $\lambda\in k^{\times}$.
	Then there exits an endfunctor $\mathcal{LM}_{\lambda}$ of the 
	category of left $k[F_{n}\rtimes_{\alpha}G]$-modules such that the following holds.
	\begin{enumerate}
		\item Suppose that $(G,\alpha)=(B_{n},\theta)$. Then 
		the functor $\mathrm{Res}^{k[F_{n}\rtimes_{\alpha} B_{n}]}_{k[B_{n}]}\circ \mathcal{LM}_{\lambda}$
		is isomorphic to the Long-Moody functor.

		\item Suppose that $G=\{e\}$ and identify $F_{n}\rtimes_{\alpha} \{e\}$ with $F_{n}$.
		Then the functor $\mathcal{LM}_{\lambda}$ is isomorphic to the convolution 
		functor introduce by Dettweiler-Reiter in \cite{DR}.
	\end{enumerate}
	Here $\mathrm{Res}^{R}_{P}$ for the rings $P\subset R$
	is the restriction functor from the category of left $R$-modules 
	to that of left $P$-modules.

	We call this functor $\mathcal{LM}_{\lambda}$ the twisted Long-Moody functor.
\end{thm}
The convolution functor appearing in the 
above was introduced by Dettweiler-Reiter
in preparation for 
defining the Katz middle convolution under their setting in \cite{DR}.
According to the definition of 
Katz middle convolution in \cite{DR},
we shall define a new 
functor from 
the twisted Long-Moody functor and show that 
this new functor
extends the Katz algorithm to $k[F_{n}\rtimes_{\alpha}G]$-modules.
\begin{thm}
	Let us take $\lambda\in k^{\times}$.
	There exists a full subcategory $\mathbf{Mod}^{\mathrm{NT}}_{k[F_{n}\rtimes_{\alpha}G]}$
	of the category of left $k[F_{n}\rtimes_{\alpha}G]$-modules and also exists
	an endfunctor $\mathcal{KLM}_{\lambda}$ of this subcategory such that 
	the following holds.
	\begin{enumerate}
		\item Suppose that $G=\{e\}$ and identify $F_{n}\rtimes_{\alpha} \{e\}$ with $F_{n}$.
		Then the functor $\mathcal{KLM}_{\lambda}$ is isomorphic to the Katz middle convolution 
		functor introduced in \cite{DR}.
		\item (Multiplicativity). For $\lambda,\tau\in k^{\times}$, we have the isomorphism 
		\[
			\mathcal{KLM}_{\lambda}\circ\mathcal{KLM}_{\tau}\cong \mathcal{KLM}_{\lambda\tau}
		\]
		as endfunctors of $\mathbf{Mod}^{\mathrm{NT}}_{k[F_{n}\rtimes_{\alpha}G]}$.
		\item (Auto-equivalence). For $\lambda\in k^{\times}$,
		$\mathcal{KLM}_{\lambda}$ is an auto-equivalence of the category $\mathbf{Mod}^{\mathrm{NT}}_{k[F_{n}\rtimes_{\alpha}G]}$,
		and $\mathcal{KLM}_{\lambda^{-1}}$ is an inverse functor.
	\end{enumerate}
	We call this functor $\mathcal{KLM}_{\lambda}$ the Katz-Long-Moody functor.
\end{thm}

The resemblance between the Katz middle convolution and the Burau representation 
had already been pointed out by V\"olklein in \cite{Vol}.
Our definition of the Katz-Long-Moody functor is based 
on this V\"olklein's idea 
and  generalizes his insight into the relationship between the Katz 
middle convolution and braid group representations.

As well as the usual Katz middle convolution,
we shall show that the Katz-Long-Moody functor preserves the 
irreducibility.
\begin{thm}
	If a left $k[F_{n}\rtimes_{\alpha}G]$-module $V$ is $F_{n}$-irreducible,
	then $\mathcal{KLM}_{\lambda}(V)$ is $F_{n}$-irreducible for any $\lambda\in k^{\times}$.	
\end{thm}
Here the $F_{n}$-irreducible is a little stronger 
condition which 
implies the usual irreducibility.
Therefore, 
the Katz-Long-Moody 
functor enable us to construct 
many irreducible $k$-local systems on various topological spaces.

We shall further give a homological interpretation of the Katz-Long-Moody functor which 
naturally arises from Katz' original definition of the middle convolution in terms of perverse sheaves.
Let $Q_{n}=\{1,\ldots,n\}$ be the set of points in $\mathbb{C}$.
Also consider the set $Q_{n}^{0,\infty}=\{0,1,\ldots,n,\infty\}$  of 
points in the Riemann sphere $\mathbb{P}^{1}=\mathbb{C}\sqcup\{\infty\}$.
For the configuration space 
$\mathcal{F}_{2}(\mathbb{C}\backslash Q_{n})=\{(t,z)\in (\mathbb{C}\backslash Q_{n})^{2}\mid t\neq z\}$
of $2$ points in $\mathbb{C}\backslash Q_{n}$,
we consider 
projection maps 
$\mathrm{pr}_{t}\colon \mathcal{F}_{2}(\mathbb{C}\backslash Q_{n})\ni (t,z)\mapsto t\in \mathbb{C}\backslash Q_{n}$,
$\mathrm{pr}_{t-z}\colon \mathcal{F}_{2}(\mathbb{C}\backslash Q_{n})\ni (t,z)\mapsto t-z\in \mathbb{C}^{\times}$,
and the embedding 
$\iota\colon \mathbb{P}^{1}\backslash Q_{n}^{0,\infty}\ni t \mapsto (t,0) \in \mathcal{F}_{2}(\mathbb{C}\backslash Q_{n}).$

Then for a left $k[F_{n}\rtimes_{\alpha}G]$-module $V$ and 
the Kummer local system $K_{\lambda}$ on $\mathbb{C}^{\times}$ associated with the 
character $\chi_{\lambda}\colon \pi_{1}(\mathbb{C}^{\times})\cong \mathbb{Z}\rightarrow k^{\times}$
satisfying $\chi_{\lambda}(1)=\lambda$,
we define the new left $k[F_{n}\rtimes_{\alpha}G]$-module in terms of homology groups with 
local system coefficients
by 
\begin{multline*}
\mathcal{MC}_{\lambda}(V)=\mathrm{Im\,}(H_{1}(\mathbb{P}^{1}\backslash Q_{n}^{0,\infty};\iota^{*}(\mathrm{pr}_{t}^{*}(V)\otimes_{k}
\mathrm{pr}_{t-z}^{*}(K_{\lambda})))\\
	\longrightarrow H^{\mathrm{BM}}_{1}(\mathbb{P}^{1}\backslash Q_{n}^{0,\infty};\iota^{*}(\mathrm{pr}_{t}^{*}(V)\otimes_{k}
	\mathrm{pr}_{t-z}^{*}(K_{\lambda})))).
\end{multline*}
Here $H_{1}^{\mathrm{BM}}$ stands for the Borel-Moore homology group or 
equivalently the homology group of locally finite chains.
Then we shall show the following.
\begin{thm}
	For $\lambda\in k^{\times}\backslash\{1\}$, 
	$\mathcal{KLM}_{\lambda}$ and $\mathcal{MC}_{\lambda}$
	are isomorphic as 
	endfunctors of the category of left $k[F_{n}\rtimes_{\alpha}G]$-modules.
\end{thm}

As we shall explain in the final section of this paper,
categories of left $k[F_{n}\rtimes_{\alpha}G]$-modules 
for several $(G,\alpha)$
are isomorphic to categories of $k$-local systems 
on various important topological spaces.
Therefore we can establish a framework of the Katz algorithm
on these local systems.
As a typical example, we can consider $B_{n}$-bundles
associated with simple Weierstrass polynomials
in the sense of 
Hansen.
It is known that some complements of hyperplane arrangements are 
obtained as these $B_{n}$-bundles. For instance, complements of
hyperplane arrangements of fiber-type have the $B_{n}$-bundle structure.
Whereas, in particular for complements of 
hyperplane arrangements,
Haraoka defined an analogue 
of the Katz middle convolution for 
logarithmic connections over 
complements of 
hyperplane arrangements in \cite{Har1}.
The Katz-Long-Moody functor 
will provide the local system counterpart 
of this Haraoka's middle convolution 
through
the Riemann-Hilbert correspondence 
in some cases.

Furthermore, we shall explain that categories of left $k[F_{n}\rtimes_{\alpha}G]$-modules can be 
identified with categories of $k$-local systems on link complements in the solid torus.
Note that link complements in the solid torus are equivalent to 
the complements of mixed links in $S^{3}$, see \cite{Lam}.
There are many known 
methods to obtain 
irreducible representations of knot groups,
see Kronheimer and Mrowka \cite{KR},
Klassen \cite{Kla}, Herald \cite{Her},
and references in Heusener \cite{Heu} for instance,
and then the Katz-Long-Moody functor 
gives a new method to 
obtain irreducible representations
of knot groups.

The Katz-Long-Moody functor will also useful 
to find integral representations 
of solutions of differential equations with accessory parameters.
In general, it may not be able to expect that 
differential equations with accessory parameters
have solutions expressed by integrations 
whose integrands are written by some elementary functions.
As we shall explain in this paper, however,
the Katz-Long-Moody functor 
will give a way to find some special non-rigid 
local systems whose corresponding differential equations may have 
solutions with 
integral representations.

\vspace{2mm}
	\noindent
	\textbf{Acknowledgement}
The authors express their gratitude to 
Arthur Souli\'e and Akihiro Takano 
for the valuable discussion on the Long-Moody functor.
K. H.  
thanks to Koki Ito, Yosuke Ohyama, and Claude Sabbah
for their constructive comments for the author's talk on the Conference 
at Universit\'e de Strasbourg. 
H. N. thanks to Shunya Adachi,
Saiei-Jaeyeong Matsubara-Heo, and Hiroshi Ogawara
for variable comments in the seminar at
Kumamoto University.
Finally the authors 
express their respects to 
Yoshishige Haraoka and Toshio Oshima 
for their inspiring works on 
the middle convolution for KZ-type equations.
\section{Twisted Long-Moody functor}\label{sec:tlm}
Long and Moody introduced 
a functor between categories of 
linear representations of the mixed braid group $B_{1,n}\cong F_{n}\rtimes B_{n}$ and that of $B_{n}$
in the paper \cite{Long}.
We shall consider a 
slight generalization of this functor which we call the twisted Long-Moody functor, 
and explain some properties of this new functor.

\subsection{Artin representation of the braid group $B_{n}$ on the free group $F_{n}$}
The Artin braid group $B_{n}$ associated with a positive integer $n\ge 2$
is the group presented by 
$n-1$ generators $\sigma_{1},\sigma_{2},\ldots,\sigma_{n-1}$
with the braid relations
\begin{align*}
	\sigma_{i}\sigma_{j}&=\sigma_{j}\sigma_{i}\quad\text{for all }i,j=1,2,\ldots,n-1\text{ with }|i-j|>2,\\
	\sigma_{i}\sigma_{i+1}\sigma_{i}&=\sigma_{i+1}\sigma_{i}\sigma_{i+1}\quad \text{for all }i=1,2,\ldots,n-2.
\end{align*}
Now we recall the Artin representation of $B_{n}$
on the free group $F_{n}$ of rank $n$. We take $x_{1},\ldots,x_{n}$ as 
a set of free generators of $F_{n}$.
Let 
$Q_{n}$ be the set of points $a_{i}:=(i,0)$, $i=1,2,\ldots,n$ in $\mathbb{R}^{2}$
and $D$ a closed Euclidean disk in $\mathbb{R}^{2}$ centered at $((n-1)/2,0)$ which contains 
$Q_{n}$ in its interior.
We write the mapping class group of the topological pair $(D,Q_{n})$ as 
$\mathfrak{M}(D,Q_{n})$, which consists of isotopy classes 
of self-homeomorphism of the pair $(D,Q_{n})$ fixing the boundary $\partial D$ pointwise.
For each $i=1,\ldots,n-1$, we define the element $\tau_{i}\in \mathfrak{M}(D,Q_{n})$
called the \emph{half-twist}
as follows. 
We may suppose that the radius of $D$ is sufficiently large
and identify $\mathbb{R}^{2}$ with the complex plane $\mathbb{C}$ as topological spaces.
Let us set $c_{i}:=(a_{i+1}-a_{i})/2$,
and then we define the homeomorphism $\tau_{i}\colon D\rightarrow D$ by 
\[
	\tau_{i}(z-c_{i}):=\begin{cases}
		z-c_{i}&(|z-c_{i}|\ge 1)\\
		\mathrm{exp}(-2\pi i|z-c_{i}|)(z-c_{i})&(\frac{1}{2}\le |z-c_{i}|<1)\\
		-(z-c_{i})&(|z-c_{i}|\le \frac{1}{2})
	\end{cases}	.
\] 
Then it is known that these half-twists $\tau_{1},\tau_{2},\ldots,\tau_{n-1}$
generate the group $\mathfrak{M}(D,Q_{n})$ and the 
correspondence  
\[
	\begin{array}{cccc}
		\eta\colon &B_{n}&\longrightarrow& \mathfrak{M}(D,Q_{n})\\
		&\sigma_{i}&\longmapsto &\tau_{i}
	\end{array}	
\]
extends to the group isomorphism.

The fundamental group $\pi_{1}(D\backslash Q_{n},d)$ with 
a base point $d\in \partial D$ is known to be isomorphic to the free group 
$F_{n}$. To give a precise description of the isomorphism,
we fix the base point $d=((n-1)/2,-R)$ where $R$ is the radius of $D$.  
Then for each $i=1,\ldots,n$ we take the closed path $\gamma_{i}\colon [0,1]\rightarrow D\backslash Q_{n}$
with the base point $d$ as follows.
Let us fix a sufficiently small $\epsilon >0$
so that $D_{<\epsilon}(a_{i}):=\{z\in \mathbb{C}\mid |z-a_{i}|<\epsilon\}$
are inside  $D$ for all $i=1,\ldots,n$
 and have no intersections with the other $D_{<\epsilon}(a_{j})$ for $j\neq i$.
Let $\iota_{i}$ be the intersection point of $D_{<\epsilon}(a_{i})$
and the line segment $\overline{d\,a_{i}}$ connecting $d$ and $a_{i}$.
Then we take the closed path $\gamma_{i}\colon [0,1]\rightarrow D\backslash Q_{n}$ with the 
base point $d$ so that 
it is homotopic to the 
path going from $d$ to $\iota_{i}$ along the segment $\overline{d\,a_{i}}$
and encircling clockwise $a_{i}$ along the circle $\partial D_{<\epsilon}(a_{i})$ 
and finally going back from $\iota_{i}$ to $d$
along the segment $\overline{d\,a_{i}}$.
Then it is known that $\pi_{1}(D\backslash Q_{n},d)$
is freely generated by $\gamma_{1},\ldots,\gamma_{n}$.
Thus the correspondence $\pi_{1}(D\backslash Q_{n},d)\ni \gamma_{i}\mapsto x_{i}\in F_{n}$
extends to the group isomorphism.

The product $\alpha\cdot \beta\colon [0,1]\rightarrow X$ 
of paths $\alpha,\beta\colon [0,1]\rightarrow X$
in a topological space $X$ satisfying $\alpha(1)=\beta(0)$ is 
defined by 
\[
	\alpha\cdot \beta(t):=\begin{cases}
		\alpha(2t)&(t\in [0,1/2])\\
		\beta(2t-1)&(t\in [1/2,1])
	\end{cases}.
\]

The push-forward of a path $\gamma\colon [0,1]\rightarrow X$
by a continuous map $\phi\colon X\rightarrow Y$ of
topological spaces
is denoted by $\phi_{*}(\gamma):=\phi\circ \gamma\colon [0,1]\rightarrow Y$.
Since the base point $d\in \partial D$ is fixed by $\mathfrak{M}(D,Q_{n})$,
the push-forward defines the group homomorphism
\[
	\begin{array}{cccc}
		\theta\colon &\mathfrak{M}(D,Q_{n})&\longrightarrow &\mathrm{Aut}_{\mathbf{Grp}}(\pi_{1}(D\backslash Q_{n},d))\\
		&\phi&\longmapsto &\phi_{*}
	\end{array},
\]
which is known to be injective and characterized by the following relations,
\[
	\theta(\tau_{i})(\gamma_{j})=\begin{cases}
		\gamma_{i+1}& j=i\\
		\gamma_{i+1}^{-1}\gamma_{i}\gamma_{i+1}&j=i+1\\
		\gamma_{j}& j \neq i,\,i+1
	\end{cases}.
\]
Here $\mathrm{Aut}_{\mathbf{Grp}}(G)$ denotes the 
group of automorphisms of a group $G$.
Then, under the identifications $\pi_{1}(D\backslash Q_{n},d)\cong F_{n}$ and $\mathfrak{M}(D,Q_{n})\cong B_{n}$
defined above,
this map $\theta$ induces the injective 
group homomorphism 
\[
	\theta_{\text{Artin}}\colon B_{n}\hookrightarrow \mathrm{Aut}_{\mathbf{Grp}}(F_{n}),	
\]
which is called the \emph{Artin representation} of $B_{n}$ on $F_{n}$.

We extend this Artin representation of $B_{n}$ to some other groups as follows.
\begin{dfn}[Artin representation of a group $G$ on $F_{n}$]\normalfont
	For a group $G$, we call a group homomorphism $\alpha\colon G\rightarrow \mathrm{Aut}_{\mathbf{Grp}}(F_{n})$,
	or the pair $(G,\alpha)$ of them,
	an \emph{Artin representation}, if $\alpha$ factors through the Artin representation 
	$\theta_{\text{Artin}}\colon B_{n}\hookrightarrow \mathrm{Aut}_{\mathbf{Grp}}(F_{n})$ of $B_{n}$. 
	Namely, there exists a group homomorphism $\alpha_{B_{n}}\colon G\rightarrow B_{n}$ which makes
	the diagram
	\[
	\begin{tikzcd}
		G \arrow[r,"\alpha"] \arrow[d,"\alpha_{B_{n}}"']& \mathrm{Aut}_{\mathbf{Grp}}(F_{n})\\
		B_{n}\arrow[ru,"\theta_{\mathrm{Artin}}"']&
	\end{tikzcd}
	\]
	commutative. Here we note that $\alpha_{B_{n}}$ is unique 
	since $\theta_{\mathrm{Artin}}$ is injective.
\end{dfn}

\subsection{Long-Moody functor}
Let us consider the outer semi-direct product 
$F_{n}\rtimes_{\theta_{\mathrm{Artin}}}B_{n}$ 
which we simply denote by $F_{n}\rtimes B_{n}$ if 
there is no risk of confusion.
We recall a procedure for constructing 
representations of $B_{n}$ from those of $F_{n}\rtimes B_{n}$
which is called the 
Long-Moody construction of braid group representations introduced by Long and Moody in \cite{Long}.

\subsubsection{Long-Moody induced representation and augmentation ideal of free group}\label{sec:LMr}
Let us recall the definition of Long-Moody induced representation.
\begin{dfn}[Long-Moody induced representation]\label{dfn:LM}\normalfont
	Let $k$ be a field, $V$ a $k$-vector space, and  
	$\rho \colon F_{n}\rtimes B_{n}\rightarrow \mathrm{Aut}_{k}(V)$
	a group homomorphism. 
	Here $\mathrm{Aut}_{k}(V)$ is the automorphism group of the $k$-vector space $V$.
	Then the \emph{Long-Moody induced representation}
	associated with $\rho$ is the linear representation 
	\[
		\rho^{\mathrm{LM}}\colon B_{n}\rightarrow \mathrm{Aut}_{k}(V^{\oplus n})	
	\]
	of $B_{n}$ which is defined by 
	\begin{align*}
		&\rho^{\mathrm{LM}}(\sigma_{i}):=\\
		&\begin{pmatrix}
			\overbrace{
				\begin{array}{cccc}
					\rho(\sigma_{i})&&&\\
					&\rho(\sigma_{i})&&\\
					&&\ddots&\\
					&&&\rho(\sigma_{i})
				\end{array}
			}^{i-1}&&&\\
			&0&\rho(\sigma_{i}x_{i})&\\
			&\rho(\sigma_{i})&\rho(\sigma_{i})-\rho(\sigma_{i}x_{i+1})&\\
			&&&
			\overbrace{
				\begin{array}{cccc}
					\rho(\sigma_{i})&&&\\
					&\rho(\sigma_{i})&&\\
					&&\ddots&\\
					&&&\rho(\sigma_{i})
				\end{array}
			}^{n-i-1}
		\end{pmatrix}
	\end{align*}
	for $i=1,\ldots,n-1$.
	Here the matrices in the right hand side
	are  block matrices whose entries are in
	 $\mathrm{Aut}_{k}(V)$. 
\end{dfn}
This can be seen as
a generalization of the Burau representation.
In fact, the representation $\rho$ is $1$-dimensional,
the resulting $\rho^{\mathrm{LM}}$
is exactly the Burau representation.
In the paper \cite{Long},
several approaches to define $\rho^{\mathrm{LM}}$
are explained. Among them, we explain 
the following construction 
by using the augmentation ideal of $F_{n}$.

Let $k[G]$ denote the group ring generated by a group $G$
over a field $k$.
A group homomorphism $\rho_{V}\colon G\rightarrow \mathrm{Aut}_{k}(V)$
from $G$ to the automorphism group of a $k$-vector space $V$
induces the left $k[G]$-module structure on $V$ and vice versa.
We frequently do not distinguish between them.
Let $I_{G}$ be 
the \emph{augmentation ideal} of the group ring $k[G]$ which 
is the two-sided ideal defined as 
the kernel of 
the \emph{augmentation map}
\[
	k[G]\ni\sum_{g\in G}r_{g}g\longmapsto \sum_{g\in G}r_{g}\in k.
\]

Let us go back to the setting in Definition \ref{dfn:LM}.
Let us regard $V$ as the left $k[F_{n}\rtimes B_{n}]$-module by the representation $\rho$
and then consider the vector space $I_{F_{n}}\otimes_{k[F_{n}]}V$.
We define a left $k[B_{n}]$-module structure on $I_{F_{n}}\otimes_{k[F_{n}]}V$
by 
\[
	b\cdot (i\otimes v):=\theta_{\mathrm{Artin}}(b)(i)\otimes (b\cdot v)\qquad (b\in B_{n},\,i\in I_{F_{n}},\,v\in V).	
\]
Since $I_{F_{n}}$ is freely generated by ${x_{i}-1,\, i=1,\ldots,n}$
as the right $k[F_{n}]$-module, the $k$-linear map
\begin{equation}\label{eq:indentfy}
	I_{F_{n}}\otimes_{k[F_{n}]}V\ni \sum_{i=1}^{n}(x_{i}-1)\otimes v_{i}
	\longmapsto (v_{1},\ldots,v_{n})\in V^{\oplus n}	
\end{equation}
is an isomorphism 
through which $V^{\oplus n}$ is equipped with the action of $B_{n}$.
Then a direct computation tells us that 
this $B_{n}$-action on $V^{\oplus n}$ 
coincides with
the Long-Moody induced representation
$\rho^{\mathrm{LM}}\colon B_{n}\rightarrow \mathrm{Aut}_{k}(V^{\oplus n})$.

We can slightly generalize this procedure as follows.
Let consider a semi-direct product group $F_{n}\rtimes_{\alpha} G$
associated with an Artin representation $(G,\alpha)$ and take
a left $k[F_{n}\rtimes_{\alpha}G]$-module $V$. Then
$I_{F_{n}}\otimes_{k[F_{n}]}V$ has the natural $k[G]$-module 
structure through the Artin representation $\alpha \colon G\rightarrow \mathrm{Aut}(F_{n})$.

\begin{dfn}[Long-Moody functor]\normalfont
	For a $k$-algebra $R$, we write  the category of left $R$-modules by 
	$\mathbf{Mod}_{R}$. 
	Let $(G,\alpha)$ be an Artin representation on $F_{n}$.
	Then we define a functor 
	$\mathcal{LM}\colon \mathbf{Mod}_{k[F_{n}\rtimes_{\alpha} G]}\rightarrow \mathbf{Mod}_{k[G]}$
	by 
	\begin{enumerate}
		\item $$\mathcal{LM}(V):=I_{F_{n}}\otimes_{k[F_{n}]}V$$ for $V\in \mathscr{O}(\mathbf{Mod}_{k[F_{n}\rtimes_{\alpha} G]})$,
		\item $$\mathcal{LM}(\phi):=\mathrm{id}_{I_{F_{n}}}\otimes \phi
		\in \mathrm{Hom}_{k[G]}(\mathcal{LM}(V),\mathcal{LM}(W))$$
		for $V,W\in \mathscr{O}(\mathbf{Mod}_{k[F_{n}\rtimes_{\alpha} G]})$ and $\phi\in \mathrm{Hom}_{k[F_{n}\rtimes_{\alpha} G]}(V,W)$,
	\end{enumerate}
	and call this functor the \emph{Long-Moody functor}.
	Here $\mathscr{O}(\mathcal{C})$ denotes the class of objects in a category $\mathcal{C}$.
\end{dfn}
Let us note that 
the augmentation ideal $I_{F_{n}}$
is flat as the right $k[F_{n}]$-module since this is free.
Therefore the Long-Moody functor which is defined by the 
tensor product with the flat module $I_{F_{n}}$
is an exact functor.
\subsubsection{Multiplication map}
Let  $(G,\alpha)$ be  an Artin representation on $F_{n}$
and $V$ a left $k[F_{n}\rtimes_{\alpha} G]$-module.
Then we define a $k$-linear map from $\mathcal{LM}(V)$ to $V$
by
\[
	\nabla\colon 	I_{F_{n}}\otimes_{k[F_{n}]}V
	\ni \sum_{i}f_{i}\otimes v_{i}
	\longmapsto \sum_{i}f_{i}\cdot v_{i}\in V,
\]
which we call the \emph{multiplication map}.
This is moreover a $k[G]$-module homomorphism
since we have 
\begin{align*}
	\nabla\left(g\cdot \left(\sum_{i}f_{i}\otimes v_{i}\right)\right)&=\nabla\left(\sum_{i}\alpha(g)(f_{i})\otimes g\cdot v_{i}\right)
	=\sum_{i}\alpha(g)(f_{i})\cdot (g\cdot v_{i})\\
	&=\sum_{i}(\alpha(g)(f_{i})\cdot g)\cdot v_{i}
	=\sum_{i}(g\cdot f_{i})\cdot v_{i}\\
	&=g\cdot \sum_{i}f_{i}\cdot v_{i}=g\cdot \nabla\left(\sum_{i}f_{i}\otimes v_{i}\right)
\end{align*}
for any $g\in G$ and $\sum_{i}f_{i}\otimes v_{i}\in I_{F_{n}}\otimes_{k[F_{n}]}V$.
\subsection{Twisted Long-Moody functor}
The Long-Moody functor is a functor between $\mathbf{Mod}_{k[F_{n}\rtimes_{\alpha} G]}$
and $\mathbf{Mod}_{k[G]}$.
We can extend this to the 
endfunctor of $\mathbf{Mod}_{k[F_{n}\rtimes_{\alpha} G]}$
in the obvious way. Namely,
by the left $k[F_{n}]$-module structure of 
$I_{F_{n}}$ as the two-sided ideal, we can define the left $F_{n}$-module structure
on 
$\mathcal{LM}(V)=I_{F_{n}}\otimes_{k[F_{n}]}V$
which is obviously compatible with the $k[G]$-module structure 
and thus $\mathcal{LM}(V)$ can be seen as a 
left $k[F_{n}\rtimes_{\alpha}G]$-module.

Now we shall introduce a twisted version of this left $k[F_{n}]$-module structure
which is compatible with the $k[G]$-module structure. 
\subsubsection{Twisted Long-Moody functor}
Let us fix an element $\lambda\in k^{\times}$ and 
consider the following functions
\[ 
	\lambda_{\ge i}\colon \{1,\ldots,n\}\ni j\mapsto \begin{cases}
		\lambda&\text{if }j\ge i,\\
		1&\text{if }j<i,
	\end{cases}
\]
for $i=0,1,\ldots,n$.
Then we define a left action of $F_{n}$ on $I_{F_{n}}$ by
\[
	x_{i}\circ_{\lambda}  (x_{j}-1):=\lambda_{\ge i}(j)x_{i}(x_{j}-1)-(\lambda_{\ge (i-1)}(j)-1)(x_{j}-1)\qquad (i,j=1,\ldots,n).
\]
If $\lambda=1$, this action coincides with the usual left action of $F_{n}$
on $I_{F_{n}}$ as the left ideal of $k[F_{n}]$.

\begin{prp}
	The left action $\circ_{\lambda}$ of $F_{n}$ is compatible 
	with the $G$-action through the Artin representation $\alpha$.
	Namely, the diagram 
	\[
		\begin{tikzcd}
			I_{F_{n}}\arrow[r,"\alpha(g)"]\arrow[d,"x\circ_{\lambda}"]
			&I_{F_{n}}\arrow[d,"\alpha(g)(x)\circ_{\lambda}"]\\
			I_{F_{n}}\arrow[r,"\alpha(g)"]
			&I_{F_{n}}
		\end{tikzcd}	
	\]
	is commutative for every  $g\in G$ and  $x\in F_{n}$.
\end{prp}
\begin{proof}
	It suffices to consider the case $(G,\alpha)=(B_{n},\theta_{\mathrm{Artin}})$
	which follows from a direct computation.
	Alternatively, we can use the 
	homological interpretation of $I_{F_{n}}$ 
	given in Proposition \ref{prop:aughom} which
	appears later.
\end{proof}
This proposition enable us to
regard $I_{F_{n}}$ as a left $k[F_{n}\rtimes_{\alpha} G]$-module.
Emphasizing that this module structure depends on the choice of $\lambda\in k^{\times}$, 
we write this  $k[F_{n}\rtimes_{\alpha}G]$-module especially by 
$I_{F_{n},\,\lambda}$.
Then for a left $k[F_{n}\rtimes_{\alpha}G]$-module $V$,
we can regard $I_{F_{n},\lambda}\otimes_{k[F_{n}]}V$
as a left $k[F_{n}\rtimes G]$-module.
\begin{rem}\label{rem:DR}\normalfont
Let $\rho\colon F_{n}\rtimes_{\alpha}G\rightarrow \mathrm{Aut}_{k}(V)$ and 
$\rho_{\lambda}^{\mathrm{LM}}\colon F_{n}\rtimes_{\alpha}G
\rightarrow \mathrm{Aut}_{k}(I_{F_{n},\lambda}\otimes_{k[F_{n}]}V)$
the associated homomorphisms with the left $k[F_{n}\rtimes_{\alpha}G]$-modules
$V$ and 
$I_{F_{n},\lambda}\otimes_{k[F_{n}]}V$.
Then the identification $I_{F_{n},\lambda}\otimes_{k[F_{n}]} V\cong V^{\oplus n}$ under the isomorphism $(\ref{eq:indentfy})$
gives the following matrix representations
\[
	\rho_{\lambda}^{\mathrm{LM}}(x_{i})=
	\begin{pmatrix}
		1&&&&&&\\
		&\ddots&&&&&\\
		&&1&&&&\\
		\lambda (\rho(x_{1})-1)&\cdots&
			\lambda (\rho(x_{i-1})-1)
		&\lambda \rho(x_{i})&
		\rho(x_{i+1})-1&
		\cdots&\rho(x_{n})-1\\
		&&&&1&&\\
		&&&&&\ddots&\\
		&&&&&&1
	\end{pmatrix}
\]
for $i=1,\ldots,n$. 
These matrices in the right hand side are essentially the same 
ones appearing in the paper \cite{DR} by Dettweiler-Reiter.
We can also check that 
\begin{align*}
	\rho_{\lambda}^{\mathrm{LM}}(x_{1}\cdots x_{n})-\lambda\cdot \mathrm{id}_{V^{\oplus n}}=&
	\begin{pmatrix}
		\rho(x_{2}\cdots x_{n})&&&\\
		&\rho(x_{3}\cdots x_{n})&&\\
		&&\ddots&\\
		&&&\mathrm{id}_{V}
	\end{pmatrix}\\
	&\times\begin{pmatrix}
		\lambda(\rho(x_{1})-1)&\lambda(\rho(x_{2})-1)&\cdots&\lambda(\rho(x_{n})-1)\\
		\lambda(\rho(x_{1})-1)&\lambda(\rho(x_{2})-1)&\cdots&\lambda(\rho(x_{n})-1)\\
		\vdots&\vdots&&\vdots\\
		\lambda(\rho(x_{1})-1)&\lambda(\rho(x_{2})-1)&\cdots&\lambda(\rho(x_{n})-1)
	\end{pmatrix}.
\end{align*}
	
\end{rem}

\begin{dfn}[Twisted Long-Moody functor]\normalfont
	For a $\lambda\in k^{\times}$,
	the functor \\$\mathcal{LM}_{\lambda}\colon \mathbf{Mod}_{k[F_{n}\rtimes_{\alpha}G]}\rightarrow \mathbf{Mod}_{k[F_{n}\rtimes_{\alpha} G]}$
	defined by 
	\begin{enumerate}
		\item $$\mathcal{LM}_{\lambda}(V):=I_{F_{n},\,\lambda}\otimes_{k[F_{n}]}V$$ for $V\in \mathscr{O}(\mathbf{Mod}_{k[F_{n}\rtimes_{\alpha} G]})$,
		\item $$\mathcal{LM}_{\lambda}(\phi):=\mathrm{id}_{I_{F_{n},\,\lambda}}\otimes \phi
		\in \mathrm{Hom}_{k[F_{n}\rtimes_{\alpha}G]}(\mathcal{LM}_{\lambda}(V),\mathcal{LM}_{\lambda}(W)),$$
		for $V,W\in \mathscr{O}(\mathbf{Mod}_{k[F_{n}\rtimes_{\alpha} G]})$ and $\phi\in \mathrm{Hom}_{k[F_{n}\rtimes_{\alpha}G]}(V,W),$
	\end{enumerate}
	is called the \emph{twisted Long-Moody functor} with respect to $\lambda$ or \emph{$\lambda$-twisted Long-Moody functor}.
\end{dfn}
Here we note that the twisted Long-Moody functor is an exact functor
as well as the usual Long-Moody functor.
\subsubsection{Augmentation ideal and multiplication map}
As previously we saw, the multiplication map $\nabla$ is a $k[G]$-module homomorphism.
When we consider the twisted action of $F_{n}$ on 
the augmentation ideal, 
the multiplication map
$\nabla$ becomes moreover a $k[F_{n}\rtimes_{\alpha}G]$-module homomorphism as below.
\begin{lmm}\label{lmm:associative}
	Let us take $x\in F_{n}$ and $g,h\in I_{F_{n}}$.
	Then the following twisted associativity holds,
	\[
		(x\circ_{\lambda}g)\circ_{\mu} h=x\circ_{\lambda\mu}(g\circ_{\mu} h).
	\]
\end{lmm}
\begin{proof}
	We may suppose that $x=x_{i}$, $g=x_{j}-1$, $h=x_{k}-1$ for some $i,j,k\in \{1,\ldots,n\}$.
	First we note that
	\begin{align*}
		(x_{j}-1)\circ_{\mu}(x_{k}-1)&=\mu_{\ge j}(k)x_{j}(x_{k}-1)-\mu_{\ge (j-1)}(k)(x_{k}-1)\\
		&=\mu_{\ge j}(k)((x_{j}-1)(x_{k}-1)+(x_{k}-1))-\mu_{\ge (j-1)}(k)(x_{k}-1)\\
		&=\mu_{\ge j}(k)(x_{j}-1)(x_{k}-1)+(\mu_{\ge j}(k)-\mu_{\ge (j-1)}(k))(x_{k}-1)\\
		&=\begin{cases}
			\mu_{\ge j}(k)(x_{j}-1)(x_{k}-1)&\text{if }j\neq k,\\
			\mu(x_{j}-1)x_{j}-(x_{j}-1)&\text{if }j=k.
		\end{cases}
	\end{align*}
	Here the final equation follows from the fact 
	\begin{equation}\label{eq:shift}
		\nu_{\ge i}(j)-\nu_{\ge (i-1)}(j)=\begin{cases}
			0&\text{if }j\neq i,\\
			\nu-1&\text{if }j=i,
		\end{cases}\quad \text{for }\nu\in k^{\times}.
	\end{equation}
	Thus in any case, we can write $(x_{j}-1)\circ_{\mu}(x_{k}-1)=(x_{j}-1)r$ with some $r\in k[F_{n}]$.
	Therefore
	we have 
	\begin{equation}\label{eq:lhs}
	\begin{aligned}
		x_{i}\circ_{\lambda\mu}((x_{j}-1)\circ_{\mu} (x_{k}-1))&=x_{i}\circ_{\lambda\mu}(x_{j}-1)r\\
		&=(\lambda\mu)_{\ge i}(j)x_{i}(x_{j}-1)r-((\lambda\mu)_{\ge (i-1)}(j)-1)(x_{j}-1)r\\
		&=(\lambda\mu)_{\ge i}(j)x_{i}((x_{j}-1)\circ_{\mu} (x_{k}-1))\\
		&\quad\quad\quad-((\lambda\mu)_{\ge (i-1)}(j)-1)((x_{j}-1)\circ_{\mu} (x_{k}-1)).
	\end{aligned}
	\end{equation}
	On the other hand, we have
	\begin{equation}\label{eq:rhs}
	\begin{aligned}
		(x_{i}\circ_{\lambda}&(x_{j}-1))\circ_{\mu} (x_{k}-1)\\
		&=(\lambda_{\ge i}(j)x_{i}(x_{j}-1)-(\lambda_{\ge (i-1)}(j)-1)(x_{j}-1))\circ_{\mu} (x_{k}-1)\\
		&=\lambda_{\ge i}(j)x_{i}\circ_{\mu}((x_{j}-1)\circ_{\mu}(x_{k}-1))
		-(\lambda_{\ge (i-1)}(j)-1)(x_{j}-1)\circ_{\mu} (x_{k}-1)\\
		&=\lambda_{\ge i}(j)\mu_{\ge i}(j)x_{i}((x_{j}-1)\circ_{\mu} (x_{k}-1))\\
		&\qquad -(\lambda_{\ge i}(j)\mu_{\ge (i-1)}(j)-\lambda_{i}(j)+\lambda_{\ge (i-1)}(j)-1)
		((x_{j}-1)\circ_{\mu} (x_{k}-1)).
	\end{aligned}
	\end{equation}
	Now let us note that 
	\begin{align}
		&\lambda_{\ge i}(j)\mu_{\ge i}(j)=(\lambda\mu)_{\ge i}(j) \label{eq:int1}\\
		&\lambda_{\ge i}(j)\mu_{\ge (i-1)}(j)-\lambda_{i}(j)+\lambda_{\ge (i-1)}(j)-1 \label{eq:int2}
		=(\lambda\mu)_{\ge (i-1)}(j)-1.
	\end{align}
	Here the second equation follows from 
	equation $(\ref{eq:shift})$.

	Comparing the equations from $(\ref{eq:lhs})$ to $(\ref{eq:int2})$, we obtain the desired equation
	\[
		x_{i}\circ_{\lambda\mu}((x_{j}-1)\circ_{\mu} (x_{k}-1))=(x_{i}\circ_{\lambda}(x_{j}-1))\circ_{\mu} (x_{k}-1).
	\]
\end{proof}
This lemma shows that the multiplication map
$\nabla$ defines the following left $k[F_{n}\rtimes_{\alpha} G]$-modules homomorphism.
\begin{prp}\label{prp:lmtrans}
	For $\lambda,\mu\in k^{\times}$, the multiplication map 
	$$\nabla\colon I_{F_{n},\lambda}\otimes_{k[F_{n}]}I_{F_{n},\mu}
	\longrightarrow I_{F_{n},\lambda\mu}$$
	is a left $k[F_{n}\rtimes_{\alpha} G]$-module homomorphism.
\end{prp}
\begin{proof}
	As previously we saw, $\nabla$ is a $k[G]$-module homomorphism.
	Thus we need to check that $\nabla$ is compatible with 
	left actions of $k[F_{n}]$ on $I_{F_{n},\lambda}\otimes_{k[F_{n}]}I_{F_{n},\mu}$
	and $I_{F_{n},\lambda\mu}$,
	which follows from Lemma \ref{lmm:associative}.
\end{proof}

\begin{rem}\normalfont
It follows from Proposition \ref{prp:lmtrans}
that 
the multiplication map  defines the 
natural transformation between the functors 
$$\nabla\colon \mathcal{LM}_{\lambda}\circ \mathcal{LM}_{\mu}
\longrightarrow 
\mathcal{LM}_{\lambda\mu}.$$
\end{rem}

We present some properties of the multiplication map which will be used later.
\begin{prp}\label{prop:multiplication}
	For a left $k[F_{n}\rtimes_{\alpha}G]$-module $V$,
	let $\rho_{V}\colon F_{n}\rightarrow \mathrm{Aut}_{k}(V)$ be the associated
	group homomorphism.
	\begin{enumerate}
		\item The image of the multiplication map $\nabla\colon I_{F_{n}}
		\otimes_{k[F_{n}]}V\rightarrow V$ is 
		\[
			\mathrm{Im\,}(\rho_{V}(x_{1})-1)+\cdots +\mathrm{Im\,}(\rho_{V}(x_{1})-1).	
		\]
		Also the kernel is 
		\[
			\left\{(v_{1},\ldots,v_{n})\in V^{\oplus n}\,\middle|\, \sum_{i=1}^{n}(x_{i}-1)v_{i}=0\right\},
		\]
		under the identification $I_{F_{n}}
		\otimes_{k[F_{n}]}V\cong V^{\oplus n}$ by the isomorphism $(\ref{eq:indentfy})$.
		\item Let us take $\lambda,\mu\in k^{\times}$.
		Then the kernel of the multiplication map 
		$\nabla\colon \mathcal{LM}_{\lambda}\circ \mathcal{LM}_{\mu}(V)
		\rightarrow \mathcal{LM}_{\lambda\mu}(V)$
		is 
		\[
			\bigoplus_{i=1}^{n}\left((x_{i}-1)\otimes \mathrm{Ker\,}(\rho_{V}^{\mu}(x_{i})-1)\right).
		\]
	\end{enumerate}
	Here $\rho_{V}^{\mu}\colon F_{n}\rightarrow \mathrm{Aut}_{k}(\mathcal{LM}_{\mu}(V))$
	is the induced group homomorphism by the left $F_{n}$-action on $\mathcal{LM}_{\tau}(V)$,
	and 
	for $x\in I_{F_{n}}$ and a subspace $W\subset V$, we set 
	$x\otimes W:=\{x\otimes w\mid w\in W\}$.
\end{prp}
\begin{proof}
	The first assertion is obvious from the definition of $\nabla$.
	The second one follows from the fact 
	\[
		\mathrm{Im\,}(\rho_{V}^{\mu}(x_{i})-1)\subset  (x_{i}-1)\otimes V
	\]
	which can be easily seen by the matrix representation of $\rho_{V}^{\mu}(x_{i})$ appearing in Remark \ref{rem:DR}.
\end{proof}

\section{Euler transform and twisted Long-Moody functor}\label{sec:ET}
The Euler transform is a classical tool 
in analysis, which is 
defined by the following integral transformation,
\[
	\phi(z)\longmapsto \int_{C}\phi(t)(t-z)^{\lambda}\,dt	
\]
for $\lambda \in \mathbb{C}$ and a suitable contour $C$.
One of the most famous integration of this form may be the integral representation of the 
Gauss hypergeometric function $F(\alpha,\beta,\gamma;z)$, 
\[
	F(\alpha,\beta,\gamma;z)=\frac{\Gamma(\gamma)}{\Gamma(\alpha)\Gamma(\gamma-\alpha)}\int_{1}^{\infty}
	t^{\beta-\gamma}(t-1)^{\gamma-\alpha-1}(t-z)^{-\beta}\,dt,
\]	
which is valid for some suitable complex numbers $\alpha,\beta,\gamma$ and 
choice of a branch of the multi-valued function $t^{\beta-\gamma}(t-1)^{\gamma-\alpha-1}(t-z)^{-\beta}$. 

In this section, we provide a homological interpretation of this Euler transform
and explain that the twisted Long-Moody functor can be obtained from 
this homological Euler transform.

\subsection{Homological Euler transform}
In this paper, 
locally constant sheaves of $k$-vector spaces on a topological space $X$
are just called $k$-local systems on $X$.
If 
$X$ is a connected topological manifold, 
it is well-known that 
the stalk $\mathcal{L}_{d}$ at $d\in X$ 
of 
a $k$-local system $\mathcal{L}$ on $X$
has the structure of left $k[\pi_{1}(X,d)]$-module,
and conversely, 
any left $k[\pi_{1}(X,d)]$-module
$L_{d}$ can be seen as the stalk
at $d$ of some $k$-local system. 
Therefore 
we frequently  do not 
distinguish between them.

\begin{dfn}[Kummer local system]\normalfont
		Let $\omega$ be the closed path in $\mathbb{C}^{\times}$
		encircling the origin counterclockwise along the 
		unit circle with the base point $1$. 
	For a 1-dimensional representation 
	$\rho \colon \pi_{1}(\mathbb{C}^{\times},1)\rightarrow \mathrm{GL}_{1}(k)$,
	the corresponding local system over $\mathbb{C}^{\times}$ is called \emph{Kummer local system} which is denoted by
	$K_{\lambda}$ with the index $\lambda:=\rho(\omega)\in k^{\times}$.
\end{dfn}

Let us consider $\mathcal{F}_{2}(\mathbb{C}\backslash Q_{n})=\{(t,z)\in \mathbb{C}^{2}
\mid t\neq z\text{ and }t,z\notin Q_{n}\}$, the configuration space of 
$2$ points in $\mathbb{C}\backslash Q_{n}$,
and $\pr_{t}$ and $\pr_{z}$ respectively denote 
the first and second projections of $\mathcal{F}_{2}(\mathbb{C}\backslash Q_{n})$.
Also we consider the map
\[
	\pr_{t-z}\colon \mathcal{F}_{2}(\mathbb{C}\backslash Q_{n})\ni (t,z)\longmapsto 
	t-z\in \mathbb{C}^{\times}.	
\]

Let $L$ be a $k$-local system on $\mathbb{C}\backslash Q_{n}$ and $K_{\lambda}$
a Kummer local system, 
and then we consider the $k$-local system on $\mathcal{F}_{2}(\mathbb{C}\backslash Q_{n})$ defined
by 
\[
	C_{\lambda}(L):=\pr_{t}^{*}(L)\otimes_{k}\pr_{t-z}^{*}(K_{\lambda}).
\]
Let us set $Q_{n}(z):=Q_{n}\cup \{z\}$ for
$z\in \mathbb{C}\backslash Q_{n}$.
Then 
we can naturally identify the fiber $\pr_{z}^{-1}(z)$ 
with $\mathbb{C}\backslash Q_{n}(z)$ and obtain
the embedding map $\iota_{z}\colon \mathbb{C}\backslash Q_{n}(z)
\hookrightarrow \mathcal{F}_{2}(\mathbb{C}\backslash Q_{n})$ through this identification.
\begin{dfn}[Homological Euler transform of a local system]\normalfont
For every open disk $D_{<r}(z)\subset \mathbb{C}\backslash Q_{n}$, 
we define the embedding 
$\iota_{r,z}\colon (\mathbb{C}\backslash Q_{n})\backslash D_{<r}(z)
\hookrightarrow \mathcal{F}_{2}(\mathbb{C}\backslash Q_{n})$
as  the composite 
of $(\mathbb{C}\backslash Q_{n})\backslash D_{<r}(z)\hookrightarrow 
\mathbb{C}\backslash Q_{n}(z)$ and $\iota_{z}\colon \mathbb{C}\backslash Q_{n}(z)
\hookrightarrow \mathcal{F}_{2}(\mathbb{C}\backslash Q_{n})$.
Then
we associate this open set with the singular homology group with local coefficient
\[
	H_{1}((\mathbb{C}\backslash Q_{n})\backslash D_{<r}(z);\iota_{r,z}^{*}(C_{\lambda}(L)))	
\]
to every open disk $D_{<r}(z)\subset \mathbb{C}\backslash Q_{n}$.
Since open disks $D_{<r}(z)$ constitute a base of open sets 
in $\mathbb{C}\backslash Q_{n}(z)$,
we can sheafify this correspondence and   
then $\mathcal{E}_{\lambda}(L)$ denotes the resulting sheaf 
on $\mathbb{C}\backslash Q_{n}$ which is a $k$-local system 
whose stalks are   
isomorphic to $H_{1}(\mathbb{C}\backslash Q_{n}(z);\iota_{z}^{*}(C_{\lambda}(L)))$
 at $z\in \mathbb{C}\backslash Q_{n}$.
We call this $k$-local system $\mathcal{E}_{\lambda}(L)$
the \emph{homological Euler transform} of the local system $L$
with respect to the Kummer local system $K_{\lambda}$.
\end{dfn}

There are many equivalent 
definitions for homology groups of local coefficients
and we adopt the following definition among them in this paper.
For a connected topological 
manifold $X$ and a $k$-local system $L$,
we regard $L$ as a   
left $k[\pi_{1}(X,x)]$-module 
by taking the stalk at $x$.
Then the $k$-vector space of singular $n$-chains with coefficients in $L$
is 
\[
	C_{n}(X;L):=C_{n}(\widetilde{X};k)\otimes_{\pi_{1}(X,x)} L
\]
where $C_{n}(\widetilde{X};k)$ is the $k$-vector space of
usual singular $n$-chains in the universal covering 
$\pi_{X}\colon \widetilde{X}\rightarrow X$ of $X$.
Then $H_{*}(X;L)$
 are defined as homology groups 
 of the chain complex $C_{*}(X;L)$.

In the definition of $C_{n}(X;L)$,
 $C_{n}(\widetilde{X};k)$ is regarded as a right $k[\pi_{1}(X,x)]$-module 
as follows.
Let us consider the natural isomorphism
\[
	\delta\colon \pi_{1}(X,x)
	\longrightarrow 	
	\mathrm{Deck}(\pi_{X})
\]
where 
$\mathrm{Deck}(\pi_{X})$ is 
the group of deck transformations of  
$\pi_{X}
\colon \widetilde{X}
\rightarrow X$.
Since the definition of $\delta$ depends on the 
choice of $\bar{x}\in \pi_{X}^{-1}(x)$,
we give the precise definition here for the accuracy.
We associate $\delta(\gamma)\in \mathrm{Deck}(\pi_{X})$
with each $\gamma\in \pi_{1}(X,x)$
as below.
For each $c\in \widetilde{X}$ we consider 
the path $\sigma_c\colon [0,1]\rightarrow \widetilde{X}$ 
from $\bar{x}$ to $c$
which is unique up to homotopy.
Then $\delta(\gamma)(c)\in \widetilde{X}$
is the end point of 
the lift of 
$\gamma\cdot (\pi_{X})_{*}(\sigma_{c})$ to $\widetilde{X}$
starting at $\bar{x}$.

Now we define the right action of $\pi_{1}(X,x)$ on 
the set of singular $n$-simplices $S_{n}(\widetilde{X})$
by 
\[
	S_{n}(\widetilde{X})\times \pi_{1}(X,x)\ni (\sigma,\gamma)\longmapsto 
	\delta(\gamma^{-1})_{*}(\sigma) \in S_{n}(\widetilde{X}) 	
\]
which extends to the right action of $\pi_{1}(X,x)$ on $C_{n}(\widetilde{X};k)$.
Here for a continuous map $\phi\colon X_{1}\rightarrow X_{2}$
of topological spaces, 
we write the induced map of singular simplices 
by $\phi_{*}\colon S_{n}(X_{1})\ni \sigma\mapsto \phi\circ \sigma\in S_{n}(X_{2})$
which linearly extends 
to the $k$-linear map $\phi_{*}\colon C_{n}(X_{1};k)\rightarrow C_{n}(X_{2};k)$
of chains.

\subsection{Euler transform functor of $\mathbf{Mod}_{k[F_{n}\rtimes_{\alpha}G]}$}\label{sec:braidhomol}
The homological Euler transform $\mathcal{E}_{\lambda}$
was introduced in the previous section 
 as an endfunctor of the category 
of $k$-local systems on $\mathbb{C}\backslash Q_{n}$.
We explain that this functor naturally extends 
to the endfunctor of $\mathbf{Mod}_{k[F_{n}\rtimes_{\alpha}G]}$.

Let $(G,\alpha)$ be an Artin representation on $F_{n}$ and 
$V$ a left $k[F_{n}\rtimes_{\alpha}G]$-module.
Then under the identifications 
$\pi_{1}(\mathbb{C}\backslash Q_{n},d)\cong \pi_{1}(D\backslash Q_{n},d)\cong F_{n}$,
$V$ is seen as a left  $k[\pi_{1}(\mathbb{C}\backslash Q_{n},d)]$-module.
Thus we obtain the local system $\mathcal{E}_{\lambda}(V)$ on $\mathbb{C}\backslash Q_{n}$,
and as the stalk at $z\in \mathbb{C}\backslash Q_{n}$
the homology group
\[
	\mathcal{E}_{\lambda}(V)_{z}=H_{1}(\mathbb{C}\backslash Q_{n}(z);\iota_{z}^{*}(C_{\lambda}(V)))
\]
appears. We shall see that this homology group is naturally  equipped with an action of $F_{n}\rtimes_{\alpha}G$.

To provide an accurate explanation, we take $z\in \mathbb{C}\backslash Q_{n}$
as the origin $0$ and write it by $a_{0}:=z=0$.
We take the closed disk $D$ as large as it contains $Q_{n}(a_{0})=\{a_{0}=0,
a_{1}=1,\ldots,a_{n}=n\}$ in its interior.
We write the half-twist 
interchanging $a_{i}$ and $a_{i+1}$
by $\tau_{i}$ for each $i=0,1,\ldots,n-1$
as in Section \ref{sec:tlm}.
Then $\mathfrak{M}(D,Q_{n}(a_{0}))$
is generated by $\tau_{0},\tau_{1},\ldots,\tau_{n-1}$
and contains $\mathfrak{M}(D,Q_{n})$
as the subgroup generated by $\tau_{1},\ldots,\tau_{n-1}$. 
Through the isomorphism $\mathfrak{M}(D,Q_{n}(a_{0}))
\cong B_{n+1}$, the generators of $B_{n+1}$
corresponding to the half-twists $\tau_{0},\tau_{1},\ldots,\tau_{n-1}$
are denoted
by $\sigma_{0},\sigma_{1},\ldots,\sigma_{n-1}$ respectively.
Since elements in $\mathfrak{M}(D,Q_{n}(a_{0}))$
are identity maps on $\partial D$, we can regard them
as self-homeomorphisms of $\mathbb{C}\backslash Q_{n}(a_{0})$ by
extend them as identity outside $D$.
As a set of generators of $\pi_{1}(\mathbb{C}\backslash Q_{n}(a_{0}),d)$,
we take closed path 
$\gamma_{i}$ encircling $a_{i}$, $i=0,1,\ldots,n$
as well as in Section \ref{sec:tlm}.

\subsubsection{$F_{n}\rtimes_{\alpha}G$-actions on $\pi_{1}(\mathbb{C}\backslash Q_{n}(a_{0}),d)$ and 
$C_{n}(\widetilde{\mathbb{C}\backslash Q_{n}(a_{0})};k)$}\label{sec:pi1action}
Let us consider the subgroup of $B_{n+1}$
generated by $\sigma_{0}^{2},\sigma_{1},\ldots,\sigma_{n-1}$
and denote it by $B_{1,n}$.
Then we can see that $x_{1}:=\sigma_{0}^{2},\,x_{2}:=\sigma_{1}x_{1}\sigma_{1}^{-1},\,
\ldots\,,x_{n}:=\sigma_{n-1}x_{n-1}\sigma_{n-1}^{-1}$
generate the free group $F_{n}$,
and $B_{1,n}$ decomposes as the inner semidirect product  
$B_{1,n}=F_{n}\rtimes B_{n}$
where $B_{n}$ acts on $F_{n}$ through the Artin representation. 
Then it follows that 
$F_{n}\rtimes_{\alpha}G$ acts on $\mathbb{C}\backslash Q_{n}(a_{0})$,
i.e,
we obtain the group homomorphism 
\[
	\eta\colon 	F_{n}\rtimes_{\alpha}G\longrightarrow \mathrm{Aut}_{\mathbf{Top}}(\mathbb{C}\backslash Q_{n}(a_{0}))
\]
which is the composition of the following maps; the first one is 
the map 
\[
	F_{n}\rtimes_{\alpha}G\ni (x,g)\mapsto (x,\alpha(g))\in F_{n}\rtimes B_{n},
\]
the second one is 
the inclusion $F_{n}\rtimes B_{n}\cong B_{1,n}\subset B_{n+1}
\cong \mathfrak{M}(D,Q_{n}(a_{0}))$, and 
the final one is 
the inclusion 
$\mathfrak{M}(D,Q_{n}(a_{0}))\hookrightarrow  
 \mathrm{Aut}_{\mathbf{Top}}(\mathbb{C}\backslash Q_{n}(a_{0}))$. 
Here $\mathrm{Aut}_{\mathbf{Top}}(X)$
denotes the group of self-homeomorphisms of a topological space $X$.
Then this $\eta$ allows us to define
the following left action of $F_{n}\rtimes_{\alpha}G$ on $\pi_{1}(\mathbb{C}\backslash Q_{n}(a_{0}),d)$,
\[
	(F_{n}\rtimes_{\alpha}G)\times \pi_{1}(\mathbb{C}\backslash Q_{n}(a_{0}),d)
	\ni(h,\gamma)\longmapsto \eta(h)_{*}(\gamma)\in \pi_{1}(\mathbb{C}\backslash Q_{n}(a_{0}),d).
\]

Also we define an action of $F_{n}\rtimes_{\alpha}G$ on  
$C_{n}(\widetilde{\mathbb{C}\backslash Q_{n}(a_{0})};k)$ as follows.
Fix an element $\bar{d}\in \pi_{\mathbb{C}\backslash Q_{n}(a_{0})}^{-1}(d)$.
Then for each $h\in F_{n}\rtimes_{\alpha}G$, let  
$\overline{\eta(h)}\colon \widetilde{\mathbb{C}\backslash Q_{n}(a_{0})}
\rightarrow \widetilde{\mathbb{C}\backslash Q_{n}(a_{0})}$
be the unique lift of $\pi_{\mathbb{C}\backslash Q_{n}(a_{0})}\circ \eta(h)
\colon \widetilde{\mathbb{C}\backslash Q_{n}(a_{0})}
\rightarrow \mathbb{C}\backslash Q_{n}(a_{0})$
satisfying $\overline{\eta(h)}(\bar{d})=\bar{d}$.
\[
	\begin{tikzcd}
		\widetilde{\mathbb{C}\backslash Q_{n}(a_{0})}\arrow[r,"\overline{\eta(h)}"]
		\arrow[d,"\pi_{\mathbb{C}\backslash Q_{n}(a_{0})}"]&\widetilde{\mathbb{C}\backslash Q_{n}(a_{0})}
		\arrow[d,"\pi_{\mathbb{C}\backslash Q_{n}(a_{0})}"]\\
		\mathbb{C}\backslash Q_{n}(a_{0})\arrow[r,"\eta(h)"]&
		\mathbb{C}\backslash Q_{n}(a_{0})
	\end{tikzcd}
\]
The uniqueness of the lift induces  
$(\overline{\eta(h)})^{-1}=\overline{\eta(h^{-1})}$ and   
$\overline{\eta(h_{1})}\circ \overline{\eta(h_{1})}=\overline{\eta(h_{1}h_{2})}$
for $h,h_{1},h_{2}\in F_{n}\rtimes_{\alpha}G$. 
Thus 
$
\bar{\eta}\colon F_{n}\rtimes_{\alpha}G\rightarrow \mathrm{Aut}_{\mathbf{Top}}(\widetilde{\mathbb{C}\backslash Q_{n}(a_{0})})	
$
becomes a group homomorphism and defines
the left action of $F_{n}\rtimes_{\alpha}G$
on $S_{n}(\widetilde{\mathbb{C}\backslash Q_{n}(a_{0})})$ by
\[
	(F_{n}\rtimes_{\alpha}G)\times S_{n}(\widetilde{\mathbb{C}\backslash Q_{n}(a_{0})})
	\ni (h,\sigma)\longmapsto \bar{\eta}(h)_{*}(\sigma) \in S_{n}(\widetilde{\mathbb{C}\backslash Q_{n}(a_{0})}),
\]
which
extends linearly to the action on $C_{n}(\widetilde{\mathbb{C}\backslash Q_{n}(a_{0})};k)$.
  
Now we recall that 
this $\bar{\eta}$ is compatible with the isomorphism 
$\delta\colon \pi_{1}(\mathbb{C}\backslash Q_{n}(a_{0}),d)
\rightarrow 	
\mathrm{Deck}(\pi_{\mathbb{C}\backslash Q_{n}(a_{0})})$,
i.e., we have the following.
\begin{lmm}\label{lem:compa}
The diagram
\[
	\begin{tikzcd}
		\widetilde{\mathbb{C}\backslash Q_{n}(a_{0})}
		\arrow[r,"\bar{\eta}(h)"]
		\arrow[d,"\delta(\gamma)"]
		&\widetilde{\mathbb{C}\backslash Q_{n}(a_{0})}
		\arrow[d,"\delta( \eta(h))_{*}(\gamma))"]\\
		\widetilde{\mathbb{C}\backslash Q_{n}(a_{0})}\arrow[r,"\bar{\eta}(h)"]
		&\widetilde{\mathbb{C}\backslash Q_{n}(a_{0})}
	\end{tikzcd}	
\]
is commutative
for each $h\in F_{n}\rtimes_{\alpha}G$ and $\gamma \in \pi_{1}(
	\mathbb{C}\backslash Q_{n}(a_{0}),d)$.
\end{lmm}
\begin{proof}
	This is well-known.
\end{proof}

\subsubsection{$F_{n}\rtimes_{\alpha}G$-action on $\iota_{a_{0}}^{*}(C_{\lambda}(V))$}
Let us define a group homomorphism 
$\tilde{\eta}\colon F_{n}\rtimes_{\alpha}G\rightarrow \mathrm{Aut}_{k}(V\otimes K_{\lambda})$
by
\begin{align*}
	\tilde{\eta}(x)(v\otimes \kappa)&:=v\otimes \kappa&\text{for }x\in F_{n},\text{ and }v\otimes \kappa\in V\otimes_{k}K_{\lambda},\\
	\tilde{\eta}(g)(v\otimes \kappa)&:=gv\otimes \kappa&\text{for }g\in G,\text{ and }v\otimes \kappa\in V\otimes_{k}K_{\lambda}.
\end{align*}
Then we check that $\tilde{\eta}$ is compatible with 
the $\pi_{1}(\mathbb{C}\backslash Q_{n}(a_{0}),d)$-action
on $\iota_{a_{0}}^{*}(C_{\lambda}(V))=V\otimes_{k}K_{\lambda}$.
Namely, 
the diagram 
\[
	\begin{tikzcd}
		\iota_{a_{0}}^{*}(C_{\lambda}(V))\arrow[r,"\tilde{\eta}(h)"]\arrow[d,"\gamma\cdot"]
		&\iota_{a_{0}}^{*}(C_{\lambda}(V))\arrow[d,"\eta(h)_{*}(\gamma)\cdot"]\\
		\iota_{a_{0}}^{*}(C_{\lambda}(V))\arrow[r,"\tilde{\eta}(h)"]
		&\iota_{a_{0}}^{*}(C_{\lambda}(V))
	\end{tikzcd}	
\]
is commutative
for every $h\in F_{n}\rtimes_{\alpha}G$ and $\gamma\in \pi_{1}(\mathbb{C}\backslash Q_{n}(a_{0}),d)$.

Recall that $\iota_{a_{0}}^{*}(C_{\lambda}(V))$
is the $k[\pi_{1}(\mathbb{C}\backslash Q_{n}(a_{0}),d)]$-module
which equals to 
$V\otimes_{k}K_{\lambda}$
as a $k$-vector space, and the
$\pi_{1}(\mathbb{C}\backslash Q_{n}(a_{0}),d)$-action 
on $\iota_{a_{0}}^{*}(C_{\lambda}(V))$
is given by 
\[
	\gamma\cdot(v\otimes \kappa)=
	(\mathrm{pr}_{t}\circ\iota_{a_{0}})_{*}(\gamma)\cdot v\otimes 
	(\mathrm{pr}_{t-z}\circ\iota_{a_{0}})_{*}(\gamma)\cdot \kappa	
\]
for $\gamma\in \pi_{1}(\mathbb{C}\backslash Q_{n}(a_{0},d))$, $v\in V$
and $\kappa\in K_{\lambda}$.

First we look at  the homomorphism 
$(\mathrm{pr}_{t}\circ\iota_{a_{0}})_{*}\colon 
\pi_{1}(\mathbb{C}\backslash Q_{n}(a_{0}),d)
\rightarrow \pi_{1}(\mathbb{C}\backslash Q_{n},d)$.
Since $\mathrm{pr}_{t}\circ\iota_{a_{0}}$
is the  
inclusion map $\mathbb{C}\backslash Q_{n}(a_{0})
\hookrightarrow \mathbb{C}\backslash Q_{n}$
and $\gamma_{0}$ in $\mathbb{C}\backslash Q_{n}(a_{0})$
is homotopic to the constant path in $\mathbb{C}\backslash Q_{n}$,
we can see that the group homomorphism 
$(\mathrm{pr}_{t}\circ\iota_{a_{0}})_{*}$
coincides with the projection map 
\[
	\begin{array}{cccc}
	\mathrm{pr}_{0}\colon& F_{n+1}=\langle x_{0},x_{1},\ldots,x_{n} \rangle&\longrightarrow &F_{n}=\langle x_{1},\ldots,x_{n}\rangle\\
	&x_{i}&\longmapsto &\begin{cases} 1&\text{if }i=0\\
		x_{i}&\text{otherwise}
	\end{cases}	
	\end{array},
\]
for which the inclusion map
$\mathrm{inc}_{0}\colon F_{n}=\langle x_{1},\ldots,x_{n}\rangle \ni x_{i}\mapsto x_{i}\in F_{n+1}
=\langle x_{0},x_{1},\ldots,x_{n}\rangle$
is a retraction homomorphism,
i.e., $\mathrm{pr}_{0}\circ \mathrm{inc}_{0}=\mathrm{id}_{F_{n}}$.

Let us regard elements in $\mathfrak{M}(D,Q_{n}(a_{0}))$
as self-homeomorphisms of the complex plane $\mathbb{C}$.
Since the subgroup $B_{1,n}$ of 
$\mathfrak{M}(D,Q_{n}(a_{0}))$
preserves $Q_{n}$, it acts not only on $D\backslash Q_{n}(a_{0})$ but also on
 $D\backslash Q_{n}$.
Thus we have the commutative diagram 
\[
	\begin{tikzcd}
		\mathbb{C}\backslash Q_{n}(a_{0})
		\arrow[r,"\mathrm{pr}_{t}\circ\iota_{a_{0}}"]
		\arrow[d,"\sigma"]&
		\mathbb{C}\backslash Q_{n}
		\arrow[d,"\sigma"]\\
		\mathbb{C}\backslash Q_{n}(a_{0})
		\arrow[r,"\mathrm{pr}_{t}\circ\iota_{a_{0}}"]
		&
		\mathbb{C}\backslash Q_{n}
	\end{tikzcd}	
\]
for each $\sigma \in B_{1,n}$.
 \begin{lmm}\label{lem:trivact1}
	The subgroup $B_{n}$ of $B_{1,n}$
	acts on $\pi_{1}(\mathbb{C}\backslash Q_{n},d)$
	by the Artin representation.
	The subgroup $F_{n}$ of $B_{1,n}$
	generated by $x_{1}=\sigma_{0}^{2},\,x_{2}=\sigma_{1}x_{1}\sigma_{1}^{-1},\,
	\ldots\,,x_{n}=\sigma_{n-1}x_{n-1}\sigma_{n-1}^{-1}$
	acts trivially on $\pi_{1}(\mathbb{C}\backslash Q_{n},d)$.
\end{lmm}
\begin{proof} 
	The above commutative diagram induces
	the commutative diagram 
	 \[
	\begin{tikzcd}
		\pi_{1}(\mathbb{C}\backslash Q_{n}(a_{0}),d)
		\arrow[r,"(\mathrm{pr}_{t}\circ\iota_{a_{0}})_{*}"]
		\arrow[d,"\sigma^{*}"]&
		\pi_{1}(\mathbb{C}\backslash Q_{n},d)
		\arrow[d,"\sigma^{*}"]\\
		\pi_{1}(\mathbb{C}\backslash Q_{n}(a_{0}),d)
		\arrow[r,"(\mathrm{pr}_{t}\circ\iota_{a_{0}})_{*}"]
		&
		\pi_{1}(\mathbb{C}\backslash Q_{n},d)
	\end{tikzcd}.
	\]
	Then this diagram tells us that 
	the map
	$\sigma_{*}\colon \pi_{1}(\mathbb{C}\backslash Q_{n},d)\rightarrow \pi_{1}(\mathbb{C}\backslash Q_{n},d)$
	decomposes as
	\[
		\pi_{1}(\mathbb{C}\backslash Q_{n},d)
		\xrightarrow{\mathrm{inc}_{0}}
		\pi_{1}(\mathbb{C}\backslash Q_{n}(a_{0}))
		\xrightarrow{\sigma_{*}}
		\pi_{1}(\mathbb{C}\backslash Q_{n}(a_{0}))
		\xrightarrow{(\mathrm{pr}_{t}\circ\iota_{a_{0}})_{*}}
		\pi_{1}(\mathbb{C}\backslash Q_{n},d)
	\]
	for every $\sigma\in B_{1,n}$.
	This immediately shows the first assertion. 

	On the other hand, since we have
	\[
		(\sigma_{0}^{2})_{*}(\gamma_{i})=
		\begin{cases}
			\gamma_{1}^{-1}\gamma_{0}\gamma_{1}&\text{if $i=0$},\\
			(\gamma_{0}\gamma_{1})^{-1}\gamma_{1}(\gamma_{0}\gamma_{1})&\text{if }i=1,\\
			\gamma_{i}&\text{otherwise},
		\end{cases}	
		\text{for }\gamma_{i}\in \pi_{1}(\mathbb{C}\backslash Q_{n}(a_{0}),d),
	\]
	the equation $(\mathrm{pr}_{t}\circ\iota_{a_{0}})_{*}(\gamma_{0})=1$
	shows that
	$(x_{1})_{*}=(\sigma_{0}^{2})_{*}\colon \pi_{1}(\mathbb{C}\backslash Q_{n},d)
	\rightarrow \pi_{1}(\mathbb{C}\backslash Q_{n},d)$
	is the identity map. 
	Then it inductively follows that all
	$(x_{i})_{*}\colon \pi_{1}(\mathbb{C}\backslash Q_{n},d)
	\rightarrow \pi_{1}(\mathbb{C}\backslash Q_{n},d)$
	are also the identity maps for $i=1,\ldots,n$.
\end{proof}
The following is a direct consequence of this lemma.
\begin{crl}\label{cor:trivact1}
	 Through the isomorphism $\pi_{1}(\mathbb{C}\backslash Q_{n},d)\ni \gamma_{i}\mapsto 
	 x_{i}\in F_{n}$, we regard $\alpha\colon G\rightarrow \mathrm{Aut}_{\mathbf{Grp}}(F_{n})$ as the 
	 group homomorphism $\alpha\colon G\rightarrow \mathrm{Aut}_{\mathbf{Grp}}(\pi_{1}(\mathbb{C}\backslash Q_{n},d))$.
	 Then the equations
	 \[
		(\mathrm{pr}_{t}\circ\iota_{a_{0}})_{*}\circ \eta(g)_{*}(\gamma)
		=\alpha(g)((\mathrm{pr}_{t}\circ\iota_{a_{0}})_{*}(\gamma))
	 \]
	 hold for all $g\in G$ and $\gamma\in \pi_{1}(\mathbb{C}\backslash Q_{n}(a_{0}),d)$, i.e.,
	 the diagrams 
	 \[
		\begin{tikzcd}
			\pi_{1}(\mathbb{C}\backslash Q_{n}(a_{0}),d)\arrow[r,"\eta(g)_{*}"]
			\arrow[d,"(\mathrm{pr}_{t}\circ\iota_{a_{0}})_{*}"]
			&\pi_{1}(\mathbb{C}\backslash Q_{n}(a_{0}),d)
			\arrow[d,"(\mathrm{pr}_{t}\circ\iota_{a_{0}})_{*}"]\\
			\pi_{1}(\mathbb{C}\backslash Q_{n},d)
			\arrow[r,"\alpha(g)"]
			&\pi_{1}(\mathbb{C}\backslash Q_{n},d)
		\end{tikzcd}
	 \]
	 are commutative.
	 Also for $x\in F_{n}$, we have
	 \[
		(\mathrm{pr}_{t}\circ\iota_{a_{0}})_{*}\circ \eta(x)_{*}(\gamma)
		=(\mathrm{pr}_{t}\circ\iota_{a_{0}})_{*}(\gamma).
	 \]
\end{crl}

Next we see 
the homomorphism 
$(\mathrm{pr}_{t-z}\circ\iota_{a_{0}})_{*}\colon 
\pi_{1}(\mathbb{C}\backslash Q_{n}(a_{0}),d)
\rightarrow \pi_{1}(\mathbb{C}^{\times},d-a_{0})$.
Similarly as above, we can see that 
the group homomorphism 
$(\mathrm{pr}_{t-z}\circ\iota_{a_{0}})_{*}$
coincides with the projection map 
\[
	\begin{array}{cccc}
	\mathrm{pr}^{0}\colon& F_{n+1}=\langle x_{0},x_{1},\ldots,x_{n} \rangle&\longrightarrow &F_{1}=\langle x\rangle\\
	&x_{i}&\longmapsto &\begin{cases} x&\text{if }i=0\\
		1&\text{otherwise}
	\end{cases}	
	\end{array},
\]
for 
which the inclusion 
map $\mathrm{inc}^{0}\colon F_{1}=\langle x\rangle \ni x\mapsto x_{0}\in F_{n+1}=\langle x_{0},x_{1},\ldots,x_{n}\rangle$
is a retraction homomorphism.

The homeomorphism $\mathbb{C}\ni z\to z-a_{0}\in \mathbb{C}$,
allows us to define 
the action of the subgroup $B_{1,n}$ of $\mathfrak{M}(D,Q_{n}(a_{0}))$
on $\mathbb{C}^{\times}$
which is compatible with 
that on $\mathbb{C}\backslash Q_{n}(a_{0})$
through  
$\mathrm{pr}_{t-z}\circ\iota_{a_{0}}
\colon \mathbb{C}\backslash Q_{n}(a_{0})\hookrightarrow \mathbb{C}^{\times}$.
Thus similarly as Lemma \ref{lem:trivact1} and Corollary \ref{cor:trivact1},
we obtain the following lemma and its corollary.
\begin{lmm}\label{lem:trivact2}
	$B_{1,n}$
	acts trivially on $\pi_{1}(\mathbb{C}^{\times},d-a_{0})$.
\end{lmm}
\begin{crl}\label{cor:trivact2}
	For all $h\in F_{n}\rtimes_{\alpha}G$, we have
	\[
	   (\mathrm{pr}_{t-z}\circ\iota_{a_{0}})_{*}\circ \eta(h)_{*}(\gamma)
	   =(\mathrm{pr}_{t-z}\circ\iota_{a_{0}})_{*}(\gamma).
	\]
\end{crl}

Combining Corollaries \ref{cor:trivact1} and \ref{cor:trivact2},
we obtain the following desired result.
\begin{prp}
	For each $h\in F_{n}\rtimes_{\alpha}G$ and $\gamma\in \pi_{1}(\mathbb{C}\backslash Q_{n}(a_{0}),d)$,
	the diagram 
\[
	\begin{tikzcd}
		\iota_{a_{0}}^{*}(C_{\lambda}(V))\arrow[r,"\tilde{\eta}(h)"]\arrow[d,"\gamma\cdot"]
		&\iota_{a_{0}}^{*}(C_{\lambda}(V))\arrow[d,"\eta(h)_{*}(\gamma)\cdot"]\\
		\iota_{a_{0}}^{*}(C_{\lambda}(V))\arrow[r,"\tilde{\eta}(h)"]
		&\iota_{a_{0}}^{*}(C_{\lambda}(V))
	\end{tikzcd}	
\]
is commutative.
\end{prp}
\begin{proof}
	Let us take $h\in F_{n}\rtimes_{\alpha}G$ arbitrarily
	and write $h=xg$ by $x\in F_{n}$ and $g\in G$.
	Then 
	we have 
	\begin{align*}
		\tilde{\eta}(h)(\gamma\cdot (v\otimes \kappa))&=
		\tilde{\eta}(x)\circ \tilde{\eta}(g)((\mathrm{pr}_{t}\circ\iota_{a_{0}})_{*}(\gamma)\cdot v\otimes (\mathrm{pr}_{t-z}\circ\iota_{a_{0}})_{*}(\gamma)\kappa)\\
		&=g\cdot((\mathrm{pr}_{t}\circ\iota_{a_{0}})_{*}(\gamma)\cdot v)\otimes (\mathrm{pr}_{t-z}\circ\iota_{a_{0}})_{*}(\gamma)\kappa\\
		&=\alpha(g)((\mathrm{pr}_{t}\circ\iota_{a_{0}})_{*}(\gamma))g\cdot v\otimes (\mathrm{pr}_{t-z}\circ\iota_{a_{0}})_{*}(\gamma)\kappa\\
		&=(\mathrm{pr}_{t}\circ\iota_{a_{0}})_{*}\circ \eta(g)_{*}(\gamma)\cdot g\cdot v\otimes (\mathrm{pr}_{t-z}\circ\iota_{a_{0}})_{*}(\gamma)\kappa\\
		&=(\mathrm{pr}_{t}\circ\iota_{a_{0}})_{*}\circ \eta(x)_{*}\circ \eta(g)_{*}(\gamma)\cdot g\cdot v\otimes (\mathrm{pr}_{t-z}\circ\iota_{a_{0}})_{*}(\gamma)\kappa\\
		&=(\mathrm{pr}_{t}\circ\iota_{a_{0}})_{*}(\eta(h)_{*}(\gamma))\cdot g\cdot v\otimes (\mathrm{pr}_{t-z}\circ\iota_{a_{0}})_{*}(\eta(h)_{*}(\gamma))\kappa\\
		&=\eta(h)_{*}(\gamma)\cdot (g\cdot v\otimes \kappa)\\
		&=\eta(h)_{*}(\gamma)\cdot \tilde{\eta}(h)(v\otimes \kappa)
	\end{align*}
	for all $\gamma\in \pi_{1}(\mathbb{C}\backslash Q_{n}(a_{0}),d)$ and
	$v\otimes \kappa\in V\otimes K_{\lambda}=\iota_{a_{0}}^{*}(C_{\lambda}(V))$.
	Here we use Corollary \ref{cor:trivact1} in 4th and 5th equations
	and use Corollary \ref{cor:trivact2}  
	in 6th equation. 
\end{proof}
\subsubsection{$F_{n}\rtimes_{\alpha}G$-action on $C_{n}(\mathbb{C}\backslash Q_{n}(a_{0});\iota_{a_{0}}^{*}(C_{\lambda}(V)))$}\label{sec:faction}
Let us define a left action of $F_{n}\rtimes_{\alpha}G$
on $C_{n}(\widetilde{\mathbb{C}\backslash Q_{n}(a_{0})})\otimes_{k}(V\otimes_{k}K_{\lambda})$
by
\[
	h\cdot (\sigma\otimes w):=\bar{\eta}(h)\circ \sigma\otimes \tilde{\eta}(h)(w)\qquad \text{for }
	\sigma\in C_{n}(\widetilde{\mathbb{C}\backslash Q_{n}(a_{0})}),\, w\in V\otimes_{k}K_{\lambda},
	\text{ and }h\in F_{n}\rtimes_{\alpha}G. 
\]
Then the following proposition assures that 
this induces  the left $F_{n}\rtimes_{\alpha}G$ action on 
 $C_{n}(\mathbb{C}\backslash Q_{n}(a_{0});\iota_{a_{0}}^{*}(C_{\lambda}(V)))$.
\begin{prp}
	The kernel of the projection map 
	\[
		\pi\colon C_{n}(\widetilde{\mathbb{C}\backslash Q_{n}(a_{0})})\otimes_{k}(V\otimes_{k}K_{\lambda})
		\rightarrow
		C_{n}(\widetilde{\mathbb{C}\backslash Q_{n}(a_{0})})\otimes_{k[\pi_{1}(\mathbb{C}\backslash Q_{n}(a_{0}),d)]}\iota_{a_{0}}^{*}(C_{\lambda}(V))		 	
	\]
	is preserved by the left $F_{n}\rtimes G$-action.
\end{prp}
\begin{proof}
	Let us take $\sigma\in C_{n}(\widetilde{\mathbb{C}\backslash Q_{n}(a_{0})})$,
	$w\in V\otimes_{k}K_{\lambda}$, and 
	$\gamma\in \pi_{1}(\mathbb{C}\backslash Q_{n}(a_{0}),d)$ arbitrarily.
	
	Then for $h\in F_{n}\rtimes_{\alpha}G$,
	we have 
	\begin{align*}
		&h\cdot (\sigma\cdot\gamma \otimes w-\sigma\otimes \gamma\cdot w)\\
		&\quad=
		h\cdot (\delta(\gamma^{-1})\circ\sigma \otimes w-\sigma\otimes \gamma\cdot w)\\
		&\quad=\bar{\eta}(h)\circ \delta(\gamma^{-1})\circ\sigma \otimes \tilde{\eta}(h)(w)-\bar{\eta}(x)\circ \sigma\otimes \tilde{\eta}(h)(\gamma\cdot w)\\
		&\quad=\delta(\eta(x)_{*}(\gamma^{-1}))\circ\bar{\eta}(x)\circ \sigma \otimes \tilde{\eta}(h)(w)-\bar{\eta}(x)\circ\sigma\otimes \eta(x)_{*}(\gamma)\cdot \tilde{\eta}(h)(w)\\
		&\quad=(\bar{\eta}(x)\circ \sigma)\cdot\eta(x)_{*}(\gamma)
		\otimes \tilde{\eta}(h)(w)- (\bar{\eta}(x)\circ \sigma)\otimes \eta(x)_{*}(\gamma)\cdot \tilde{\eta}(h)(w)
		\in \mathrm{Ker\,}\pi
	\end{align*}
	as desired.
\end{proof}
\subsubsection{Euler transform functor of $\mathbf{Mod}_{k[F_{n}\rtimes_{\alpha}G]}$}
The left $F_{n}\rtimes_{\alpha}G$-action on $C_{n}(\widetilde{\mathbb{C}\backslash Q_{n}(a_{0})};
\iota_{a_{0}}^{*}(C_{\lambda}(V)))$ is compatible with the boundary map, and thus 
it induces
the left $F_{n}\rtimes_{\alpha}G$-action on
$H_{n}(\widetilde{\mathbb{C}\backslash Q_{n}(a_{0})};
\iota_{a_{0}}^{*}(C_{\lambda}(V)))$.
\begin{dfn}[Euler transform functor of $\mathbf{Mod}_{k[F_{n}\rtimes_{\alpha}G]}$]\normalfont
	For $\lambda\in k^{\times}$,
	we define the functor 
	$\mathcal{E}_{\lambda}\colon \mathbf{Mod}_{k[F_{n}\rtimes_{\alpha}G]}
	\rightarrow \mathbf{Mod}_{k[F_{n}\rtimes_{\alpha}G]}$
	by 
	\[
		\mathcal{E}_{\lambda}(V):=H_{1}(\mathbb{C}\backslash Q_{n}(a_{0});
		\iota_{a_{0}}^{*}(C_{\lambda}(V)))\quad 
		\text{ for }V\in \mathscr{O}( \mathbf{Mod}_{k[F_{n}\rtimes_{\alpha}G]}).
	\]
	Also for $\phi\in \mathrm{Hom}_{k[F_{n}\rtimes_{\alpha}G]}(V,W)$,
	we define $\mathcal{E}_{\lambda}(V)\in 
	\mathrm{Hom}_{k[F_{n}\rtimes_{\alpha}G]}(\mathcal{E}_{\lambda}(V),\mathcal{E}_{\lambda}(W))$
	by 
	the homomorphism 
	$H_{1}(\mathbb{C}\backslash Q_{n}(a_{0});
	\iota_{a_{0}}^{*}(C_{\lambda}(V)))\rightarrow H_{1}(\mathbb{C}\backslash Q_{n}(a_{0});
	\iota_{a_{0}}^{*}(C_{\lambda}(W)))$
	induced from $\phi\colon V\rightarrow W$.
\end{dfn}

\subsection{Euler transform and twisted Long-Moody functor}
The Pochhammer contour integral is a classical tool to describe
the analytic continuation of 
the integral representation of the Gauss hypergeometric function
$F(\alpha,\beta,\gamma;z)=\frac{\Gamma(\gamma)}{\Gamma(\alpha)\Gamma(\gamma-\alpha)}\int_{1}^{\infty}
	t^{\beta-\gamma}(t-1)^{\gamma-\alpha-1}(t-z)^{-\beta}\,dt$
with respect to the parameters $\alpha,\beta$, and $\gamma$.
We shall explain that the Pochhammer contour integral
gives a standard decomposition of the homology group
$H_{1}(\mathbb{C}\backslash Q_{n}(a_{0});\iota_{a_{0}}^{*}(C_{\lambda}(V)))$.

\subsubsection{Open covering of $\mathbb{C}\backslash Q_{n}(a_{0})$}
Take an open cover $\mathbb{C}=X_{1}\cup Y_{1}$
so that  
\begin{enumerate}
	\item $X_{1}$ and $Y_{1}$ are connected,
	\item $X_{1}\cap Q_{n}(a_{0})=\{a_{0},a_{1}\}$, $Y_{1}\cap Q_{n}(a_{0})=\{a_{0},a_2,\ldots,a_n\}$ and $d\in X_{1}\cap Y_{1}$,
	\item $X_{1}\backslash\{a_{0},a_{1}\}$, $Y_{1}\backslash\{a_{0},a_{2},\ldots,a_{n}\}$,
	and $X_{1}\cap Y_{1}\backslash\{a_{0}\}$ respectively contain   
	$\gamma_{0}\cup \gamma_{1}$,\\
	\noindent $\gamma_{0}\cup \gamma_{2}\cup 
	\ldots\cup\gamma_{n}$, and $\gamma_{0}$ as deformation retracts.
\end{enumerate}
Put $X_{1}^{*}:=X_{1}\backslash\{a_{0},a_{1}\}$ and 
$Y_{1}^{*}:=Y_{1}\backslash\{a_{0},a_{2},\ldots,a_{n}\}$.
Let us consider the 
Mayer-Vietoris exact sequence 
\begin{multline*}
H_{1}(X_{1}^{*}\cap Y_{1}^{*};\iota_{a_{0}}^{*}(C_{\lambda}(V))|_{X_{1}^{*}\cap Y_{1}^{*}})
\rightarrow H_{1}(X_{1}^{*};\iota_{a_{0}}^{*}(C_{\lambda}(V))|_{X_{1}^{*}})\oplus 
H_{1}(Y_{1}^{*};\iota_{a_{0}}^{*}(C_{\lambda}(V))|_{Y_{1}^{*}})\\
\rightarrow H_{1}(\mathbb{C}\backslash Q_{n}(a_{0});\iota_{a_{0}}^{*}(C_{\lambda}(V)))
\rightarrow H_{0}(X_{1}^{*}\cap Y_{1}^{*};\iota_{a_{0}}^{*}(C_{\lambda}(V))|_{X_{1}^{*}\cap Y_{1}^{*}}),
\end{multline*}
where $\iota_{a_{0}}^{*}(C_{\lambda}(V))|_{X}$ denotes
 the pull-back of $\iota_{a_{0}}^{*}(C_{\lambda}(V))$
by the inclusion map $\iota_{X}\colon X\hookrightarrow \mathbb{C}\backslash Q_{n}(a_{0})$
for a subspace $X\subset \mathbb{C}\backslash Q_{n}(a_{0}).$
Then under the assumption $\lambda\neq 1$ we can show the following.
\begin{prp}\label{prop:decomp}
	Suppose that $\lambda\neq 1$.
	Then the above Mayer-Vietoris exact sequence gives 
the $k$-linear isomorphism 
\begin{equation*}
	H_{1}(X_{1}^{*};\iota_{a_{0}}^{*}(C_{\lambda}(V))|_{X_{1}^{*}})\oplus 
H_{1}(Y_{1}^{*};\iota_{a_{0}}^{*}(C_{\lambda}(V))|_{Y_{1}^{*}})\cong H_{1}(\mathbb{C}\backslash Q_{n}(a_{0});\iota_{a_{0}}^{*}(C_{\lambda}(V))).
\end{equation*}
\end{prp}
\begin{proof}
It suffices to show $H_{i}(X_{1}^{*}\cap Y_{1}^{*};\iota_{a_{0}}^{*}(C_{\lambda}(V))|_{X_{1}^{*}\cap Y_{1}^{*}})=\{0\}$,
$i=0,1$.

First we recall that 
the 0-th homology 
group $H_{0}(X;W)$ with a $k[\pi_{1}(X)]$-module $W$
is isomorphic to $W/W^{0}$ as vector spaces, where 
$W^{0}$ is the subspace of $W$ generated by  
$(x-1)\cdot w$ for $x\in \pi_{1}(X)$ and $w\in V$.
Since the generator $\gamma_{0}$ of $\pi_{1}(X_{1}^{*}\cap Y_{1}^{*},d)$ 
acts on $\iota_{a_{0}}^{*}(C_{\lambda}(V))|_{X_{1}^{*}\cap Y_{1}^{*}}$
as the multiplication of the scalar $\lambda$,
the assumption $\lambda\neq 1$
deduce the equation 
$\iota_{a_{0}}^{*}(C_{\lambda}(V))=\iota_{a_{0}}^{*}(C_{\lambda}(V))^{0}$,
which
shows 
$H_{0}(X_{1}^{*}\cap Y_{1}^{*};\iota_{a_{0}}^{*}(C_{\lambda}(L))|_{X_{1}^{*}\cap Y_{1}^{*}})
=\{0\}$.

For the 1st homology group, 
the following well-known lemma and the assumption $\lambda\neq 1$ deduce 
$H_{1}(X_{1}^{*}\cap Y_{1}^{*};\iota_{a_{0}}^{*}(C_{\lambda}(V)))\cong 
\iota_{a_{0}}^{*}(C_{\lambda}(V))^{\pi_{1}(X_{1}^{*}\cap Y_{1}^{*},d)}=\{0\}$.
Here, $W^{H}:=\{w\in W
\mid h\cdot w=w\text{ for all }h\in H\}$ is the $H$-invariant subspace  of $k[H]$-module $W$
for a group $H$. 
\end{proof}
\begin{lmm}\label{lem:Hat}
	Let us take the point $s_{0}:=(1,0)\in \mathbb{R}^{2}$
	and closed path $\omega \colon [0,1]\ni t\mapsto 
	(\cos 2\pi t,\sin 2\pi t)\in \mathbb{R}^{2}$ in $S^{1}$.
	Let $\pi_{S^{1}}\colon \widetilde{S_{1}}\rightarrow S^{1}$
	be the universal cover.
	We fix a point $\tilde{s}_{0}\in \pi_{S^{1}}^{-1}(s_{0})$
	and write the lift of $\omega$ to $\widetilde{S^{1}}$
	starting at $\tilde{s}_{0}$ by $\tilde{\omega}$.
	Then, for a left $k[\pi_{1}(S^{1},s_{0})]$-module $W$, the map 
	\[
		W^{\pi_{1}(S^{1},s_{0})}\ni v\longmapsto 
			\widetilde{\omega}\otimes v\in H_{1}(S^{1};W)
	\]
	is a $k$-linear isomorphism.
\end{lmm}
\begin{proof}
		This is a standard fact in algebraic topology, see \cite{Hat} for example.
\end{proof}

Let us look at 
$H_{1}(Y_{1}^{*};\iota_{a_{0}}^{*}(C_{\lambda}(V))|_{Y_{1}^{*}})$ 
and take an open cover 
$Y_{1}=X_{2}\cup Y_{2}$
so that 
\begin{enumerate}
	\item $X_{2}$ and $Y_{2}$ are connected,
	\item $X_{2}\cap Q_{n}(a_{0})=\{a_{0},a_{2}\}$, $Y_{2}\cap Q_{n}(a_{0})=\{a_{0},a_3,\ldots,a_n\}$ and $d\in X_{2}\cap Y_{2}$,
	\item $X_{2}\backslash\{a_{0},a_{2}\}$, $Y_{2}\backslash\{a_{0},a_{3},\ldots,a_{n}\}$,
	and $X_{2}\cap Y_{2}\backslash\{a_{0}\}$ respectively contain   \\
	\noindent$\gamma_{0}\cup \gamma_{2}$, 
	 $\gamma_{0}\cup \gamma_{3}\cup 
	\ldots\cup\gamma_{n}$, and $\gamma_{0}$ as deformation retracts.
\end{enumerate}
Put $X_{2}^{*}:=X_{2}\backslash\{a_{0},a_{2}\}$ and 
$Y_{2}^{*}:=Y_{2}\backslash\{a_{0},a_{3},\ldots,a_{n}\}$.
Then the same argument as above 
gives us the decomposition 
$$H_{1}(X_{2}^{*};\iota_{a_{0}}^{*}(C_{\lambda}(V))|_{X_{2}^{*}})\oplus 
H_{1}(Y_{2}^{*};\iota_{a_{0}}^{*}(C_{\lambda}(V))|_{Y_{2}^{*}})\cong H_{1}(Y_{1}^{*};\iota_{a_{0}}^{*}(C_{\lambda}(V))).$$
Iterating this procedure, we obtain the decomposition 
\begin{equation}\label{eq:pochdecomp}
	H_{1}(\mathbb{C}\backslash Q_{n}(a_{0});\iota_{a_{0}}^{*}(C_{\lambda}(V)))\cong
	\bigoplus_{i=1}^{n}	H_{1}(X_{i}^{*};\iota_{a_{0}}^{*}(C_{\lambda}(V))|_{X_{i}^{*}})
\end{equation}
under the assumption $\lambda\neq 1$,
where $X_{1}^{*},X_{2}^{*}\ldots,X_{n}^{*}$
are open covering of $\mathbb{C}\backslash Q_{n}(a_{0})$ defined as above.
\subsubsection{Pochhammer cycles}
For a path connected subspace $U\subset \mathbb{C}\backslash Q_{n}(a_{0})$ containing $d$,
we write the inclusion map by $\iota_U\colon U\hookrightarrow \mathbb{C}\backslash Q_{n}(a_{0})$
and consider the universal cover $\pi_{U}\colon \widetilde{U}\rightarrow U$.
Let us fix a point $\bar{d}_{U}\in \pi_{U}^{-1}(d)$ and 
write the unique lift of $\iota_{U}\circ \pi_{U}$ to the universal cover 
$\widetilde{\mathbb{C}\backslash Q_{n}(a_{0})}$
sending $\bar{d}_{U}$ to $\bar{d}$
by $\widetilde{\iota_{U}}$.
\[
	\begin{tikzcd}
		\widetilde{U}\arrow[r,"\widetilde{\iota_{U}}"]\arrow[d,"\pi_{U}"]&\widetilde{\mathbb{C}\backslash Q_{n}(a_{0})}
		\arrow[d,"\pi_{\mathbb{C}\backslash Q_{n}(a_{0})}"]\\
		U\arrow[r,"\iota_{U}"]&\mathbb{C}\backslash Q_{n}(a_{0})
	\end{tikzcd}
\]
For a path $\gamma$ in $U$ starting at $d$, $\widetilde{\gamma}$ denotes the unique lift 
of $\gamma$ to $\widetilde{U}$ starting at $\bar{d}_{U}$.
Then $\widetilde{\iota_{U}}_{*}(\widetilde{\gamma})$
is the lift of $\gamma$ to $\widetilde{\mathbb{C}\backslash Q_{n}(a_{0})}$
starting at $\bar{d}$.

For two elements $a,b$ in a group $H$, we write the commutator of them by $[a,b]:=a^{-1}b^{-1}ab$.
\begin{dfn}[Pochhammer cycle]\normalfont
	 $1$-chains 
	$\widetilde{[\gamma_{0},\gamma_{i}]}\otimes c\in C_{1}(X_{i}^{*};\iota_{a_{0}}^{*}(C_{\sigma}(V))|_{X_{i}^{*}})$
	for $c\in \iota_{a_{0}}^{*}(C_{\sigma}(V))$
	are called 
	\emph{Pochhammer cycles} in $X_{i}^{*}$.
	Also 1-chains 
	 $\widetilde{\iota_{X_{i}^{*}}}_{*}(\widetilde{[\gamma_{0},\gamma_{i}]})\otimes c\in C_{1}(\mathbb{C}\backslash Q_{n}(a_{0});\iota_{a_{0}}^{*}(C_{\sigma}(V)))$
	are called Pochhammer cycles in $\mathbb{C}\backslash Q_{n}(a_{0})$.
\end{dfn}
As we see below, a Pochhammer cycles is not just a chain but literally a cycle.
\begin{lmm}\label{lem:pochhammer}
For $c\in \iota_{a_{0}}^{*}(C_{\sigma}(V))$ we 
consider $\widetilde{[\gamma_{0},\gamma_{i}]}\otimes c\in C_{1}(X_{i}^{*};\iota_{a_{0}}^{*}(C_{\lambda}(V))|_{X_{i}^{*}})$.
Then $\widetilde{[\gamma_{0},\gamma_{i}]}\otimes c$ is a cycle,
i.e., it is contained in the kernel of the boundary map 
$\partial_{1}\colon C_{1}(X_{i}^{*};\iota_{a_{0}}^{*}(C_{\lambda}(V))|_{X_{i}^{*}})\rightarrow C_{0}(X_{i}^{*};\iota_{a_{0}}^{*}(C_{\lambda}(V))|_{X_{i}^{*}})$.

Similarly, the Pochhammer cycle $\widetilde{\iota_{X_{i}^{*}}}_{*}(\widetilde{[\gamma_{0},\gamma_{i}]})\otimes c$
in $\mathbb{C}\backslash Q_{n}(a_{0})$ is also a cycle.
\end{lmm}
\begin{proof}
	The second assertion follows from the first one since 
	$\widetilde{\iota_{X_{i}^{*}}}$ induces a chain map.
	Let us show the first assertion.
	Since $\gamma_{0}$ acts on $\iota_{a_{0}}^{*}(C_{\lambda}(V))$
	as the multiplication of scalar $\lambda$,
	we have the following as desired, 
	\begin{align*}
	\partial_{1}(\widetilde{[\gamma_{0},\gamma_{i}]}\otimes c)
	&=\widetilde{[\gamma_{0},\gamma_{i}]}(0)\otimes c-\widetilde{[\gamma_{0},\gamma_{i}]}(1)\otimes c\\
	&=\bar{d}_{X_{i}^{*}}\otimes c-[\gamma_{0},\gamma_{i}]\cdot \bar{d}_{X_{i}^{*}}\otimes c\\
	&=\bar{d}_{X_{i}^{*}}\otimes c-\bar{d}_{X_{i}^{*}}\otimes [\gamma_{0},\gamma_{i}]^{-1}\cdot c\\
	&=\bar{d}_{X_{i}^{*}}\otimes c-\bar{d}_{X_{i}^{*}}\otimes (\gamma_{i}^{-1}\gamma_{0}^{-1}\gamma_{i}\gamma_{0})\cdot c\\
	&=\bar{d}_{X_{i}^{*}}\otimes c-\bar{d}_{X_{i}^{*}}\otimes \lambda^{-1}\cdot\lambda(\gamma_{i}^{-1}\cdot \gamma_{i})c\\
	&=\bar{d}_{X_{i}^{*}}\otimes c-\bar{d}_{X_{i}^{*}}\otimes c=0.
	\end{align*}
\end{proof}
It follows that $H_{1}(X_{i}^{*};\iota_{a_{0}}^{*}(C_{\lambda}(V))|_{X_{i}^{*}})$
is generated by Pochhammer cycles, as below.
\begin{prp}\label{prop:poch}
	Under the assumption $\lambda\neq 1$, the map
	\[
		\mathrm{Poch}_{i}\colon \iota_{a_{0}}^{*}(C_{\lambda}(V))\ni c\longmapsto \widetilde{[\gamma_{0},\gamma_{i}]}\otimes c
		\in H_{1}(X_{i}^{*};\iota_{a_{0}}^{*}(C_{\lambda}(V))|_{X_{i}^{*}})
	\]
	is a linear isomorphism.
\end{prp}
Before showing this, we prepare the following lemmas.
\begin{lmm}\label{lem:dim}
	Let us suppose $\lambda\neq 1$.
	Let us take an open cover $X^{*}_{i}=U_{a_{i}}\cup U_{a_{0}}$ 
	where $U_{a_{i}}$ (resp. $U_{a_{0}}$) is the intersection of 
	$X^{*}_{i}$ and an open neighborhood of $a_{i}$
	(resp. $a_{0}$) in $\mathbb{C}$.
	We assume that 
	$U_{a_{i}}$ and $U_{a_{0}}$
	contain respectively $\gamma_{i}$ and $\gamma_{0}$ as deformation retracts and    
	moreover assume that the intersection $U_{a_{i}}\cap U_{a_{0}}$
	is simply connected.
	Then there exists an isomorphism
	\[
		H_{1}(X^{*}_{i};\iota_{a_{0}}^{*}(C_{\lambda}(V))|_{X^{*}_{i}})\xrightarrow{\sim}
		H_{0}(U_{a_{i}}\cap U_{a_{0}};\iota_{a_{0}}^{*}(C_{\lambda}(V))|_{U_{a_{i}}\cap U_{a_{0}}}).
	\]
	In particular, we have the isomorphism $H_{1}(X^{*}_{i};\iota_{a_{0}}^{*}(C_{\lambda}(V))|_{X^{*}_{i}})
	\cong \iota_{a_{0}}^{*}(C_{\lambda}(V))$ as vector spaces.
\end{lmm}
\begin{proof}
	The second assertion immediately follows from the first one
	since we have $H_{0}(U_{a_{i}}\cap U_{a_{0}};\iota_{a_{0}}^{*}(C_{\lambda}(V))|_{U_{a_{i}}\cap U_{a_{0}}})
	\cong \iota_{a_{0}}^{*}(C_{\lambda}(V))$ from the simply connectedness of $U_{a_{i}}\cap U_{a_{0}}$.

	We set $\bar{V}=\iota_{a_{0}}^{*}(C_{\lambda}(V))$ for simplicity.
	Recalling that $\gamma_{0}\cdot v=\lambda v$ for all $v\in \bar{V}$,
	and also $\lambda \neq 1$, we obtain $H_{0}(X^{*}_{i};\bar{V}|_{X^{*}_{i}})=\{0\}$ and 
	 $H_{j}(U_{a_{0}};\bar{V}|_{U_{a_{0}}})=\{0\}$, $j=0,1$, as in Proposition \ref{prop:decomp}.
	 Therefore the Mayer-Vietoris long exact sequence gives 
	 us the following exact sequence. 	
	 \begin{multline*}
		\{0\}\rightarrow H_{1}(U_{a_{i}};\bar{V}|_{U_{a_{i}}})\rightarrow H_{1}(X^{*}_{i};\bar{V}|_{X^{*}_{i}})\\\xrightarrow{\delta}
		H_{0}(U_{a_{i}}\cap U_{a_{0}};\bar{V}|_{U_{a_{i}}\cap U_{a_{0}}})\rightarrow H_{0}(U_{a_{i}};\bar{V}|_{U_{a_{i}}})
		\rightarrow \{0\}
	 \end{multline*}

	Let us consider the multiplication map 
	$(\gamma_{i}-1)\colon \bar{V}\ni v\mapsto (\gamma_{i}-1)v\in \bar{V}$
	and then we obtain  
	that 
	\begin{align*}
	H_{1}(U_{a_{i}};\bar{V}|_{U_{a_{i}}})&\cong 
	\{v\in \bar{V}\mid (\gamma_{i}-1)\cdot v=0\}=\mathrm{Ker\,}(\gamma_{i}-1),\\
		H_{0}(U_{a_{i}};\bar{V}|_{U_{a_{i}}})&\cong 
		\bar{V}/\{(\gamma_{i}-1)v\mid v\in\bar{V}\}
		=\mathrm{Coker\,}(\gamma_{i}-1),
	\end{align*}
	since $U_{a_{i}}$ is homotopy equivalent to $S^{1}$.  
	Thus it follows that  
	\[
		H_{1}(U_{a_{i}};\bar{V}|_{U_{a_{i}}})\cong \mathrm{Ker\,}(\gamma_{i}-1)\cong \mathrm{Coker\,}(\gamma_{i}-1)\cong H_{0}(U_{a_{i}};\bar{V}|_{U_{a_{i}}})	
	\]
	since every vector space is an injective and projective module, and therefore the
	kernel and cokernel of an endmorphism of a vector space are isomorphic in general.
	
	Thus we obtain that $\mathrm{Ker\,}\delta\cong \mathrm{Coker\,}\delta$. Moreover,
	since vector spaces are injective and projective,
	we have isomorphisms
	\begin{align*}
		H_{1}(X^{*}_{i};\iota_{a_{0}}^{*}(C_{\lambda}(V))|_{X^{*}_{i}})&\cong \mathrm{Coim\,}\delta\oplus \mathrm{Ker\,}\delta,\\
		H_{0}(U_{a_{i}}\cap U_{a_{0}};\iota_{a_{0}}^{*}(C_{\lambda}(V))|_{U_{a_{i}}\cap U_{a_{0}}})&\cong \mathrm{Im\,}\delta\oplus \mathrm{Coker\,}\delta. 
	\end{align*}
	Then isomorphisms $\mathrm{Ker\,}\delta\cong \mathrm{Coker\,}\delta$ and $\mathrm{Coim\,}\delta\cong \mathrm{Im\,}\delta$
	induce the desired isomorphism.
\end{proof}

For a subspace $C\subset X$ of a connected manifold $X$, $S_{n}^{C}(\widetilde{X})$ denotes 
the set of all singular $n$-simplices $\sigma$ satisfying 
$\mathrm{Im}(\pi_{X}\circ \sigma)\subset C$ and $C_{n}^{C}(\widetilde{X};k)$
denotes the $k$-vector space freely generated by $S_{n}^{C}(\widetilde{X})$.
For a $k$-local system $L$, we define $C_{n}^{C}(X;L):=C_{n}^{C}(\widetilde{X};k)\otimes_{k[\pi_{1}(X)]}L$.
Then the following is well-known as a key lemma for 
the Mayer-Vietoris long exact sequence
and    
proved by 
the similar argument as for usual singular homology groups
with constant coefficients.
\begin{lmm}\label{lmm:mv}
	Let $X$ be a topological manifold and $X=A\cup B$ an open covering.
	Let $L$ be a $k$-local system on $X$.
	We consider the sum
	 $C_{n}^{A}(X;L)+C_{n}^{B}(X;L)$ in the vector space $C_{n}(X;L)$.
	Then we have the following.
	\begin{enumerate}
		\item The inclusion map $C_{n}^{A}(X;L)+C_{n}^{B}(X;L)\hookrightarrow C_{n}(X;L)$
		induces the isomorphism of homology groups.

		\item 
	We have a short exact sequence 
	\begin{multline*}
		\{0\}\rightarrow C_{n}(A\cap B; L|_{A\cap B})\xrightarrow[]{f}
		C_{n}(A;L|_{A})\oplus C_{n}(B;L|_{B})\\
		\xrightarrow[]{g}
		C_{n}^{A}(X;L)+C_{n}^{B}(X;L)\rightarrow\{0\}.
	\end{multline*}
	\end{enumerate}
	Here $f,\,g$ are defined as follows.
	The map $g$ is the sum of the maps $C_{n}(*;L|_{*})\rightarrow C^{*}_{n}(X;L)$ induced by the 
	inclusions $*\hookrightarrow X$
	for $*=A,B$,
	and $f$ is the difference of the maps $C_{n}(A\cap B; L|_{A\cap B})
	\rightarrow C_{n}(*;V|_{*})$ from inclusions $A\cap B\hookrightarrow *$
	for $*=A,B$.
	
\end{lmm}

Now we come back to the proof of the proposition.
\begin{proof}[Proof of Proposition \ref{prop:poch}]
	Lemma \ref{lem:dim} tells us that 
	$H_{1}(X_{i}^{*};\iota_{a_{0}}^{*}(C_{\lambda}(V))|_{X_{i}^{*}})$ is isomorphic to $\iota_{a_{0}}^{*}(C_{\lambda}(V))$
	as a vector space.
	Thus we only need to check the injectivity of the 
	map $\mathrm{Poch}_{i}$ from the freeness of vector spaces.
		
	Let us write $\bar{V}=\iota_{a_{0}}^{*}(C_{\lambda}(V))$.
	We assume that
	$v\in \bar{V}$ satisfies $\mathrm{Poch}_{i}(v)=0$ and will show that $v=0$.
	We take the same cover $X_{i}^{*}=U_{a_{i}}\cup U_{a_{0}}$
	as in Lemma \ref{lem:dim}.
	Let us consider the commutative diagram 
	\[
	  \begin{tikzcd}
		C_{2}(U_{a_{i}}\cap U_{a_{0}})\arrow[r,"f_{2}"]\arrow[d,"\partial_{2}"]& C_{2}(U_{a_{i}})\oplus C_{2}(U_{a_{0}})\arrow[r,"g_{2}"]\arrow[d,"\partial_{2}"]&
		C_{2}^{U_{a_{i}}}(X_{i}^{*})+C_{2}^{U_{a_{0}}}(X_{i}^{*})\arrow[d,"\partial_{2}"]\\
		C_{1}(U_{a_{i}}\cap U_{a_{0}})\arrow[r,"f_{1}"]\arrow[d,"\partial_{1}"]& C_{1}(U_{a_{i}})\oplus C_{1}(U_{a_{0}})\arrow[r,"g_{1}"]\arrow[d,"\partial_{1}"]&
		C_{1}^{U_{a_{i}}}(X_{i}^{*})+C_{1}^{U_{a_{0}}}(X_{i}^{*})\arrow[d,"\partial_{1}"]\\
		C_{0}(U_{a_{i}}\cap U_{a_{0}})\arrow[r,"f_{0}"]& C_{0}(U_{a_{i}})\oplus C_{0}(U_{a_{0}})\arrow[r,"g_{0}"]&
		C_{0}^{U_{a_{i}}}(X_{i}^{*})+C_{0}^{U_{a_{0}}}(X_{i}^{*})
	  \end{tikzcd},
	\]
	following from Lemma \ref{lmm:mv}, where we omit the coefficients $\bar{V}$ for simplicity.
	We note that horizontal sequences are exact and also note that 
	$f_{i}$ are injective and $g_{i}$ are surjective.

	By the assumption, $\mathrm{Poch}_{i}(v)=\widetilde{[\gamma_{0},\gamma_{i}]}\otimes v$
	is homologous to $0$. Thus from the surjectivity of $g_{2}$, there exits 
	$\sigma_{2}\in C_{2}(U_{a_{1}};\bar{V}|_{U_{a_{1}}})\oplus C_{2}(U_{a_{0}};\bar{V}|_{U_{a_{0}}})$
	such that $\widetilde{[\gamma_{0},\gamma_{i}]}\otimes v=\partial_{2}\circ g_{2}(\sigma_{2})$.
	On the other hand, we have 
	\begin{align*}
		\widetilde{[\gamma_{0},\gamma_{i}]}\otimes v&=\widetilde{\gamma_{0}^{-1}}\otimes v
		+\widetilde{\gamma_{i}^{-1}}\otimes (\gamma_{0}\cdot v)
		+\widetilde{\gamma_{0}}\otimes (\gamma_{i}\gamma_{0}\cdot v)
		+\widetilde{\gamma_{i}}\otimes (\gamma_{0}^{-1}\gamma_{i}\gamma_{0}\cdot v)\\
		&=\left(
			\widetilde{\gamma_{i}^{-1}}\otimes \lambda c+\widetilde{\gamma_{i}}\otimes (\gamma_{i}\cdot v)
		\right)
		+\left(
			\widetilde{\gamma_{0}^{-1}}\otimes v+\widetilde{\gamma_{0}}\otimes \lambda (\gamma_{i}\cdot v)
		\right),
	\end{align*}
	which shows that 
	\[
		g_{1}\left(
			\left(
			\widetilde{\gamma_{i}^{-1}}\otimes \lambda v+\widetilde{\gamma_{i}}\otimes (\gamma_{i}\cdot v)
		\right)
		+\left(
			\widetilde{\gamma_{0}^{-1}}\otimes v+\widetilde{\gamma_{0}}\otimes \lambda (\gamma_{i}\cdot v)
		\right)
		\right)
		=	\widetilde{[\gamma_{0},\gamma_{i}]}\otimes v.
	\]
	From the commutativity of the diagram, we also have $g_{1}\circ \partial_{2}(\sigma_{2})=\widetilde{[\gamma_{0},\gamma_{i}]}\otimes v$.
	Therefore it follows from the injectivity of $f_{1}$ that 
	there uniquely exists $\sigma_{1}\in C_{1}(U_{a_{i}}\cap U_{a_{0}};\bar{V}|_{U_{a_{i}}\cap U_{a_{0}}})$ such that 
	\[
		f_{1}(\sigma_{1})=	\left(
			\widetilde{\gamma_{i}^{-1}}\otimes \lambda v+\widetilde{\gamma_{i}}\otimes (\gamma_{i}\cdot v)
		\right)
		+\left(
			\widetilde{\gamma_{0}^{-1}}\otimes v+\widetilde{\gamma_{0}}\otimes \lambda (\gamma_{i}\cdot v)
		\right)
		-\partial_{2}(\sigma_{2}).
	\]

	By the way, the image of  $\left(
		\widetilde{\gamma_{i}^{-1}}\otimes \lambda c+\widetilde{\gamma_{i}}\otimes (\gamma_{i}\cdot v)
	\right)
	+\left(
		\widetilde{\gamma_{0}^{-1}}\otimes v+\widetilde{\gamma_{0}}\otimes \lambda (\gamma_{i}\cdot v)
	\right)
	\in
	C_{1}(U_{a_{1}};\bar{V}|_{U_{a_{1}}})\oplus C_{1}(U_{a_{0}};\bar{V}|_{U_{a_{0}}})
	$
	by the boundary map $\partial_{1}$
	is 
	\[
		\left(\bar{d}_{U_{a_{i}}}\otimes
		-(\lambda-1)(\gamma_{i}-1)v
		\right)
		+
		\left(\bar{d}_{U_{a_{0}}}\otimes
		(\lambda-1)(\gamma_{i}-1)v
		\right),
	\]
	which is the image of 
	$\bar{d}_{U_{a_{i}}\cap U_{a_{0}}}\otimes
	-(\lambda-1)(\gamma_{i}-1)v
	\in C_{0}(U_{a_{i}}\cap U_{a_{0}};\bar{V}|_{U_{a_{i}}\cap U_{a_{0}}})$
	by $f_{0}$.
	Then from the injectivity of $f_{0}$,
	we have 
	\[
		\partial_{1}(\sigma_{1})=	\bar{d}_{U_{a_{i}}\cap U_{a_{0}}}\otimes
		-(\lambda-1)(\gamma_{i}-1)v.
	\]
	Thus we obtain $(\lambda-1)(\gamma_{i}-1)v=0$
	from the following two facts.
	One is that we have $C_{0}(U_{a_{i}}\cap U_{a_{0}};\bar{V}|_{U_{a_{i}}\cap U_{a_{0}}})=C_{0}(U_{a_{i}}\cap U_{a_{0}};k)\otimes_{k}\bar{V}$
	since $U_{a_{i}}\cap U_{a_{0}}$ is simply connected.
	Another is that 
	for a $0$-connected topological space $Y$, 
	the image of the boundary map $\partial_{1}\colon C_{1}(Y;k)\rightarrow C_{0}(Y;k)$
	equals the kernel of the augmentation map $\epsilon \colon C_{0}(Y;k)\ni \sum_{\sigma\in S_{0}(Y)}r_{\sigma}\cdot \sigma 
	\mapsto \sum_{\sigma\in S_{0}(Y)}r_{\sigma}\in k$.
	
	Therefore we obtain $(\gamma_{i}-1)v=0$ since $\lambda\neq 1$, and moreover obtain 
	$\partial_{1}\circ f_{1}(\sigma_{1})=f_{0}\circ \partial_{1}(\sigma_{1})=0$.
	This implies that 
	\begin{align*}
		0&=\partial_{1}(\widetilde{\gamma_{i}^{-1}}\otimes \lambda v+\widetilde{\gamma_{i}}\otimes (\gamma_{i}\cdot v))
		=\partial_{1}(\widetilde{\gamma_{i}^{-1}}\otimes \lambda v+\widetilde{\gamma_{i}}\otimes v)\\
		&=\partial_{1}(-\widetilde{\gamma_{i}}\otimes \lambda v+\widetilde{\gamma_{i}}\otimes v)
		=\partial_{1}(\widetilde{\gamma_{i}}\otimes (1-\lambda)v)
	\end{align*}
	which shows $(1-\lambda)v=0$ since we can use Lemma \ref{lem:Hat} by the condition $(\gamma_{i}-1)v=0$.
	Therefore we obtain $v=0$ as desired.
\end{proof}
Combining this proposition with the decomposition $(\ref{eq:pochdecomp})$,
we conclude that 
a base of 
$H_{1}(\mathbb{C}\backslash Q_{n}(a_{0});\iota_{a_{0}}^{*}(C_{\lambda}(V)))$
is given by Pochhammer cycles.
\begin{thm}\label{thm:poch}
	Let $V$ be a left $k[F_{n}\rtimes_{\alpha}G]$-module.
	Let us take $\lambda\in k^{\times}\backslash\{1\}$.
	Since we have $\iota_{a_{0}}(C_{\lambda}(V))=V\otimes_{k}K_{\lambda}\cong V$ as vector spaces, 
	we identify $\iota_{a_{0}}(C_{\lambda}(V))$ with $V$.
	Then 
	\[
		\begin{array}{cccc}
			\mathrm{Poch}\colon &V^{\oplus n}&\longrightarrow &H_{1}(\mathbb{C}\backslash Q_{n}(a_{0});\iota_{a_{0}}(C_{\lambda}(V)))\\
			&(v_{1},\ldots,v_{n})&\longmapsto &\sum_{i=1}^{n}\widetilde{[\gamma_{0},\gamma_{i}]}\otimes v_{i}
		\end{array}
	\]
	is an isomorphism of vector spaces.
\end{thm}
\subsubsection{K\"unneth formula}
Let us consider $\mathcal{E}_{\lambda}(k[F_{n}])=H_{1}(\mathbb{C}\backslash Q_{n}(a_{0});
\iota_{a_{0}}^{*}(C_{\lambda}(k[F_{n}])))$.
Then since 
$\iota_{a_{0}}^{*}(C_{\lambda}(k[F_{n}]))=
k[F_{n}]\otimes_{k}K_{\lambda}$
has the natural $k[F_{n}]$-bimodule structure,
$C_{m}(\mathbb{C}\backslash Q_{n}(a_{0});\iota_{a_{0}}^{*}(C_{\lambda}(k[F_{n}])))
=C_{m}(\mathbb{C}\backslash Q_{n}(a_{0});k)\otimes_{k[\pi_{1}(\mathbb{C}\backslash Q_{n}(a_{0}),d)]}
\iota_{a_{0}}^{*}(C_{\lambda}(k[F_{n}]))$
can be naturally regarded as right $k[F_{n}]$-modules
and also the homology groups $H_{m}(\mathbb{C}\backslash Q_{n}(a_{0});
\iota_{a_{0}}^{*}(C_{\lambda}(k[F_{n}])))$
have the induced right $k[F_{n}]$-module structures.

Then 
Theorem \ref{thm:poch} implies that the first homology $H_{1}(\mathbb{C}\backslash Q_{n}(a_{0});
\iota_{a_{0}}^{*}(C_{\lambda}(k[F_{n}])))$
is free as the right $k[F_{n}]$-module and 
we obtain the following K\"unneth formula.
\begin{prp}\label{prop:kunneth}
	Let $V$ be a $k[F_{n}\rtimes_{\alpha}G]$-module.
	Let us take $\lambda\in k^{\times}\backslash\{1\}$.
	Then there exists an
	isomorphism 
	\[
		H_{1}(\mathbb{C}\backslash Q_{n}(a_{0});\iota_{a_{0}}^{*}(C_{\lambda}(k[F_{n}])))
		\otimes_{k[F_{n}]}V
		\xrightarrow{\sim}
		H_{1}(\mathbb{C}\backslash Q_{n}(a_{0});\iota_{a_{0}}^{*}(C_{\lambda}(V))).		
	\]
\end{prp}
\begin{proof}
Theorem \ref{thm:poch} shows that 
$H_{1}(\mathbb{C}\backslash Q_{n}(a_{0});\iota_{a_{0}}^{*}(C_{\lambda}(k[F_{n}])))$
is the free right $k[F_{n}]$-module generated by 
$\widetilde{[\gamma_{0},\gamma_{i}]}\otimes \mathbf{1}$
for $i=1,\ldots,n$, where $\mathbf{1}\in \iota_{a_{0}}^{*}(C_{\lambda}(k[F_{n}]))=
k[F_{n}]\otimes_{k} K_{\lambda}$ is $1\otimes 1$ 
with $1\in k[F_{n}]$ and  $1\in k=K_{\lambda}$.

Therefore every element in 
$H_{1}(\mathbb{C}\backslash Q_{n}(a_{0});\iota_{a_{0}}^{*}(C_{\lambda}(k[F_{n}])))
\otimes_{k[F_{n}]}V$ is uniquely 
written in the form $\sum_{i=1}^{n}(\widetilde{[\gamma_{0},\gamma_{i}]}\otimes \mathbf{1})
\otimes v_{i}$
with $v_{i}\in V$,
and then we can define the map
\[
	\begin{array}{ccc}
		H_{1}(\mathbb{C}\backslash Q_{n}(a_{0});\iota_{a_{0}}^{*}(C_{\lambda}(k[F_{n}])))
		\otimes_{k[F_{n}]}V&
		\longrightarrow&
		H_{1}(\mathbb{C}\backslash Q_{n}(a_{0});\iota_{a_{0}}^{*}(C_{\lambda}(V)))\\
		\sum_{i=1}^{n}(\widetilde{[\gamma_{0},\gamma_{i}]}\otimes \mathbf{1})
\otimes v_{i}&
		\longmapsto&
		\sum_{i=1}^{n}\widetilde{[\gamma_{0},\gamma_{i}]}\otimes (v_{i}\otimes 1)
	\end{array}	
\]
which is an isomorphism by Theorem \ref{thm:poch}.
\end{proof}

\subsubsection{Euler transform and twisted Long-Moody functor}
We give an equivalence of functors between the Euler transform functor and 
twisted Long-Moody functor.

\begin{prp}\label{prop:aughom}
	Let us take $\lambda\in k^{\times}\backslash\{1\}$.
	Then the map 
	\[
		\begin{array}{cccc}
			\Phi\colon& H_{1}(\mathbb{C}\backslash Q_{n}(a_{0});\iota_{a_{0}}^{*}(C_{\lambda}(k[F_{n}])))&
			\longrightarrow&
			I_{F_{n},\,\lambda}\\
			&\sum_{i=1}^{n}(\widetilde{[\gamma_{0},\gamma_{i}]}\otimes (x^{(i)}\otimes \kappa^{(i)}))
			&\longmapsto&
			\sum_{i=1}^{n}(x_{i}-1)(\kappa^{(i)}x^{(i)})
		\end{array}	
	\]
	is a left $k[F_{n}\rtimes_{\theta_{\mathrm{Artin}}}B_{n}]$-module isomorphism.
	Where $x^{(i)}\in k[F_{n}]$ and $\kappa^{(i)}\in k=K_{\lambda}$.
	Moreover this is a right $k[F_{n}]$-module isomorphism.
\end{prp}
\begin{proof}
	Theorem \ref{thm:poch} shows that 
	$\Phi$ is a right $k[F_{n}]$-module isomomorphism.
	Thus it suffices to show the $\Phi$ is a left $k[F_{n}\rtimes_{\theta_{\mathrm{Artin}}}B_{n}]$-module
	homomorphism.

	Let us
	notice that for $\chi\in \iota_{a_{0}}(C_{\lambda}(k[F_{n}]))$ we have 
	\begin{align*}
		\widetilde{[\gamma_{0},\gamma_{i}]}\otimes \chi&=
		\widetilde{\gamma_{0}^{-1}}\otimes \chi+\widetilde{\gamma_{i}^{-1}}\otimes (\gamma_{0}\cdot \chi)
		+\widetilde{\gamma_{0}}\otimes (\gamma_{i}\gamma_{0}\cdot \chi)+\widetilde{\gamma_{i}}\otimes (\gamma_{0}^{-1}\gamma_{i}\gamma_{0}\cdot \chi)\\
		&=-\widetilde{\gamma_{0}}\otimes (\gamma_{0}\cdot \chi)-\widetilde{\gamma_{i}}\otimes (\gamma_{i}\gamma_{0}\cdot \chi)
		+\widetilde{\gamma_{0}}\otimes (\gamma_{i}\gamma_{0}\cdot \chi)+\widetilde{\gamma_{i}}\otimes (\gamma_{0}^{-1}\gamma_{i}\gamma_{0}\cdot \chi)\\
		&=-\widetilde{\gamma_{0}}\otimes \lambda \chi-\widetilde{\gamma_{i}}\otimes \lambda(\gamma_{i}\cdot \chi)
		+\widetilde{\gamma_{0}}\otimes \lambda (\gamma_{i}\cdot \chi)+\widetilde{\gamma_{i}}\otimes (\gamma_{i}\cdot \chi)\\
		&=\widetilde{\gamma_{0}}\otimes \lambda(\gamma_{i}-1)\cdot \chi+\widetilde{\gamma_{i}}\otimes (1-\lambda)(\gamma_{i}\cdot \chi)
	\end{align*}
	in $H_{1}(\mathbb{C}\backslash Q_{n}(a_{0});\iota_{a_{0}}(C_{\lambda}(k[F_{n}])))$.
	Then we have
	\begin{align*}
		&\sigma_{j}\cdot (\widetilde{[\gamma_{0},\gamma_{i}]}\otimes \chi)\\
		&=
		\sigma_{j}\cdot (\widetilde{\gamma_{0}}\otimes \lambda(\gamma_{i}-1)\cdot \chi+\widetilde{\gamma_{i}}\otimes (1-\lambda)(\gamma_{i}\cdot \chi) )\\
		&=\bar{\eta}(\sigma_{j})_{*}(\widetilde{\gamma_{0}})\otimes \tilde{\eta}(\sigma_{j})(\lambda(\gamma_{i}-1)\cdot \chi)
		+\bar{\eta}(\sigma_{j})_{*}(\widetilde{\gamma_{i}})\otimes \tilde{\eta}(\sigma_{j})((1-\lambda)(\gamma_{i}\cdot \chi))\\
		&=\widetilde{\gamma_{0}}\otimes \tilde{\eta}(\sigma_{j})(\lambda(\gamma_{i}-1)\cdot \chi)
		+\begin{cases}
			\widetilde{\gamma_{i+1}}\otimes  \tilde{\eta}(\sigma_{i})((1-\lambda)(\gamma_{i}\cdot \chi))&\text{if }j=i,\\
			\widetilde{(\gamma_{i+1}^{-1}\cdot \gamma_{i}\cdot \gamma_{i+1})} \otimes \tilde{\eta}(\sigma_{i+1})((1-\lambda)(\gamma_{i}\cdot \chi))&\text{if }j=i+1,\\
			\widetilde{\gamma_{i}}\otimes \tilde{\eta}(\sigma_{j})((1-\lambda)(\gamma_{i}\cdot \chi))&\text{otherwise},
		\end{cases}
	\end{align*}
	for $j=1,\ldots,n-1$.
	
	Let us look at the case $j=i$. Then we have 
	\begin{align*}
		\widetilde{\gamma_{0}}\otimes \tilde{\eta}(\sigma_{i})\cdot(\lambda(\gamma_{i}-1)\cdot \chi)&+	\widetilde{\gamma_{i+1}}\otimes  \tilde{\eta}(\sigma_{i})((1-\lambda)(\gamma_{i}\cdot \chi))\\
		&=\widetilde{\gamma_{0}}\otimes \lambda(\gamma_{i+1}-1)\cdot\tilde{\eta}(\sigma_{i})(\chi)+	\widetilde{\gamma_{i+1}}\otimes  (1-\lambda)\gamma_{i+1}\cdot \tilde{\eta}(\sigma_{i})(\chi)\\
		&=\widetilde{[\gamma_{0},\gamma_{i+1}]}\otimes \tilde{\eta}(\sigma_{i})(\chi).
	\end{align*}
	Also look at the case $j=i+1$. Similarly we have 
	\begin{align*}
		\widetilde{\gamma_{0}}\otimes \tilde{\eta}(\sigma_{j})(\lambda(\gamma_{i}-1)\cdot \chi)&+\widetilde{(\gamma_{i+1}^{-1}\cdot \gamma_{i}\cdot \gamma_{i+1})} \otimes \tilde{\eta}(\sigma_{i+1})((1-\lambda)(\gamma_{i}\cdot \chi))\\
		&=\cdots\\
		&=\widetilde{[\gamma_{0},\gamma_{i}]}\otimes \tilde{\eta}(\sigma_{i})(\gamma_{i}\cdot \chi)
		+\widetilde{[\gamma_{0},\gamma_{i+1}]}\otimes \tilde{\eta}(\sigma_{i})((1-\gamma_{i+1})\cdot \chi).
	\end{align*}
	Therefore if we write the associated group homomorphism with
	the left $k[F_{n}\rtimes_{\theta_{\mathrm{Artin}}}B_{n}]$-module $\mathcal{E}_{\lambda}(V)$
	by  
	$\rho_{\lambda}^{\mathrm{Euler}}\colon F_{n}\rtimes_{\theta_{\mathrm{Artin}}}B_{n}\rightarrow \mathrm{Aut}_{k}(\mathcal{E}_{\lambda}(V))$,
	we obtain the following matrix representations
	under the identification  $H_{1}(\mathbb{C}\backslash Q_{n}(a_{0});\iota_{a_{0}}^{*}(C_{\lambda}(k[F_{n}])))
	\cong k[F_{n}]^{\oplus n}$ by the isomorphism $\mathrm{Poch}$,
	\[
		\rho_{\lambda}^{\mathrm{Euler}}(\sigma_{i})=
		\begin{pmatrix}
			\overbrace{
				\begin{array}{cccc}
					\sigma_{i}&&&\\
					&\sigma_{i}&&\\
					&&\ddots&\\
					&&&\sigma_{i}
				\end{array}
			}^{i-1}&&&\\
			&0&\sigma_{i}x_{i}&\\
			&\sigma_{i}&\sigma_{i}-\sigma_{i}x_{i+1}&\\
			&&&
			\overbrace{
				\begin{array}{cccc}
					\sigma_{i}&&&\\
					&\sigma_{i}&&\\
					&&\ddots&\\
					&&&\sigma_{i}
				\end{array}
			}^{n-i-1}
		\end{pmatrix}
	\]
	for $i=1,\ldots,n-1$.
	Here we notice that  
	through the isomorphism $\phi\colon \iota_{a_{0}}^{*}(C_{\lambda}(k[F_{n}]))=k[F_{n}]\otimes_{k}K_{\lambda}\ni
	x\otimes \kappa\mapsto \kappa x\in k[F_{n}]$,
	we have 
	\begin{align*}
		\phi(\tilde{\eta}(b)(\chi))&=b\cdot\phi(\chi)&\text{for }b\in B_{n},\\
		\phi(\tilde{\eta}(x)(\chi))&=\phi(\chi)&\text{for }x\in F_{n},\\
		\phi(\gamma_{i}\cdot \chi)&=x_{i}\cdot \phi(\chi)&\text{for }x_{i}\in F_{n},\,i=1,\ldots,n,\\
		\phi(\gamma_{0}\cdot \chi)&=\lambda\cdot \phi(\chi).
	\end{align*}

	Since the above matrices are same as for the Long-Moody functor,
	it follows that 
	$\Phi\colon H_{1}(\mathbb{C}\backslash Q_{n}(a_{0});\iota_{a_{0}}(C_{\lambda}(k[F_{n}])))
	\rightarrow I_{F_{n},\,\lambda}$
	is a left $k[B_{n}]$-module isomorphism.
	
	Let us check that $\Phi$ is a left $k[F_{n}]$-isomorphism as well.
	Recalling that $x_{1}=\sigma_{0}^{2}$, we have 
	\begin{align*}
		&x_{1}\cdot (\widetilde{[\gamma_{0},\gamma_{1}]}\otimes v)\\
		&\quad=
		\sigma_{0}^{2}\cdot (\widetilde{\gamma_{0}}\otimes \lambda(\gamma_{1}-1)\cdot v+\widetilde{\gamma_{1}}\otimes (1-\lambda)(\gamma_{1}\cdot v) )\\
		&\quad=\bar{\eta}(\sigma_{0}^{2})_{*}(\widetilde{\gamma_{0}})\otimes \lambda\cdot \tilde{\eta}(\sigma_{0}^{2})((\gamma_{1}-1)\cdot v)+\bar{\eta}(\sigma_{0}^{2})_{*}(\widetilde{\gamma_{1}})\otimes (1-\lambda)\cdot\tilde{\eta}(\sigma_{0}^{2})(\gamma_{1}\cdot v) \\
		&\quad=\widetilde{(\gamma_{1}^{-1}\cdot \gamma_{0}\cdot \gamma_{1})}\otimes(\lambda(\gamma_{1}-1)v)
		+\widetilde{((\gamma_{0}\cdot\gamma_{1})^{-1}\cdot \gamma_{1}\cdot (\gamma_{0}\cdot\gamma_{1}))}
		\otimes  (1-\lambda)(\gamma_{1}\cdot v) \\
		&\quad=\cdots\\
		&\quad =\widetilde{\gamma_{0}}\otimes \lambda(\gamma_{1}-1)\cdot(\lambda\gamma_{1}\cdot v)+
		\widetilde{\gamma_{1}}\otimes (1-\lambda)\gamma_{1}\cdot (\lambda\gamma_{1}\cdot v)\\
		&\quad 	=\widetilde{[\gamma_{0},\gamma_{1}]}\otimes \lambda\gamma_{1}\cdot v,
	\end{align*}
	where we note that $\tilde{\eta}(\sigma_{0}^{2})(w)=\tilde{\eta}(x_{1})(w)=w$
	for any $w\in \bar{V}$.
	Similarly we also have
	\[
		x_{1}\cdot (\widetilde{[\gamma_{0},\gamma_{i}]}\otimes v)
		=\widetilde{[\gamma_{0},\gamma_{i}]}\otimes v+\widetilde{[\gamma_{0},\gamma_{1}]}\otimes (\gamma_{i}-1)\cdot v
	\]
	for $i=2,\ldots,n$. Therefore we obtain the matrix representation,
	\[
	\rho_{\lambda}^{\mathrm{Euler}}(x_{1})=
	\begin{pmatrix}
		\lambda x_{1}&(x_{2}-1)&\cdots&(x_{n}-1)\\
		&1&&\\
		&&\ddots&\\
		&&&1
	\end{pmatrix},
	\]
	and this matrix form tells us that 
	\begin{align*}
		\rho_{\lambda}^{\mathrm{Euler}}(x_{2})&=\rho_{\lambda}^{\mathrm{Euler}}(\sigma_{1}x_{1}\sigma_{1}^{-1})
		=\rho_{\lambda}^{\mathrm{Euler}}(\sigma_{1})\rho_{\lambda}^{\mathrm{Euler}}(x_{1})\rho_{\lambda}^{\mathrm{Euler}}(\sigma_{1}^{-1})\\
		&=
		\begin{pmatrix}
			1&&&&\\
			\lambda (x_{1} -1)&\lambda x_{2}&(x_{3}-1)&\cdots&(x_{n}-1)\\
			&&1&&\\
			&&&\ddots&\\
			&&&&1
		\end{pmatrix},
	\end{align*}
	and inductively
	\[
	\rho_{\lambda}^{\mathrm{Euler}}(x_{i})=
	\begin{pmatrix}
		1&&&&&&\\
		&\ddots&&&&&\\
		&&1&&&&\\
		\lambda (x_{1}-1)&\cdots&
			\lambda (x_{i-1}-1)
		&\lambda x_{i}&
		x_{i+1}-1&
		\cdots&x_{n}-1\\
		&&&&1&&\\
		&&&&&\ddots&\\
		&&&&&&1
	\end{pmatrix},
\]
for $i=3,\ldots,n$.
These matrices are exactly the same ones in Remark \ref{rem:DR} and 
therefore we obtain that $\Phi$ is a $k[F_{n}]$-module isomorphism.
\end{proof}

\begin{thm}
	Let us take $\lambda\in k^{\times}\backslash\{1\}$.
	Then the isomorphism  
	\[
	\Phi\colon H_{1}(\mathbb{C}\backslash Q_{n}(a_{0});\iota_{a_{0}}(C_{\lambda}(k[F_{n}])))
	\longrightarrow I_{F_{n},\,\lambda}
	\]
	in Proposition \ref{prop:aughom}
	induces an 
	isomorphism of functors 
	\[
		\widetilde{\Phi}\colon \mathcal{E}_{\lambda}\longrightarrow 
		\mathcal{LM}_{\lambda}.
	\]
\end{thm}
\begin{proof}
For each $V\in \mathscr{O}(\mathbf{Mod}_{k[F_{n}\rtimes_{\alpha}G]})$,
the isomorphism $\Phi$ and the K\"unneth formula in Proposition \ref{prop:kunneth}
give the desired isomorphism 
\begin{multline*}
	\mathcal{E}_{\lambda}(V)=H_{1}(\mathbb{C}\backslash Q_{n}(a_{0});
	\iota_{a_{0}}^{*}(C_{\lambda}(V)))\\\cong
	H_{1}(\mathbb{C}\backslash Q_{n}(a_{0});
	\iota_{a_{0}}^{*}(C_{\lambda}(k[F_{n}])))\otimes_{k[F_{n}]}V
	\cong I_{F_{n},\lambda}\otimes_{k[F_{n}]}V
	=\mathcal{LM}_{\lambda}(V). 	
\end{multline*}
\end{proof}
\section{Katz-Long-Moody functor}\label{sec:KLM}
As it is easily seen, 
$\mathcal{LM}_{\lambda}(V)$ or equivalently 
$\mathcal{E}_{\lambda}(V)$
contain some trivial submodules.
Hence Long and Moody, and also Katz
considered reduction procedures to eliminate 
these trivial submodules.
We shall explain these reduction procedures 
define a new endfunctor of $\mathbf{Mod}_{k[F_{n}\rtimes G]}$.

We first recall the reduced Long-Moody functor defined in \cite{Long}.
For a left $k[F_{n}\rtimes_{\alpha}G]$-module $V$, we consider the subspace of $\mathcal{LM}(V)$ as a vector space
defined by 
\[
	V^{[\infty]}:=\{(x_{1}x_{2}\cdots x_{n}-1)\otimes v\in I_{F_{n}}\otimes_{k[F_{n}]}V\mid v\in V\}.	
\]
Recalling that we have $\alpha(g)(x_{1}\cdots x_{n})=x_{1}\cdots x_{n}$ for any $g\in G$
since $\alpha$ is an Artin representation,
we see that $V^{[\infty]}$ is a subspace as a $k[G]$-module as well.
Then the \emph{reduced Long-Moody functor} 
$\mathcal{LM}^{\mathrm{red}}\colon \mathbf{Mod}_{k[F_{n}\rtimes_{\alpha}G]} \rightarrow \mathbf{Mod}_{k[G]}$ is 
defined by
\[
	\mathcal{LM}^{\mathrm{red}}(V):=\left(I_{F_{n}}\otimes_{k[F_{n}]}V\right)/V^{[\infty]}.
\]

However $V^{[\infty]}$ is not preserved by the $\lambda$-twisted action of $F_{n}$,
i.e., $V^{[\infty]}$ is not a subspace of $\mathcal{LM}_{\lambda}(V)=I_{F_{n},\lambda}\otimes_{k[F_{n}]}V$
as a left $k[F_{n}\rtimes_{\alpha}G]$-module, in general.
Thus to consider an analogue of the reduce Long-Moody functor for the twisted case,
we need to find some other subspaces.

Let us  consider the $F_{n}$-invariant subspace $\mathcal{LM}_{\lambda}(V)^{F_{n}}$
of $\mathcal{LM}_{\lambda}(V)$.
Obviously the subspace $\mathcal{LM}_{\lambda}(V)^{F_{n}}$ is a $k[F_{n}\rtimes_{\alpha}G]$-submodule of $\mathcal{LM}_{\lambda}(V)$.
Moreover this submodule has the following description when $\lambda\neq 1$.
\begin{prp}[Dettweiler-Reiter, Lemma 2.7 in \cite{DR}]\label{prop:infinity}
	Suppose that $\lambda\neq 1$. Then we have 
	\[
		\mathcal{LM}_{\lambda}(V)^{F_{n}}
		=\{(x_{1}x_{2}\cdots x_{n}-1)\otimes v\in I_{F_{n}}\otimes_{k[F_{n}]}V\mid 
		(\lambda x_{1}x_{2}\cdots x_{n})\cdot v=v\}.
	\]
\end{prp}
\begin{proof}
	Directly follows from Lemmas 2.4 and 2.7 in \cite{DR}.
\end{proof}
It follows from this proposition that the $F_{n}$-invariant subspace $\mathcal{LM}_{\lambda}(V)^{F_{n}}$ is 
a subspace of $V^{[\infty]}$ when $\lambda\neq 1$.

Further we consider the following subspaces 
\[
	\mathcal{LM}_{\lambda}(V)^{[i]}:=\left\{w\in \mathcal{LM}_{\lambda}(V)\,\middle|\, 
	x_{j}\cdot w=\begin{cases} \lambda w&\text{if }j=i\\ w&\text{otherwise}\end{cases}\right\}
\]
for $i=1,\ldots,n$.
Then we can see that the sum $\sum_{i=1}^{n}\mathcal{LM}_{\lambda}(V)^{[i]}$
becomes a $k[F_{n}\rtimes_{\alpha}G]$-submodule, as follows.
\begin{lmm}
	The sum $\sum_{i=1}^{n}\mathcal{LM}_{\lambda}(V)^{[i]}$
	is closed under the $G$-action on $\mathcal{LM}_{\lambda}(V)$.
\end{lmm}
\begin{proof}
	Let us take $w\in \mathcal{LM}_{\lambda}(V)^{[i]}$ and $g\in G$.
	Then since $\alpha(g)\in B_{n}$, there exists 
	an element $s_{g}\in \mathfrak{S}_{n}$ in the symmetric group of degree $n$
	such that 
	$\alpha(g^{-1})(x_{j})$ are  conjugate 
	to $x_{s_{g}(j)}$ in $F_{n}$ for all $j=1,\ldots,n$.
	Namely, for each $j=1,\ldots,n$, there exists $x^{[g,j]}\in F_{n}$
	such that $\alpha(g^{-1})(x_{j})=g^{-1}x_{j}g=x^{[g,j]}x_{s_{g}(j)}(x^{[g,j]})^{-1}$.
	Then we have
	\begin{align*}
		x_{j}(gw)&=g(g^{-1}x_{j}g)w\\
		&=g(x^{[g,j]}x_{s_{g}(j)}(x^{[g,j]})^{-1})w\\
		&=g(x_{s_{g}(j)}w)\\
		&=\begin{cases}
			\lambda gw&\text{if }s_{g}(j)=i\\
			gw&\text{otherwise}
		\end{cases}.
	\end{align*}
	The third equation follows form the fact that every
	element of $F_{n}$ acts on $\mathcal{LM}_{\lambda}(V)^{[i]}$
	as a scalar multiple.
	This shows that $gw\in \mathcal{LM}_{\lambda}(V)^{[s_{g}^{-1}(i)]}$
	and therefore  $\sum_{i=1}^{n}\mathcal{LM}_{\lambda}(V)^{[i]}$
	is closed under the $G$-action.
\end{proof}
We note that if $\lambda=1$, then $\mathcal{LM}_{\lambda}(V)^{[i]}=\sum_{i=1}^{n}\mathcal{LM}_{\lambda}(V)^{[i]}
=\mathcal{LM}_{\lambda}(V)^{F_{n}}$.
On the other hand, when $\lambda\neq 1$, 
the sum $\sum_{i=1}^{n}\mathcal{LM}_{\lambda}(V)^{[i]}$ becomes the 
direct sum $\bigoplus_{i=1}^{n}\mathcal{LM}_{\lambda}(V)^{[i]}$ and each direct summand
$\mathcal{LM}_{\lambda}(V)^{[i]}$ has the 
following description.
\begin{prp}\label{prop:finite}
	We have inclusions 
	\[
		\mathcal{LM}_{\lambda}(V)^{[i]}\supset \{(x_{i}-1)\otimes v\in I_{F_{n},\lambda}\otimes_{k[F_{n}]}V\mid v\in V^{x_{i}}\}
	\]
	for $i=1,\ldots,n$. Here we set $V^{x}:=\{v\in V\mid xv=v\}$ for $x\in F_{n}$.
	Especially when $\lambda\neq 1$, these inclusions are equations.
\end{prp}
\begin{proof}
	The inclusion $\supset$ is obvious.
	Thus assuming $\lambda\neq 1$, we show the converse inclusion.
	Let us take  $v\in \mathcal{LM}_{\lambda}(V)^{[i]}$.
	According to the identification $I_{F_{n},\lambda}\otimes_{k[F_{n}]}V\cong \bigoplus_{i=1}^{n}V$ by $(\ref{eq:indentfy})$,
	we can write $v=(v_{1},\ldots,v_{n})$, and 
	we shall show that $v_{i}\in V^{x_{i}}$ and $v_{j}=0$ for $j\neq i$.

	Since $v\in \mathcal{LM}_{\lambda}(V)^{[i]}$, we have 
	\begin{equation}\label{eq:eqmat}
		\sum_{l<j}\lambda(x_{l}-1)v_{l} +(\lambda x_{j}-1)v_{j} +\sum_{l>j}(x_{l}-1)v_{l}=
		\begin{cases}
			(\lambda-1)v_{i}&\text{if }j=i\\
			0&\text{otherwise}
		\end{cases}.
	\end{equation}
	From these equations, we obtain
	\begin{equation}\label{eq:result}
	\begin{aligned}
		v_{l}-x_{l+1}v_{l+1}&=0\qquad \text{for }l\neq i-1,i,\\
		v_{i-1}-(x_{i}-1)v_{i}&=0,\\
		x_{i+1}v_{i+1}&=0.
	\end{aligned}
	\end{equation}
	The first and third equations show $v_{l}=0$ for $l\ge i+1.$ 
	Then if $i=1$, we obtain $x_{1}v_{1}=v_{1}$ and $v_{j}=0$, $j\neq 0$ as desired.
	We assume that $i>1$.
	Let us look at the equations $l=1$ and $l=i$ in $(\ref{eq:eqmat})$ and then we obtain
	the equations
	\begin{align*}
		(\lambda x_{1}-1)v_{1}+(x_{2}-1)v_{2}+\cdots +(x_{i}-1)v_{i}&=0\\
		\lambda(x_{1}-1)v_{1}+\lambda(x_{2}-1)v_{2}+\cdots +\lambda(x_{i}-1)v_{i}&=0,
	\end{align*}
	which deduce $x_{1}v_{1}=0$, i.e., $v_{1}=0$. Thus the first equation in $(\ref{eq:result})$
	shows $v_{1}=\cdots=v_{i-1}=0$.
	Thus we obtain $x_{i}v_{i}=v_{i}$. 
\end{proof}

As an analogue of the middle convolution functor defined by Dettweiler-Reiter in 
\cite{DR} and also by V\"olklein in \cite{Vol},
we consider the following endfunctor of $\mathbf{Mod}_{k[F_{n}\rtimes_{\alpha}G]}$.
\begin{dfn}[Katz-Long-Moody functor]\normalfont
	Let us take $\lambda\in k^{\times}$.
	The functor $$\mathcal{KLM}_{\lambda}\colon \mathbf{Mod}_{k[F_{n}\rtimes_{\alpha}G]}\rightarrow \mathbf{Mod}_{k[F_{n}\rtimes_{\alpha}G]}$$
	defined by 
	\[
		\mathcal{KLM}_{\lambda}(V):=\mathcal{LM}_{\lambda}(V)/\left(\sum_{i=1}^{n}\mathcal{LM}_{\lambda}(V)^{[i]}
		+\mathcal{LM}_{\lambda}(V)^{F_{n}}\right)
	\]
	is called the \emph{reduced twisted Long-Moody functor} or \emph{Katz-Long-Moody functor} associated with $\lambda$.
	Sometimes we call this functor the \emph{KLM-functor} in short.

	We write the quotient map by 
	\[
		\theta_{V}^{\lambda}\colon \mathcal{LM}_{\lambda}(V)\longrightarrow \mathcal{KLM}_{\lambda}(V),
	\]
	which gives a natural transformation between the functors $\mathcal{LM}_{\lambda}$ and $\mathcal{KLM}_{\lambda}$.
\end{dfn}
Especially for $\lambda=1$, the KLM-functor arises also from the multiplication map as follows.
Let us suppose $\lambda=1$. Then the multiplication map $\nabla\colon I_{F_{n},\,1}\otimes_{k[F_{n}]}V\rightarrow V$
becomes a $k[F_{n}\rtimes_{\alpha}G]$-module homomorphism.
Therefore the sequence
\[
	0\rightarrow \mathrm{Ker\,}\nabla\rightarrow I_{F_{n},\,1}\otimes_{k[F_{n}]}V
	\xrightarrow[]{\nabla}V\rightarrow \mathrm{Coker\,}\nabla\rightarrow 0
\]
is an exact sequence of $k[F_{n}\rtimes_{\alpha}G]$-modules.

On the other hand, we have $\sum_{i=1}^{n}\mathcal{LM}_{1}(V)^{[i]}
+\mathcal{LM}_{1}(V)^{F_{n}}=	\mathcal{LM}_{1}(V)^{F_{n}}$ when $\lambda=1$.
Also Remark \ref{rem:DR} shows that 
$v\in \mathcal{LM}_{1}(V)^{F_{n}}$ if and only if 
$\sum_{i=1}^{n}(x_{i}-1)v_{i}=0$ where we write $v=(v_{1},\ldots, v_{n})\in V^{\oplus n}$
through the isomorphism $(\ref{eq:indentfy})$.
Therefore we have 
\[
	\mathrm{Ker\,}\theta_{V}^{1}=\mathcal{LM}_{1}(V)^{F_{n}}=
	\left\{(v_{1},\ldots,v_{n})\in V^{\oplus n}\,\middle|\, \sum_{i=1}^{n}(x_{i}-1)v_{i}=0\right\}
	=\mathrm{Ker\,}\nabla
\]
by Proposition \ref{prop:multiplication}.
Thus we obtain the following.
\begin{prp}\label{prop:1middle}
	Let  $V$ be a left $k[F_{n}\rtimes_{\alpha}G]$-module. 
	Let us  consider the multiplication 
	map $\nabla\colon I_{F_{n},\,1}\otimes_{k[F_{n}]}V\rightarrow V$
	as a $k[F_{n}\rtimes_{\alpha}G]$-module homomorphism.
	Then we have 
	\[
		\mathcal{KLM}_{1}(V)\cong \mathrm{Coim}\theta_{V}^{1}=\mathrm{Coim}\nabla\cong \mathrm{Im\,}\nabla\subset V.	
	\]
	In particular when $\mathrm{Coker\,}\nabla=\{0\}$, we have 
	\[
		\mathcal{KLM}_{1}(V)\cong V.	
	\]
\end{prp}

The KLM-functor is not exact. However the following holds.
\begin{prp}[cf. Lemma 2.8 in \cite{DR} and Section 2.1 in \cite{Vol}]
	Let us consider a left $k[F_{n}\rtimes_{\alpha}G]$-module homomorphism 
	$\phi\colon V\rightarrow W$. Then the induced homomorphism 
	$\mathcal{KLM}_{\lambda}(\phi)\colon \mathcal{KLM}_{\lambda}(V)\rightarrow \mathcal{KLM}_{\lambda}(W)$
	is injective (resp. surjective) if 
	$\phi$ is injective (resp. surjective).
\end{prp}
\begin{proof}
	Let us note that we have the following commutative 
	diagram,
	\[
	\begin{tikzcd}
		0\arrow[r]&\mathrm{Ker\,}\theta_{V}^{\lambda}\arrow[r]
		\arrow[d,"\mathcal{LM}_{\lambda}(\phi)|_{\mathrm{Ker\,}\theta_{V}^{\lambda}}"]
		&\mathcal{LM}_{\lambda}(V) \arrow[r]
		\arrow[d,"\mathcal{LM}_{\lambda}(\phi)"]
		&\mathcal{KLM}_{\lambda}(V) \arrow[d,"\mathcal{KLM}_{\lambda}(\phi)"]\arrow[r]&0\\
		0\arrow[r]& \mathrm{Ker\,}\theta_{W}^{\lambda}\arrow[r]&\mathcal{LM}_{\lambda}(W) \arrow[r]
		&\mathcal{KLM}_{\lambda}(W) \arrow[r]&0
	\end{tikzcd},
	\]
	where the horizontal sequences are exact.
	
	Suppose that $\phi$ is injective.
	Since $\mathcal{LM}_{\lambda}$ is an exact functor,
	$\mathcal{LM}_{\lambda}(\phi)$
	is injective.
	Moreover we have
	\[
		\mathrm{Im\,}\mathcal{LM}_{\lambda}(\phi)|_{\mathrm{Ker\,}\theta_{V}^{\lambda}}
		=\mathrm{Im\,}\mathcal{LM}_{\lambda}(\phi)
		\cap 
		\mathrm{Ker\,}\theta_{W}^{\lambda}
	\]
	by Lemma 2.8 in \cite{DR}. Thus a diagram chasing shows that $\mathcal{KLM}_{\lambda}(\phi)$ is injective.

	Next we suppose that $\phi$ is surjective. Then  $\mathcal{LM}_{\lambda}(\phi)$ is surjective
	from the exactness of $\mathcal{LM}_{\lambda}$. Therefore the above diagram 
	directly shows $\mathcal{KLM}_{\lambda}(\phi)$ is surjective.
\end{proof}
\section{Middle convolution and Katz-Long-Moody functor}\label{sec:MCKLM}
In the book \cite{Katz},
Katz introduced the notion of middle convolution functor 
by means of the middle extension operation for perverse sheaves,
and this middle convolution functor is 
acknowledged widely as a sophisticated generalization of 
the Euler transform.
In this section we shall give an interpretation of Katz' middle 
convolution by the twisted homology theory 
and show that this homological middle convolution 
is equivalent to the KLM-functor.  
\subsection{Homological middle convolution}
We use the same notations as in Section \ref{sec:ET}.
Let us write the Riemann sphere by $\mathbb{P}^{1}:=\mathbb{C}\sqcup\{\infty\}$.
According to the identity 
$\mathbb{C}\backslash Q_{n}(a_{0})=\mathbb{P}^{1}\backslash Q_{n}^{\infty}(a_{0})$
by setting $Q_{n}^{\infty}:=Q_{n}\sqcup \{\infty\}$ and $Q_{n}^{\infty}(a_{0}):=Q_{n}^{\infty}
\sqcup \{a_{0}\}$,
we rewrite $H_{1}(\mathbb{C}\backslash Q_{n}(a_{0});\iota_{a_{0}}^{*}(C_{\lambda}(V)))$ as
the homology group on $\mathbb{P}^{1}\backslash Q_{n}^{\infty}(a_{0})$,
$H_{1}(\mathbb{P}^{1}\backslash Q_{n}^{\infty}(a_{0});\iota_{a_{0}}^{*}(C_{\lambda}(V))).$

Let us also consider the Borel-Moore homology groups 
$H_{*}^{\mathrm{BM}}(\mathbb{P}^{1}\backslash Q_{n}^{\infty}(a_{0});\iota_{a_{0}}^{*}(C_{\lambda}(V)))$
which are homology groups of the chain complex defined by the projective limit of
relative chains 
\[
	C_{*}^{\mathrm{BM}}(\mathbb{P}^{1}\backslash Q_{n}^{\infty}(a_{0})
	;\iota_{a_{0}}^{*}(C_{\lambda}(V)))
	:=\varprojlim
	C_{*}(\mathbb{P}^{1}\backslash Q_{n}^{\infty}(a_{0}),
	(\mathbb{P}^{1}\backslash Q_{n}^{\infty}(a_{0}))\backslash K
	;\iota_{a_{0}}^{*}(C_{\lambda}(V)))
\]
as $K$ ranges over the directed set of compact subsets of $\mathbb{P}^{1}\backslash Q_{n}^{\infty}(a_{0})$.
Then the canonical projection maps
\[
C_{*}(\mathbb{P}^{1}\backslash Q_{n}^{\infty}(a_{0})
;\iota_{a_{0}}^{*}(C_{\lambda}(V)))
\rightarrow 
C_{*}(\mathbb{P}^{1}\backslash Q_{n}^{\infty}(a_{0}),
	(\mathbb{P}^{1}\backslash Q_{n}^{\infty}(a_{0}))\backslash K
	;\iota_{a_{0}}^{*}(C_{\lambda}(V)))
\]	
induce the natural maps of homology groups 
\[
	H_{*}(\mathbb{P}^{1}\backslash Q_{n}^{\infty}(a_{0});\iota_{a_{0}}^{*}(C_{\lambda}(V)))
	\rightarrow
	H_{*}^{\mathrm{BM}}(\mathbb{P}^{1}\backslash Q_{n}^{\infty}(a_{0});\iota_{a_{0}}^{*}(C_{\lambda}(V))).
\]

Let us recall that the group homomorphism 
$\eta\colon F_{n}\rtimes_{\alpha}G\rightarrow \mathrm{Aut}_{\mathbf{Top}}(\mathbb{P}^{1}\backslash Q_{n}^{\infty}(a_{0}))$
given in Section \ref{sec:pi1action}.
Then as well as in Section \ref{sec:faction}, this $\eta$
induces the  
maps 
\begin{multline*}
	C_{*}(\mathbb{P}^{1}\backslash Q_{n}^{\infty}(a_{0}),
	(\mathbb{P}^{1}\backslash Q_{n}^{\infty}(a_{0}))\backslash K
	;\iota_{a_{0}}^{*}(C_{\lambda}(V))) \\
	\longrightarrow C_{*}(\mathbb{P}^{1}\backslash Q_{n}^{\infty}(a_{0}),
	(\mathbb{P}^{1}\backslash Q_{n}^{\infty}(a_{0}))\backslash \eta(h)(K)
	;\iota_{a_{0}}^{*}(C_{\lambda}(V))) 
\end{multline*}
for $h\in F_{n}\rtimes_{\alpha}G$.
Since $\eta(h)(K)$ are still compact in the above maps,
these maps defines
the action of $F_{n}\rtimes_{\alpha}G$
on this projective system.
Therefore we can regard the projective limit  
$H_{1}^{\mathrm{BM}}(\mathbb{P}^{1}\backslash Q_{n}^{\infty}(a_{0});\iota_{a_{0}}^{*}(C_{\lambda}(V)))$
as a $k[F_{n}\rtimes_{\alpha}G]$-module and 
the canonical map $H_{1}(\mathbb{P}^{1}\backslash Q_{n}^{\infty}(a_{0});\iota_{a_{0}}^{*}(C_{\lambda}(V)))
\longrightarrow H_{1}^{\mathrm{BM}}(\mathbb{P}^{1}\backslash Q_{n}^{\infty}(a_{0});\iota_{a_{0}}^{*}(C_{\lambda}(V)))$
as a $k[F_{n}\rtimes_{\alpha}G]$-module homomorphism.
\begin{dfn}[Homological middle convolution functor]\label{dfn:MC}\normalfont
	For $\lambda\in k^{\times}\backslash\{1\}$, we define the functor
	$\mathcal{MC}_{\lambda}\colon 
\mathbf{Mod}_{k[F_{n}\rtimes_{\alpha}G]}\rightarrow \mathbf{Mod}_{k[F_{n}\rtimes_{\alpha}G]}$
by 
\[
	\mathcal{MC}_{\lambda}(V):=\mathrm{Im\,}(H_{1}(\mathbb{P}^{1}\backslash Q_{n}^{\infty}(a_{0});\iota_{a_{0}}^{*}(C_{\lambda}(V)))
	\longrightarrow H_{1}^{\mathrm{BM}}(\mathbb{P}^{1}\backslash Q_{n}^{\infty}(a_{0});\iota_{a_{0}}^{*}(C_{\lambda}(V))))
\]
for $V\in \mathscr{O}(\mathbf{Mod}_{k[F_{n}\rtimes_{\alpha}G]})$. Also for $\phi\in \mathrm{Hom}_{k[F_{n}\rtimes_{\alpha}G]}(V,W)$,
we define \\$\mathcal{MC}_{\lambda}(\phi)\in \mathrm{Hom}_{k[F_{n}\rtimes_{\alpha}G]}(\mathcal{MC}_{\lambda}(V),\mathcal{MC}_{\lambda}(W))$ by 
the homomorphism induced by the commutative diagram
\[
	\begin{tikzcd}
		H_{1}(\mathbb{P}^{1}\backslash Q_{n}^{\infty}(a_{0});\iota_{a_{0}}^{*}(C_{\lambda}(V)))\arrow[r]
		\arrow[d,"H_{1}(\phi)"]&
	H^{\mathrm{BM}}_{1}(\mathbb{P}^{1}\backslash Q_{n}^{\infty}(a_{0});\iota_{a_{0}}^{*}(C_{\lambda}(V)))
	\arrow[d,"H^{\mathrm{BM}}_{1}(\phi)"]\\
	H_{1}(\mathbb{P}^{1}\backslash Q_{n}^{\infty}(a_{0});\iota_{a_{0}}^{*}(C_{\lambda}(W)))\arrow[r]
	&H_{1}^{\mathrm{BM}}(\mathbb{P}^{1}\backslash Q_{n}^{\infty}(a_{0});\iota_{a_{0}}^{*}(C_{\lambda}(W)))
	\end{tikzcd},	
\]
where the horizontal maps are the canonical maps and vertical ones are induced map by $\phi$. 
We call this functor the \emph{homological middle convolution}.
\end{dfn}
\begin{rem}\normalfont
	The integral representation 
	$$F(\alpha,\beta,\gamma;z)=\frac{\Gamma(\gamma)}{\Gamma(\alpha)\Gamma(\gamma-\alpha)}\int_{1}^{\infty}
	t^{\beta-\gamma}(t-1)^{\gamma-\alpha-1}(t-z)^{-\beta}\,dt$$
	of the Gauss hypergeometric function 
	is valid only for $\mathrm{Re\,}(\gamma-\alpha),\,\mathrm{Re\,}\alpha>0$
	which assures the convergence of the integral.
	However, the hypergeometric series $$\displaystyle F(\alpha,\beta,\gamma;z)
	=\sum_{n=0}^{\infty}\frac{\alpha(\alpha+1)\cdots (\alpha+(n-1))\beta(\beta+1)\cdots (\beta+(n-1))}
	{\gamma(\gamma+1)\cdots (\gamma+(n-1)) n!}z^{n}$$ 
	is defined for $\alpha,\beta\in \mathbb{C}$ and $\gamma\in \mathbb{C}\backslash \mathbb{Z}_{\le 0}$.
	Therefore the integral in the right hand side in the above equation
	can be analytically continued and it is well-known that 
	this analytic continuation is given by the formula
	\begin{multline}\label{eq:poch}
		\int_{[\gamma_{1},\gamma_{\infty}]}
	t^{\beta-\gamma}(t-1)^{\gamma-\alpha-1}(t-z)^{-\beta}\,dt\\
	=(1-e^{2\pi i\alpha})(1-e^{2\pi i (\gamma-\alpha)})
	\int_{1}^{\infty}
	t^{\beta-\gamma}(t-1)^{\gamma-\alpha-1}(t-z)^{-\beta}\,dt
	\end{multline}
	where $[\gamma_{1},\gamma_{\infty}]$ is the Pochhammer 
	contour consisting of 
	closed paths $\gamma_{1}$ encircling $1$ clockwise and 
	$\gamma_{\infty}$ encircling $\infty$ as well.

	As we saw previously,
	the homology group 
	$H_{1}(\mathbb{P}^{1}\backslash Q_{n}^{\infty}(a_{0});\iota_{a_{0}}^{*}(C_{\lambda}(V)))$
	can be seen as an analogue of 
	integrations along Pochhammer contours.
	On the other hand, 
	as it is well-known,
	the Borel-Moore homology groups
	are isomorphic to 
	the homology groups of locally finite chains 
	which correspond to usual path integrals connecting 
	points in $Q_{n}^{\infty}(a_{0})$
	under our setting.
	Namely, we can roughly say that 
	the left hand side in the equation $(\ref{eq:poch})$
	corresponds to $H_{1}(\mathbb{P}^{1}\backslash Q_{n}^{\infty}(a_{0});\iota_{a_{0}}^{*}(C_{\lambda}(V)))$, 
	the domain of the map defining the homological middle convolution,
	and the right hand side corresponds to $H_{1}^{\mathrm{BM}}(\mathbb{P}^{1}\backslash Q_{n}^{\infty}(a_{0});\iota_{a_{0}}^{*}(C_{\lambda}(V)))$,
	the target of the map for the middle convolution.
	Therefore 
	our definition of the middle convolution 
	gives a homological interpretation of the classical formula $(\ref{eq:poch})$.
\end{rem}
\subsection{Euler transform as Borel-Moore homology}
Let us write a punctured 
disk of radius $r$ and centered at $a$
by $D^{*}_{<r}(a):=\{z\in \mathbb{C}\mid 0<|z-a|<r\}$
if $a\in \mathbb{C}$, and by 
$D^{*}_{<r}(\infty):=\{z\in \mathbb{C}\mid 1/r < |z|\}$
if $a=\infty$.
Let us take $\varepsilon>0$ and 
define 
\[
	D^{*}_{<\varepsilon}(Q_{n}^{\infty}(a_{0})):=\bigsqcup_{i=0}^{n}D^{*}_{<\varepsilon}(a_{i})
	\sqcup D^{*}_{<\varepsilon}(\infty).
\]
Also define 
$K_{m}:=(\mathbb{P}^{1}\backslash Q_{n}^{\infty}(a_{0}))\backslash D^{*}_{<1/m}(Q_{n}^{\infty}(a_{0}))$
for $m\in \mathbb{Z}_{\ge 1}$ which 
are compact subsets of $\mathbb{P}^{1}\backslash Q_{n}^{\infty}(a_{0})$.
Then we obtain the increasing sequence of compact subsets 
$K_{1}\subset K_{2}\subset \cdots$ whose union is $\mathbb{P}^{1}\backslash Q_{n}^{\infty}(a_{0})$
and the projective limit 
\[
	\varprojlim_{K_{m}}H_{*}(\mathbb{P}^{1}\backslash Q_{n}^{\infty}(a_{0}),
	(\mathbb{P}^{1}\backslash Q_{n}^{\infty}(a_{0}))\backslash K_{m};\iota_{a_{0}}^{*}(C_{\lambda}(V))).
\]
There is the natural map 
\[
	H^{\mathrm{BM}}_{*}(\mathbb{P}^{1}\backslash Q_{n}^{\infty}(a_{0});\iota_{a_{0}}^{*}(C_{\lambda}(V)))
\longrightarrow 	
\varprojlim_{K_{m}}H_{*}(\mathbb{P}^{1}\backslash Q_{n}^{\infty}(a_{0}),
	(\mathbb{P}^{1}\backslash Q_{n}^{\infty}(a_{0}))\backslash K_{m};\iota_{a_{0}}^{*}(C_{\lambda}(V)))
\]
by the universality of the projective limit.
Since the projective limit functor is not compatible with the 
homology functor, this map may not be isomorphism in general setting. 
However we can show the following.
\begin{prp}
The natural map 
\[
	H^{\mathrm{BM}}_{*}(\mathbb{P}^{1}\backslash Q_{n}^{\infty}(a_{0});\iota_{a_{0}}^{*}(C_{\lambda}(V)))
\longrightarrow 	
\varprojlim_{K_{m}}H_{*}(\mathbb{P}^{1}\backslash Q_{n}^{\infty}(a_{0}),
	(\mathbb{P}^{1}\backslash Q_{n}^{\infty}(a_{0}))\backslash K_{m};\iota_{a_{0}}^{*}(C_{\lambda}(V)))
\]
is an isomorphism.
\end{prp}
\begin{proof}
	We omit the coefficients of homology groups for simplicity.
	Let us take  $0<\varepsilon <\varepsilon'<1/2$ and 
	recall that 
	$D^{*}_{<\varepsilon'}(Q_{n}^{\infty}(a_{0}))$ is a 
	deformation retract of $D^{*}_{<\varepsilon}(Q_{n}^{\infty}(a_{0}))$.
	Then
	the map 
	\begin{equation*}
	H_{*}(\mathbb{P}^{1}\backslash Q_{n}^{\infty}(a_{0}),
	(\mathbb{P}^{1}\backslash Q_{n}^{\infty}(a_{0}))\backslash K_{m})\\
	\rightarrow H_{*}(\mathbb{P}^{1}\backslash Q_{n}^{\infty}(a_{0}),
	(\mathbb{P}^{1}\backslash Q_{n}^{\infty}(a_{0}))\backslash K_{m'})
	\end{equation*}
	induced by 
	the map $(\mathbb{P}^{1}\backslash Q_{n}^{\infty}(a_{0}), (\mathbb{P}^{1}\backslash Q_{n}^{\infty}(a_{0}))\backslash K_{m})
	\rightarrow (\mathbb{P}^{1}\backslash Q_{n}^{\infty}(a_{0}), (\mathbb{P}^{1}\backslash Q_{n}^{\infty}(a_{0}))\backslash K_{m'})$
	is an isomorphism 
	for every $2<m'\le m\in \mathbb{Z}_{\ge 1}$.
	Therefore we obtain
	\[
		\varprojlim_{K_{m}}\!{}^{1}H_{*}(\mathbb{P}^{1}\backslash Q_{n}^{\infty}(a_{0}),
	(\mathbb{P}^{1}\backslash Q_{n}^{\infty}(a_{0}))\backslash K_{m};\iota_{a_{0}}^{*}(C_{\lambda}(V)))
	=\{0\},
	\]
	where $\varprojlim^{1}$ is the first derived functor 
	$R^{1}\varprojlim$ of the limit functor $\varprojlim$.
	Then the result follows from Theorem 7.3 in \cite{Spa}.
\end{proof}
\begin{crl}\label{crl:projisom}
	For all $m>2$,
	projection maps 
	\[
	H^{\mathrm{BM}}_{*}(\mathbb{P}^{1}\backslash Q_{n}^{\infty}(a_{0});\iota_{a_{0}}^{*}(C_{\lambda}(V)))
\longrightarrow 	
	H_{*}(\mathbb{P}^{1}\backslash Q_{n}^{\infty}(a_{0}),
	(\mathbb{P}^{1}\backslash Q_{n}^{\infty}(a_{0}))\backslash K_{m};\iota_{a_{0}}^{*}(C_{\lambda}(V)))
	\]
	are isomorphisms.
\end{crl}
\begin{proof}
	As we saw in the proof of the above proposition,
	the map 
	\begin{equation*}
		H_{*}(\mathbb{P}^{1}\backslash Q_{n}^{\infty}(a_{0}),
		(\mathbb{P}^{1}\backslash Q_{n}^{\infty}(a_{0}))\backslash K_{m})\\
		\rightarrow H_{*}(\mathbb{P}^{1}\backslash Q_{n}^{\infty}(a_{0}),
		(\mathbb{P}^{1}\backslash Q_{n}^{\infty}(a_{0}))\backslash K_{m'})
	\end{equation*}
	is an isomorphism 
	for every $2<m'\le m$.
	This shows that 
	projection maps 
	\begin{multline*}
		\varprojlim_{K_{m}}H_{*}(\mathbb{P}^{1}\backslash Q_{n}^{\infty}(a_{0}),
	(\mathbb{P}^{1}\backslash Q_{n}^{\infty}(a_{0}))\backslash K_{m};\iota_{a_{0}}^{*}(C_{\lambda}(V)))\\
	\rightarrow 
	H_{*}(\mathbb{P}^{1}\backslash Q_{n}^{\infty}(a_{0}),
	(\mathbb{P}^{1}\backslash Q_{n}^{\infty}(a_{0}))\backslash K_{m'};\iota_{a_{0}}^{*}(C_{\lambda}(V)))
	\end{multline*}
	are isomorphisms for $m'>2$.
	Then the above proposition shows the result.
\end{proof}
\subsection{Equivalence between homological middle convolution and Katz-Long-Moody functor}
Now we recall that 
\begin{multline*}
	H_{*}(\mathbb{P}^{1}\backslash Q_{n}^{\infty}(a_{0}),
	(\mathbb{P}^{1}\backslash Q_{n}^{\infty}(a_{0}))\backslash K_{m};\iota_{a_{0}}^{*}(C_{\lambda}(V)))\\
	=
	H_{*}(\mathbb{P}^{1}\backslash Q_{n}^{\infty}(a_{0}),D^{*}_{<1/m}(Q_{n}^{\infty}(a_{0}))
	;\iota_{a_{0}}^{*}(C_{\lambda}(V))).
\end{multline*}
Then the kernel of the map 
\[
	H_{1}(\mathbb{P}^{1}\backslash Q_{n}^{\infty}(a_{0});\iota_{a_{0}}^{*}(C_{\lambda}(V)))
	\longrightarrow H_{1}^{\mathrm{BM}}(\mathbb{P}^{1}\backslash Q_{n}^{\infty}(a_{0});\iota_{a_{0}}^{*}(C_{\lambda}(V)))	
\]
in Definition \ref{dfn:MC} is computable by the long exact sequence 
of the relative homology $H_{*}(\mathbb{P}^{1}\backslash Q_{n}^{\infty}(a_{0}),D^{*}_{<1/m}(Q_{n}^{\infty}(a_{0}))
;\iota_{a_{0}}^{*}(C_{\lambda}(V)))$ as follows.
\begin{prp}\label{cor:mc}
	Let us take $m\in \mathbb{Z}_{>2}$.
	Then we have the equation
	\begin{multline*}
		\mathrm{Ker\,}(H_{1}(\mathbb{P}^{1}\backslash Q_{n}^{\infty}(a_{0});\iota_{a_{0}}^{*}(C_{\lambda}(V)))
		\longrightarrow H_{1}^{\mathrm{BM}}(\mathbb{P}^{1}\backslash Q_{n}^{\infty}(a_{0});\iota_{a_{0}}^{*}(C_{\lambda}(V)))	
	)=\\
	\mathrm{Im\,}(H_{1}(D^{*}_{<1/m}(Q_{n}^{\infty}(a_{0}));\iota_{a_{0}}^{*}(C_{\lambda}(V))|_{D^{*}_{<1/m}(Q_{n}^{\infty}(a_{0}))})
	\rightarrow H_{1}(\mathbb{P}^{1}\backslash Q_{n}^{\infty}(a_{0});\iota_{a_{0}}^{*}(C_{\lambda}(V)))).
	\end{multline*}
\end{prp}
\begin{proof}
	Let us recall the long exact sequence of the relative homology groups 
	\begin{multline*} 
		\cdots \rightarrow H_{1}(D^{*}_{<1/m}(Q_{n}^{\infty}(a_{0}));\iota_{a_{0}}^{*}(C_{\lambda}(V))|_{D^{*}_{<1/m}(Q_{n}^{\infty}(a_{0}))})	\rightarrow
	H_{1}(\mathbb{P}^{1}\backslash Q_{n}^{\infty}(a_{0});\iota_{a_{0}}^{*}(C_{\lambda}(V)))\\
	\rightarrow H_{1}(\mathbb{P}^{1}\backslash Q_{n}^{\infty}(a_{0}),D^{*}_{<1/m}(Q_{n}^{\infty}(a_{0}));\iota_{a_{0}}^{*}(C_{\lambda}(V)))
	\rightarrow \cdots.
	\end{multline*}
	We also recall the  commutative diagram 
	\[
		\begin{tikzcd}
			H_{1}(\mathbb{P}^{1}\backslash Q_{n}^{\infty}(a_{0});\iota_{a_{0}}^{*}(C_{\lambda}(V)))
			\arrow[r]\arrow[rd]
			&H_{1}^{\mathrm{BM}}(\mathbb{P}^{1}\backslash Q_{n}^{\infty}(a_{0});\iota_{a_{0}}^{*}(C_{\lambda}(V)))
			\arrow[d]
			\\
			&H_{1}(\mathbb{P}^{1}\backslash Q_{n}^{\infty}(a_{0}),D^{*}_{<1/m}(Q_{n}^{\infty}(a_{0}));\iota_{a_{0}}^{*}(C_{\lambda}(V)))
		\end{tikzcd}
	\]
	whose vertical map is isomorphism by Corollary \ref{crl:projisom}.
	Then we obtain the result.
\end{proof}
\begin{thm}
	Let us take $\lambda\in k^{\times}\backslash\{1\}$.
	The isomorphism $\widetilde{\Phi}\colon \mathcal{E}_{\lambda}\rightarrow \mathcal{LM}_{\lambda}$
	induces an isomorphism 
	\[
		\widehat{\Phi}\colon \mathcal{MC}_{\lambda}\rightarrow \mathcal{KLM}_{\lambda}.
	\]
\end{thm}
\begin{proof}
	Let us take a left $k[F_{n}\rtimes_{\alpha}G]$-module $V$
	and a fixed $m\in \mathbb{Z}_{>2}$.
	Proposition \ref{cor:mc} tells us that 
	the kernel of the projection map $\Theta_{V}^{\lambda}\colon \mathcal{E}_{\lambda}(V)\rightarrow \mathcal{MC}_{\lambda}(V)$
	equals to 
	\[
		\mathrm{Im\,}(H_{1}(\mathbb{P}^{1}\backslash Q_{n}^{\infty}(a_{0}),D^{*}_{<1/m}(Q_{n}^{\infty}(a_{0}));\iota_{a_{0}}^{*}(C_{\lambda}(V)))
		\rightarrow H_{1}(\mathbb{P}^{1}\backslash Q_{n}^{\infty}(a_{0});\iota_{a_{0}}^{*}(C_{\lambda}(V)))).
	\]
	Now we claim that 
	\begin{multline*}
		\mathrm{Im\,}(H_{1}(D^{*}_{<1/m}(a_{i});\iota_{a_{0}}^{*}(C_{\lambda}(V))|_{D^{*}_{<1/m}(a_{i})})
		\rightarrow H_{1}(\mathbb{P}^{1}\backslash Q_{n}^{\infty}(a_{0});\iota_{a_{0}}^{*}(C_{\lambda}(V))))\\
		=\{\widetilde{[\gamma_{0},\gamma_{i}]}\otimes v\in H_{1}(\mathbb{P}^{1}\backslash Q_{n}^{\infty}(a_{0});\iota_{a_{0}}^{*}(C_{\lambda}(V)))\mid 
		v\in (\iota_{a_{0}}^{*}(C_{\lambda}(V)))^{\gamma_{i}}\}
	\end{multline*}
	for $i=1,\ldots,n$, and 
	\begin{multline*}
		\mathrm{Im\,}(H_{1}(D^{*}_{<1/m}(\infty);\iota_{a_{0}}^{*}(C_{\lambda}(V))|_{D^{*}_{<\varepsilon}(\infty)})
		\rightarrow H_{1}(\mathbb{P}^{1}\backslash Q_{n}^{\infty}(a_{0});\iota_{a_{0}}^{*}(C_{\lambda}(V))))\\
		=\{\widetilde{[\gamma_{0},\gamma_{\infty}]}\otimes v\in H_{1}(\mathbb{P}^{1}\backslash Q_{n}^{\infty}(a_{0});\iota_{a_{0}}^{*}(C_{\lambda}(V)))\mid 
		v\in (\iota_{a_{0}}^{*}(C_{\lambda}(V)))^{\gamma_{\infty}}\}
	\end{multline*}
	where $\gamma_{\infty}$ is the homotopy class of the path homotopic to 
	$(\gamma_{0}\cdots \gamma_{n})^{-1}$.

	Let us suppose the claim is true. 
	Since $D^{*}_{<1/m}(Q_{n}^{\infty}(a_{0}))$ is the disjoint union of 
	$D^{*}_{<1/m}(a_{i})$, $i=0,\ldots,n$, and $D^{*}_{<1/m}(\infty)$,
	we have 
	\begin{multline*}
	H_{1}(D^{*}_{<1/m}(Q_{n}^{\infty}(a_{0}));\iota_{a_{0}}^{*}(C_{\lambda}(V))|_{D^{*}_{<1/m}(Q_{n}^{\infty}(a_{0}))})\\=
	\bigoplus_{i=0}^{n}
	H_{1}(D^{*}_{<1/m}(a_{i});\iota_{a_{0}}^{*}(C_{\lambda}(V))|_{D^{*}_{<1/m}(a_{i})})
	\oplus H_{1}(D^{*}_{<1/m}(\infty);\iota_{a_{0}}^{*}(C_{\lambda}(V))|_{D^{*}_{<1/m}(\infty)}).
	\end{multline*}
	Let us note that $H_{1}(D^{*}_{<1/m}(a_{0});\iota_{a_{0}}^{*}(C_{\lambda}(V))|_{D^{*}_{<1/m}(a_{0})})=\{0\}$
	because of $\lambda\neq 1$.
	Also note that the isomorphism 
	$\iota_{a_{0}}^{*}(C_{\lambda}(V))=V\otimes_{k}K_{\lambda}\cong V$
	sends $ (\iota_{a_{0}}^{*}(C_{\lambda}(V)))^{\gamma_{i}}$ to 
	$V^{x_{i}}$ for $i=1,\ldots,n$,
	and $ (\iota_{a_{0}}^{*}(C_{\lambda}(V)))^{\gamma_{\infty}}$
	to $V^{\lambda x_{1}\cdots x_{n}}$
	since $\gamma_{0}$ acts on $\iota_{a_{0}}^{*}(C_{\lambda}(V))$
	as the multiplication of the scalar $\lambda$.
	Thus 
	$\widetilde{\Phi}$ sends 
	$\mathrm{Im\,}(H_{1}(D^{*}_{<1/m}(Q_{n}^{\infty}(a_{0}));\iota_{a_{0}}^{*}(C_{\lambda}(V))|_{D^{*}_{<1/m}(Q_{n}^{\infty}(a_{0}))})
	\rightarrow H_{1}(\mathbb{P}^{1}\backslash Q_{n}^{\infty}(a_{0});\iota_{a_{0}}^{*}(C_{\lambda}(V))))$
	to $\bigoplus_{i=1}^{n}\mathcal{LM}_{\lambda}(V)^{[i]}\oplus \mathcal{LM}_{\lambda}(V)^{F_{n}}$
	bijectively by the claim.
	
	Therefore there uniquely exists the homomorphism 
	$\widehat{\Phi}_V\colon \mathcal{MC}_{\lambda}(V)\rightarrow \mathcal{KLM}_{\lambda}(V)$
	such that the diagram 
	\[
		\begin{tikzcd}
			0\arrow[r]&\mathrm{Ker\,}\Theta_{V}^{\lambda}\arrow[r]\arrow[d,"\widetilde{\Phi}_{V}|_{\Theta_{V}^{\lambda}}"]
			&\mathcal{E}_{\lambda}(V)\arrow[r,"\Theta_{V}^{\lambda}"]\arrow[d,"\widetilde{\Phi}_{V}"]&\mathcal{MC}_{\lambda}(V)\arrow[r]\arrow[d,"\widehat{\Phi}_V"]
			&0\\
			0\arrow[r]&\mathrm{Ker\,}\theta_{V}^{\lambda}\arrow[r]&\mathcal{LM}_{\lambda}(V)\arrow[r,"\theta_{V}^{\lambda}"]&
			\mathcal{KLM}_{\lambda}(V)\arrow[r]&0
		\end{tikzcd}	
	\]
	is commutative, and in particular 
	this $\widehat{\Phi}_V$ is isomorphism 
	since the left and middle vertical morphisms of the diagram are isomorphisms.

	Then it remains to prove the claim.
	Since $D_{<1/m}^{*}(a_{i})$, $i=1,\ldots,n$, and $D_{<1/m}^{*}(\infty)$
	are homotopy equivalent to $S^{1}$,
	Lemma \ref{lem:Hat} shows that 
	\begin{multline*}
		\mathrm{Im\,}(H_{1}(D^{*}_{<1/m}(a_{i});\iota_{a_{0}}^{*}(C_{\lambda}(V))|_{D^{*}_{<1/m}(a_{i})})
		\rightarrow H_{1}(\mathbb{P}^{1}\backslash Q_{n}^{\infty}(a_{0});\iota_{a_{0}}^{*}(C_{\lambda}(V))))\\
		=\{\widetilde{\gamma_{i}}\otimes v\in H_{1}(\mathbb{P}^{1}\backslash Q_{n}^{\infty}(a_{0});\iota_{a_{0}}^{*}(C_{\lambda}(V)))\mid 
		v\in (\iota_{a_{0}}^{*}(C_{\lambda}(V)))^{\gamma_{i}}\}
	\end{multline*}
	for $i=1,\ldots,n$, and 
	\begin{multline*}
		\mathrm{Im\,}(H_{1}(D^{*}_{<1/m}(\infty);\iota_{a_{0}}^{*}(C_{\lambda}(V))|_{D^{*}_{<1/m}(\infty)})
		\rightarrow H_{1}(\mathbb{P}^{1}\backslash Q_{n}^{\infty}(a_{0});\iota_{a_{0}}^{*}(C_{\lambda}(V))))\\
		=\{\widetilde{\gamma_{\infty}}\otimes v\in H_{1}(\mathbb{P}^{1}\backslash Q_{n}^{\infty}(a_{0});\iota_{a_{0}}^{*}(C_{\lambda}(V)))\mid 
		v\in (\iota_{a_{0}}^{*}(C_{\lambda}(V)))^{\gamma_{\infty}}\}.
	\end{multline*}
	As in the first part in the proof of Proposition \ref{prop:aughom},
	we have 
	\begin{align*}
		\widetilde{[\gamma_{0},\gamma_{i}]}\otimes v&=
		\widetilde{\gamma_{0}}\otimes \lambda(\gamma_{i}-1)\cdot v+
		\widetilde{\gamma_{i}}\otimes (1-\lambda)(\gamma_{i}\cdot v)\\
		&=\widetilde{\gamma_{i}}\otimes (1-\lambda) v
	\end{align*}
	for $v\in (\iota_{a_{0}}^{*}(C_{\lambda}(V)))^{\gamma_{i}}$ and $i=1,\ldots,n$,
	and similarly
	\[
		\widetilde{[\gamma_{0},\gamma_{\infty}]}\otimes v=\widetilde{\gamma_{\infty}}\otimes (1-\lambda) v
	\]
	for $v\in (\iota_{a_{0}}^{*}(C_{\lambda}(V)))^{\gamma_{\infty}}$.
	Then since $\lambda\neq 1$, the claim follows.
\end{proof}

\section{Subcategory $\mathbf{Mod}_{k[F_{n}\rtimes_{\alpha}G]}^{\mathrm{NT}}$}
In \cite{Katz}, Katz introduced 
certain conditions called the property $\wp$
which assures that the middle convolution functor behaves 
well. 
As an analogue of these conditions, 
Dettweiler-Reiter \cite{DR} and V\"olklein \cite{Vol}
defined a full subcategory of $\mathbf{Mod}_{k[F_{n}]}$,
and showed many good properties of the middle convolution 
functor inside this subcategory.
In this section, we shall extend this subcategory 
to a full subcategory of 
$\mathbf{Mod}_{k[F_{n}\rtimes_{\alpha}G]}$,
and 
show fundamental properties of the Katz-Long-Moody functor
on this subcategory.
\subsection{Definition of $\mathbf{Mod}_{k[F_{n}\rtimes_{\alpha}G]}^{\mathrm{NT}}$}
Let us consider the projection map from $F_{n}$ to the infinite cyclic group $F_{1}=\langle x\rangle$ defined by
\[
	\begin{array}{cccc}	
		p_{i}\colon &F_{n}&\longrightarrow &F_{1}\\
		&x_{j}&\longmapsto & p_{i}(x_{j}):=
		\begin{cases}
			x&\text{if }j= i\\
			1&\text{otherwise}
		\end{cases}
	\end{array},
\]
which extends to the group homomorphism.
\begin{dfn}[almost trival module, V\"olklein \cite{Vol}]\normalfont
	We say that a left $k[F_{n}]$-module $V$ is 
	\emph{almost trivial} if there exists 
	a 1-dimensional $k[F_{1}]$-module $K$
	and $V$ is isomorphic to $p_{i}^{*}(K)$ for some $i=1,\ldots,n$.
\end{dfn}
For 
	a left $R$-module $M$ and ring homomorphism $\phi\colon P\rightarrow R$,
	 $\phi^{*}(M)$ 
	denotes the pull-back module of $M$.
	Namely, $\phi^{*}(M)=M$ as Abel groups and the $P$-module structure is 
	define by 
	$p\cdot m:=\phi(p)\cdot m$ for $p\in P$ and $m\in M$.

\begin{dfn}[cf. Katz, Lemma 2.6.14 and Corollay 2.6.15 in \cite{Katz};
	V\"olklein, Section 2.2 in \cite{Vol}; Dettweiler-Reiter, Section 3 in \cite{DR}]\label{df:NT}
	\normalfont
	We say that  a left $k[F_{n}]$-module $V$ satisfies the property $(P1)$
	if 
	\[
		\mathrm{Hom}_{k[F_{n}]}(T,V)=\{0\}	
	\]
	for all almost trivial modules $T$.
	Similarly 
	we say that  a left $k[F_{n}]$-module $V$ satisfies the property $(P2)$
	if 
	\[
		\mathrm{Hom}_{k[F_{n}]}(V,T)=\{0\}	
	\]
	for all almost trivial modules $T$.

	Let $\mathbf{Mod}_{k[F_{n}\rtimes_{\alpha}G]}^{\mathrm{NT}}$ be 
	the full subcategory of $\mathbf{Mod}_{k[F_{n}\rtimes_{\alpha}G]}$ 
	whose object consists 
	of 
	$k[F_{n}\rtimes_{\alpha}G]$-module $V$ satisfying both of $(P1)$ and $(P2)$
	when we regard $V$ as a $k[F_{n}]$-module through the inclusion 
	$k[F_{n}]\hookrightarrow k[F_{n}\rtimes_{\alpha}G]$.
\end{dfn}

\begin{rem}\label{rem:nontriv}\normalfont
	Let $V$ be a left $k[F_{n}]$-module and 
	$\rho_{V}\colon F_{n}\rightarrow \mathrm{Aut}_{k}(V)$
	the induced group homomorphism.
	Then the property $(P1)$
	is equivalent to that
	\begin{equation}\label{eq:nontriv1}
		(\mathrm{Ker\,}(\rho_{V}(x_{i})-\tau))\cap\bigcap_{j\neq i}\mathrm{Ker\,}(\rho_{V}(x_{j})-1)=\{0\}
	\end{equation}
	for all $i=1,\ldots,n$ and $\tau\in k^{\times}$.
	Similarly,
	the property $(P2)$
	is equivalent to that
	\begin{equation}\label{eq:nontriv2}
		(\mathrm{Im\,}(\rho_{V}(x_{i})-\tau))+\sum_{j\neq i}\mathrm{Im\,}(\rho_{V}(x_{j})-1)=V
	\end{equation}
	for all $i=1,\ldots,n$ and $\tau\in k^{\times}$.
\end{rem}
\subsection{Katz-Long-Moody functor in $\mathbf{Mod}_{k[F_{n}\rtimes_{\alpha}G]}^{\mathrm{NT}}$}
The following proposition
assures that 
the full subcategory $\mathbf{Mod}_{k[F_{n}\rtimes_{\alpha}G]}^{\mathrm{NT}}$
is preserved by the Katz-Long-Moody functor.
\begin{prp}[cf. Corollary 6.2.7 in \cite{Katz}, Lemma 2.2 in \cite{Vol}, and Proposition 3.4 in \cite{DR}]\label{prop:preserve}
	Let $V$ be a left $k[F_{n}\rtimes_{\alpha}G]$-module and $\lambda\in k^{\times}$.
	If $V$ has the property $(P1)$, 
	then so does $\mathcal{KLM}(V)$.
	Similarly, if $V$ has the property $(P2)$,
	then so does $\mathcal{KLM}(V)$.
\end{prp}
\begin{proof}
	First we suppose $\lambda=1$. If $\mathrm{Hom}_{k[F_{n}]}(T,V)=\{0\}$
	for an almost trivial module $T$, then we have $\mathrm{Hom}_{k[F_{n}]}(T,\mathcal{KLM}_{1}(V))=\{0\}$
	since $\mathcal{KLM}_{1}(V)$ is isomorphic to a submodule of $V$
	by Proposition \ref{prop:1middle}.
	Meanwhile, if $\mathrm{Hom}_{k[F_{n}]}(V,T)=\{0\}$ for  any almost trivial modules $T$,
	then Propositions \ref{prop:multiplication} and \ref{prop:1middle} show that 
	$\mathcal{KLM}_{1}(V)\cong V$. Thus we have 
	$\mathrm{Hom}_{k[F_{n}]}(\mathcal{KLM}_{1}(V),T)=\{0\}$ for  any almost trivial modules $T$

	Next we suppose $\lambda\neq 1$.
	We shall only show that if $V$ has the property $(P1)$,
	then $\mathcal{KLM}_{\lambda}(V)$ has the property $(P1)$.
	The dual argument shows the property $(P2)$
	is preserved by the KLM-functor.

	We identify $\mathcal{LM}_{\lambda}(V)$ with $V^{\oplus n}$ by the isomorphism $(\ref{eq:indentfy})$.
	Let $\rho_{\lambda}^{\mathrm{LM}}\colon F_{n}\rtimes_{\alpha}G\rightarrow \mathrm{Aut}_{k}(V^{\oplus n})$
	be the associated group homomorphism with $\mathcal{LM}_{\lambda}(V)$.
	Then to show the condition $(\ref{eq:nontriv1})$ holds for $\mathcal{KLM}_{\lambda}(V)$,
	it suffices to see that
	the following claim holds for every $i=1,\ldots,n$ and $\tau\in k^{\times}$. Namely,
	if $v\in \mathcal{LM}_{\lambda}(V)$ satisfies 
	\begin{equation*}
		(\rho_{\lambda}^{\mathrm{LM}}(x_{i})-\tau)v\in\mathrm{Ker\,}\theta_{V}^{\lambda},\quad
		(\rho_{\lambda}^{\mathrm{LM}}(x_{j})-1)v\in \mathrm{Ker\,}\theta_{V}^{\lambda},
	\end{equation*}
	for $j\in \{1,\ldots,n\}\backslash\{i\}$, then $v\in \mathrm{Ker\,}\theta_{V}^{\lambda}.$
	This claim follows from the lemmas below.
\end{proof}
Before showing the claim in the above proposition, we prepare a lemma.
\begin{lmm}\label{lem:prep}
	Let $V$ be a left $k[F_{n}\rtimes_{\alpha}G]$-module and $\lambda\in k^{\times}\backslash\{1\}$.
	If $v\in \mathcal{LM}_{\lambda}(V)$ satisfies 
	that $(x-1)v\in \bigoplus_{i=1}^{n}\mathcal{LM}_{\lambda}(V)^{[i]}$
	for all $x\in F_{n}$,
	then $v\in \mathrm{Ker\,}\theta_{V}^{\lambda}$.
\end{lmm}
\begin{proof}
	By the assumption, we have $(x_{i}-1)v\in  \bigoplus_{j=1}^{n}\mathcal{LM}_{\lambda}(V)^{[j]}$ for all $i=1,\ldots,n$.
	Since $(x_{i}-1)v\in (x_{i}-1)\otimes V$ and $\left( (x_{i}-1)\otimes V\right)\cap \bigoplus_{j=1}^{n}\mathcal{LM}_{\lambda}(V)^{[j]}=
	\mathcal{LM}_{\lambda}(V)^{[i]}$,
	we have $(x_{i}-1)v\in \mathcal{LM}_{\lambda}(V)^{[i]}$.
	Set $w_{i}:=(x_{i}-1)v$, $i=1,\ldots,n$ and $w:=\sum_{i=1}^{n}w_{i}$.
	Then $w\in \bigoplus_{j=1}^{n} \mathcal{LM}_{\lambda}(V)^{[j]}$ and moreover we have 
	$(x_{i}-1)w=\sum_{j=1}^{n}(x_{j}-1)w_{j}=(x_{i}-1)w_{i}=(\lambda-1)w_{i}$ since $w_{j}\in \mathcal{LM}_{\lambda}(V)^{[j]}$.
	Therefore 
	we obtain 
	\[
		(x_{i}-1)(v-(\lambda-1)^{-1}w)=w_{i}-(\lambda-1)^{-1}(\lambda-1)w_{i}=0\quad \text{for all } i=1,\ldots,n,	
	\]
	which shows that $v=(\lambda-1)^{-1}w+(v-(\lambda-1)^{-1}w)\in \bigoplus_{j=1}^{n}\mathcal{LM}_{\lambda}(V)^{[j]}+\mathcal{LM}_{\lambda}(V)^{F_{n}}
	=\mathrm{Ker\,}\theta_{V}^{\lambda}$.
\end{proof}
Now we give a proof of the above claim
\begin{lmm}\label{lem:prep2}
	Let $V$ be a left $k[F_{n}\rtimes_{\alpha}G]$-module satisfying $(P1)$.
	Let us take $\lambda\in k^{\times}\backslash\{1\}$.
	Let us take $v\in \mathcal{LM}_{\lambda}(V)$, $\tau\in k^{\times}$, and 
	$i\in \{1,\ldots,n\}$ arbitrarily.
	Then if 
	it is satisfied that 
	\begin{equation}\label{eq:assump}
		(\rho^{\mathrm{LM}}_{\lambda}(x_{i})-\tau)v\in\mathrm{Ker\,}\theta_{V}^{\lambda},\quad
		(\rho^{\mathrm{LM}}_{\lambda}(x_{j})-1)v\in \mathrm{Ker\,}\theta_{V}^{\lambda},
	\end{equation}
	for $j\in \{1,\ldots,n\}\backslash\{i\}$, then $v\in \mathrm{Ker\,}\theta_{V}^{\lambda}.$
\end{lmm}	
\begin{proof}
	Now we claim that $(\rho^{\mathrm{LM}}_{\lambda}(x_{j})-1)v\in \mathrm{Ker\,}\theta_{V}^{\lambda}$
	is equivalent to $(\rho^{\mathrm{LM}}_{\lambda}(x_{j})-1)v\in \mathcal{LM}_{\lambda}(V)^{[j]}$.
	Since
	$(\rho^{\mathrm{LM}}_{\lambda}(x_{j})-1)v\in (x_{j}-1)\otimes V$,
	this claim follows from the equation
	\begin{equation}\label{eq:incl}
	\left((x_{j}-1)\otimes V\right)\cap \mathrm{Ker\,}\theta_{V}^{\lambda}=
	\mathcal{LM}_{\lambda}(V)^{[j]}
	\end{equation}
	which holds under the condition $(\ref{eq:nontriv1})$.
	The inclusion $\supset$ obviously holds by definitions without the condition $(\ref{eq:nontriv1})$.
	Let us show the inclusion $\subset$.
	Take $(x_{j}-1)\otimes v_{j}\in \left((x_{j}-1)\otimes V\right)\cap \mathrm{Ker\,}\theta_{V}^{\lambda}$.
	Then we can write $(x_{j}-1)\otimes v_{j}=
	\sum_{l=1}^{n}(x_{l}-1)\otimes (k_{l}+(x_{l-1}\cdots x_{n})\cdot u)$
	with some $k_{l}\in V^{x_{l}}$ and $u\in V^{\lambda x_{1}\cdots x_{n}}$
	by Propositions \ref{prop:infinity} and \ref{prop:finite}.
	Then we have $k_{l}+(x_{l-1}\cdots x_{n})\cdot u=0$ for $l\neq j$, i.e.,
	$(x_{l-1}\cdots x_{n})\cdot u\in V^{x_{l}}$ for $l\neq j$. This shows that 
	$u\in V^{x_{l}}=\mathrm{Ker\,}(\rho_{V}(x_{l})-1)$ for $l\neq j$
	and $u\in V^{\lambda x_{j}}=\mathrm{Ker\,}(\rho_{V}(x_{j})-\lambda^{-1})$. Then the condition $(\ref{eq:nontriv1})$
	for $V$ implies that $u=0$, which shows the inclusion 
	$\left((x_{j}-1)\otimes V\right)\cap \mathrm{Ker\,}\theta_{V}^{\lambda}\subset
	\mathcal{LM}_{\lambda}(V)^{[j]}$. Thus we obtain the equation $(\ref{eq:incl})$.

	Let us suppose that the condition $(\ref{eq:assump})$ holds for $\tau=1$,
	i.e., we have $(\rho^{\mathrm{LM}}_{\lambda}(x)-1)v\in \theta_{V}^{\lambda}$ for all $x\in F_{n}$.
	Then the above claim implies that $(\rho^{\mathrm{LM}}_{\lambda}(x)-1)v\in \bigoplus_{i=1}^{n}\mathcal{LM}_{\lambda}(V)^{[i]}$
	and Lemma \ref{lem:prep} shows that $v\in \mathrm{Ker\,}\theta_{V}^{\lambda}$ as desired.
	
	Next we suppose that the condition $(\ref{eq:assump})$ holds for $\tau\neq 1$.
	Then the above claim implies that for $j\neq i$,
	there exists $k_{j}\in \mathcal{LM}_{\lambda}(V)^{[j]}$ such that 
	$(\rho_{\lambda}^{\mathrm{LM}}(x_{j})-1)v=(x_{j}-1)\otimes k_{j}$.
	Let us  write $v=(v_{1},\ldots,v_{n})\in V^{\oplus n}$
	under the identification $\mathcal{LM}_{\lambda}(V)\cong V^{\oplus n}$
	by the isomorphism $(\ref{eq:indentfy})$.
	Also we write $(\rho_{\lambda}^{\mathrm{LM}}(x_{i})-\tau)v=v^{(0)}+v^{(\infty)}$
	by $v^{(0)}=(v^{(0)}_{1},\ldots,v^{(0)}_{n})\in \bigoplus_{l=1}^{n}\mathcal{LM}_{\lambda}(V)^{[l]}$, $v^{(\infty)}=(v^{(\infty)}_{1},\ldots,v^{(\infty)}_{n})\in \mathcal{LM}_{\lambda}(V)^{F_{n}}$.
	Then these equations induces the following equations.
	\begin{align}
		\sum_{l<i}\lambda(x_{l}-1)v_{l}+(\lambda x_{i}-\tau)v_{i} +\sum_{l>i}\lambda(x_{l}-1)v_{l}&=v^{(0)}_{i}+v^{(\infty)}_{i}\label{eq:ue}\\
		(1-\tau)v_{j}&=v^{(0)}_{j}+v^{(\infty)}_{j}&\text{for }j\neq i\label{eq:mid}\\
		\sum_{l<j}\lambda(x_{l}-1)v_{l}+(\lambda x_{j}-1)v_{i} +\sum_{l>k}\lambda(x_{l}-1)v_{l}&=k_{j}&\text{for }j\neq i\label{eq:sita}
	\end{align}
	Substituting the equations $(\ref{eq:mid})$ to $(\ref{eq:ue})$ and $(\ref{eq:sita})$,
	we moreover have the following.
	\begin{equation}\label{eq:saigo}
	\begin{aligned}
		(\lambda x_{i}-\tau)u_{i}+(1-\tau)^{-1}(v^{(\infty)}_{i-1}-\lambda v_{i}^{(\infty)})&=v^{(0)}_{i}+v^{(\infty)}_{i}\\
		\lambda(x_{i}-1)u_{i}+(1-\tau)^{-1}(\lambda v_{i-1}^{(\infty)}-\lambda v_{i}^{(\infty)})&=k_{j}+(1-\lambda)(1-\tau)^{-1}v_{j}^{(0)}&\text{for }j<i\\
		(x_{i}-1)u_{i}+(1-\tau)^{-1}( v_{i-1}^{(\infty)}- v_{i}^{(\infty)})&=k_{j}+(1-\lambda)(1-\tau)^{-1}v_{j}^{(0)}&\text{for }j>i
	\end{aligned}
	\end{equation}
	Here we use the fact that $v_{l}^{(0)}\in V^{x_{l}}$ and $x_{l}\cdot v_{l}^{(\infty)}=v_{l-1}^{(\infty)}$, and 
	we formally put $v_{0}^{(0)}=\lambda^{-1}v_{n}^{(\infty)}$.
	Since $k_{j},v_{j}^{(0)}\in V^{x_{j}}$, the equations $(\ref{eq:saigo})$ induces
	\begin{align}
		(x_{i}-1)u_{i}+(1-\tau)^{-1}( v_{i-1}^{(\infty)}- v_{i}^{(\infty)})&=(\frac{\tau}{\lambda}-1)v_{1}+v_{1}^{(0)}+(1-\frac{\tau}{\lambda})(1-\tau)^{-1}v_{i}^{(\infty)}\label{eq:saiggo}\\
		(x_{i}-1)u_{i}+(1-\tau)^{-1}( v_{i-1}^{(\infty)}- v_{i}^{(\infty)})&\in V^{x_{j}}\quad \text{for }j\neq i
	\end{align}
	and moreover the equation $(\ref{eq:saiggo})$
	shows that $(x_{i}-1)u_{i}+(1-\tau)^{-1}( v_{i-1}^{(\infty)}- v_{i}^{(\infty)})\in V^{\frac{\tau}{\lambda}x_{i}}$.
	Thus the condition $(\ref{eq:nontriv1})$ for $V$ implies that 
	\[
		(x_{i}-1)u_{i}+(1-\tau)^{-1}( v_{i-1}^{(\infty)}- v_{i}^{(\infty)})\in V^{\frac{\tau}{\lambda}x_{i}}\cap \bigcap_{j\neq i}V^{x_{j}}=\{0\},
	\]
	namely, we have 
	\begin{equation}\label{eq:ssaigo}
		(x_{i}-1)(v_{i}-(1-\tau)^{-1}v_{i}^{(\infty)})=(x_{i}-1)u_{i}+(1-\tau)^{-1}( v_{i-1}^{(\infty)}- v_{i}^{(\infty)})=0.
	\end{equation}

	Finally recalling that we have $
	v_{j}-(1-\tau)^{-1}v_{j}^{(\infty)}=(1-\tau)^{-1}v_{j}^{(0)}\in V^{x_{j}}$
	for $j\neq i$  by $(\ref{eq:mid})$ and 
	$v_{i}-(1-\tau)^{-1}v_{i}^{(\infty)}\in V^{x_{i}}$ by $(\ref{eq:ssaigo})$,
	we conclude that 
	\[
		v=(v-(1-\tau)^{-1}v^{(\infty)})+(1-\tau)^{-1}v^{(\infty)}\in \bigoplus_{l=1}^{n}\mathcal{LM}_{\lambda}(V)^{[l]}\oplus \mathcal{LM}_{\lambda}(V)^{F_{n}}=
		\mathrm{Ker\,}\theta_{V}^{\lambda}
	\]
	as desired.
\end{proof}
The following fundamental theorem shows that 
the Katz-Long-Moody functor gives a nontrivial  
auto-equivalence of category $\mathbf{Mod}_{k[F_{n}\rtimes_{\alpha}G]}^{\mathrm{NT}}$.
This is a natural generalization of Theorem 2.9.7 in \cite{Katz},
Theorem 2.4 in \cite{Vol}, and Theorem 3.5 in \cite{DR}.

\begin{thm}\label{thm:KLM}
	Let us take $\lambda,\tau\in k^{\times}$.
	Then we have the following.
	\begin{enumerate}
		\item As endfunctors of $\mathbf{Mod}_{k[F_{n}\rtimes_{\alpha}G]}^{\mathrm{NT}}$,
		 $\mathcal{KLM}_{\lambda}\circ \mathcal{KLM}_{\tau}$
		and $\mathcal{KLM}_{\lambda\tau}$ are isomorphic.
		\item The functor $\mathcal{KLM}_{\lambda}$ is an auto-equivalence of the category 
		$\mathbf{Mod}_{k[F_{n}\rtimes_{\alpha}G]}^{\mathrm{NT}}$.
		In particular, $\mathcal{KLM}_{\lambda^{-1}}$ is an inverse functor.
	\end{enumerate}
\end{thm}
The second assertion follows from the first one whose 
proof is given in the next section.

As well as the the Katz algorithm for 
local systems on $\mathbb{C}\backslash\{n\text{-points}\}$,
we also have another 
operation called the \emph{multiplication}, which is defined by
\[
	\mathcal{M}ul_{\chi}\colon \mathbf{Mod}_{k[F_{n}\rtimes_{\alpha}G]}\ni V\longmapsto \chi\otimes_{\mathbb{C}}V\in  	\mathbf{Mod}_{k[F_{n}\rtimes_{\alpha}G]}
\]
for an arbitrarily $1$-dimensional module $\chi\in \mathbf{Mod}_{k[F_{n}\rtimes_{\alpha}G]}$.
Applying these  KLM-functors and multiplication functors repeatedly,
we can construct families of infinitely many non-isomorphic $k[F_{n}\rtimes_{\alpha}G]$-modules.

\subsection{A proof of Theorem \ref{thm:KLM}}
This section is dedicated to a proof of the first assertion of Theorem \ref{thm:KLM}.
\begin{lmm}\label{lem:prep3}
	Let $V$ be a left $k[F_{n}\rtimes_{\alpha}G]$-module satisfying $(P1)$.
	Let us take $\lambda\in k^{\times}\backslash\{1\}$.
	If $v\in V$ satisfies that 
	\[
		(x_{1}\cdots x_{n}-1)\otimes v\in \bigoplus_{i=1}^{n}\mathcal{LM}_{\lambda}(V)^{[i]},
	\]
	then $v=0$.
\end{lmm}
\begin{proof}
	Under the isomorphism $(\ref{eq:indentfy})$,
	we can write $$(x_{1}\cdots x_{n}-1)\otimes v=((x_{2}\cdots x_{n})\cdot v,(x_{3}\cdots x_{n})\cdot v,\ldots, v)
	\in V^{\oplus n}.$$
	Then since $v\in \bigoplus_{i=1}^{n}\mathcal{LM}_{\lambda}(V)^{[i]}$,
	we obtain that $$x_{n}\cdot v=v,\ x_{n-1}x_{n}\cdot v=x_{n-1}v=v,\ \ldots\ ,(x_{2}\cdots x_{n})\cdot v=v,$$
	i.e., $(x_{1}\cdots x_{n}-1)\otimes v=(v,\ldots,v)$.
	Then it follows that $v\in \bigcap_{i=1}^{n}\mathrm{Ker\,}(x_{i}-1)=\{0\}$, i.e., $v=0.$
\end{proof}
\begin{lmm}\label{lem:surj}
	Let $V$ be a left $k[F_{n}\rtimes_{\alpha}G]$-module satisfying $(P1)$.
	Let us take $\lambda,\tau\in k^{\times}\backslash\{1\}$.
Let us consider the commutative diagram of $k[F_{n}\rtimes_{\alpha}G]$-modules,
\[
	\begin{tikzcd}
		0\arrow[r]&\mathrm{Ker\,}\theta_{\mathcal{LM}_{\lambda}(V)}^{\lambda}\arrow[r]\arrow[d,"\mathcal{LM}_{\lambda}(\theta_{V}^{\tau})"]
		&\mathcal{LM}_{\lambda}\circ\mathcal{LM}_{\tau}(V)\arrow[r,"\theta_{\mathcal{LM}_{\tau}(V)}^{\lambda}"]\arrow[d,"\mathcal{LM}_{\lambda}(\theta_{V}^{\tau})"]
		&\mathcal{KLM}_{\lambda}\circ\mathcal{LM}_{\tau}(V)\arrow[r]\arrow[d,"\mathcal{KLM}_{\lambda}(\theta_{V}^{\tau})"]&0\\
		0\arrow[r]&\mathrm{Ker\,}\theta_{\mathcal{KLM}_{\lambda}(V)}^{\lambda}\arrow[r]
		&\mathcal{LM}_{\lambda}\circ\mathcal{KLM}_{\tau}(V)\arrow[r,"\theta_{\mathcal{KLM}_{\tau}(V)}^{\lambda}"]
		&\mathcal{KLM}_{\lambda}\circ\mathcal{KLM}_{\tau}(V)\arrow[r]&0
	\end{tikzcd}	
\]
whose horizontal sequences are exact. Then 
the left vertical map 
\[
	\mathcal{LM}_{\lambda}(\theta_{V}^{\tau})\colon \mathrm{Ker\,}\theta_{\mathcal{LM}_{\tau}(V)}^{\lambda}\longrightarrow \mathrm{Ker\,}\theta_{\mathcal{KLM}_{\tau}(V)}^{\lambda}
\]
is surjective.
\end{lmm}
\begin{proof}
	First we claim that 
	\[
		\theta_{V}^{\tau}\colon \mathcal{LM}_{\tau}(V)^{x_{i}}\longrightarrow 	\mathcal{KLM}_{\tau}(V)^{x_{i}}
	\]
	is surjective for any $i=1,\ldots,n.$
	Since $\theta_{V}^{\tau}\colon \mathcal{LM}_{\tau}(V)\rightarrow 	\mathcal{KLM}_{\tau}(V)$ is surjective,
	there exists 
	$\bar{w}\in \mathcal{LM}_{\tau}(V)$ such that $\theta_{V}^{\tau}(\bar{w})=w$ for any $w\in \mathcal{KLM}_{\tau}(V)^{x_{i}}$.
	Then it follows that 
	$0=(x_{i}-1)w=\theta_{V}^{\tau}((x_{i}-1)\bar{w})$, i.e., $(x_{i}-1)\bar{w}\in \mathrm{Ker\,}\theta_{V}^{\tau}$.
	We  apply the equation $(\ref{eq:incl})$ in Lemma \ref{lem:prep2}
	and obtain $(x_{i}-1)\bar{w}\in \mathcal{LM}_{\tau}(V)^{[i]}$.
	Therefore putting 
	$\tilde{w}:=(x_{i}-1)\bar{w}\in \mathcal{LM}_{\tau}(V)^{[i]}$,
	we obtain $$(x_{i}-1)(\bar{w}-(\tau-1)^{-1}\tilde{w})=\tilde{w}-(\tau-1)^{-1}(\tau-1)\tilde{w}=0.$$
	Namely,
	$\hat{w}:=(\bar{w}-(\tau-1)^{-1}\tilde{w})\in \mathcal{LM}_{\tau}(V)^{x_{i}}$.
	Since $\tilde{w}\in \mathcal{LM}_{\tau}(V)^{[i]}\subset \mathrm{Ker\,}\theta_{V}^{\tau}$,
	we have $\theta_{V}^{\tau}(\hat{w})=\theta_{V}^{\tau}(\bar{w})=w$, which shows the claim.

	Recalling that 
	\begin{align*}
	\mathcal{LM}_{\lambda}(\mathcal{LM}_{\tau}(V))^{[i]}&=\{(x_{i}-1)\otimes w\mid w\in \mathcal{LM}_{\tau}(V)^{x_{i}}\},\\
	\mathcal{LM}_{\lambda}(\mathcal{KLM}_{\tau}(V))^{[i]}&=\{(x_{i}-1)\otimes w\mid w\in \mathcal{KLM}_{\tau}(V)^{x_{i}}\},
	\end{align*}
	$i=1,\ldots,n$,
	and $\mathcal{LM}_{\lambda}(\theta_{V}^{\tau})=\mathrm{id}_{I_{F_{n},\lambda}}\otimes \theta_{V}^{\tau}$,
	we can conclude that 
	\[
		\mathcal{LM}_{\lambda}(\theta_{V}^{\tau})\colon \bigoplus_{i=1}^{n}\mathcal{LM}_{\lambda}(\mathcal{LM}_{\tau}(V))^{[i]}
		\longrightarrow \bigoplus_{i=1}^{n}\mathcal{LM}_{\lambda}(\mathcal{KLM}_{\tau}(V))^{[i]}	
	\]
	is surjective by the above claim.

	Next we claim that 
	\[
		\theta_{V}^{\tau}\colon \mathcal{LM}_{\tau}(V)^{\lambda x_{1}\cdots x_{n}}\longrightarrow 	\mathcal{KLM}_{\tau}(V)^{\lambda x_{1}\cdots x_{n}}
	\]
	is surjective.
	As well as the above claim,
	there exists 
	$\bar{w}\in \mathcal{LM}_{\tau}(V)$ such that $\theta_{V}^{\tau}(\bar{w})=w$ for any $w\in \mathcal{KLM}_{\tau}(V)^{\lambda x_{1}\cdots x_{n}}$.
	Then it follows that 
	$(\lambda x_{1}\cdots x_{n}-1)\bar{w}\in \mathrm{Ker\,}\theta_{V}^{\tau}$.
	Let us consider the case $\lambda\tau= 1$.
	In this case, 
	the matrix representation of $x_{1}\cdots x_{n}-\tau=x_{1}\cdots x_{n}-\lambda^{-1}$
	in Remark \ref{rem:DR} shows that $(\lambda x_{1}\cdots x_{n}-1)\bar{w}\in (x_{1}\cdots x_{n}-1)\otimes V$.
	Thus since $(\lambda x_{1}\cdots x_{n}-1)\bar{w}\in \mathrm{Ker\,}\theta_{V}^{\tau}=
	\bigoplus_{i=1}^{n}\mathcal{LM}_{\tau}(V)^{[i]}\oplus \mathcal{LM}_{\tau}(V)^{F_{n}}$,
	Lemma \ref{lem:prep3} 
	implies that $(\lambda x_{1}\cdots x_{n}-1)\bar{w}\in \mathcal{LM}_{\tau}(V)^{F_{n}}$.
	Put $\tilde{w}:=(\lambda x_{1}\cdots x_{n}-1)\bar{w}$.
	Then we obtain 
	\[
		(\lambda x_{1}\cdots x_{n}-1)(\bar{w}-(\lambda-1)^{-1}\tilde{w})=
		\tilde{w}-(\lambda-1)^{-1}(\lambda-1)\tilde{w}=0.
	\]
	Thus $\bar{w}-(\lambda-1)^{-1}\tilde{w}\in \mathcal{LM}_{\tau}(V)^{\lambda x_{1}\cdots x_{n}}$
	and $\theta_{V}^{\tau}(\bar{w}-(\lambda-1)^{-1}\tilde{w})=\theta_{V}^{\tau}(\bar{w})=w$ as desired.
	Also we consider the case $\lambda\tau\neq 1$.
	In this case, we can write $(\lambda x_{1}\cdots x_{n}-1)\bar{w}=w^{(0)}+w^{(\infty)}$
	by $w^{(0)}\in \bigoplus_{i=1}^{n}\mathcal{LM}_{\tau}(V)^{[i]}$
	and $w^{(\infty)}\in \mathcal{LM}_{\tau}(V)^{F_{n}}$.
	Then we obtain
	\begin{multline*}
		(\lambda x_{1}\cdots x_{n}-1)(\bar{w}-(\lambda\tau-1)^{-1}w^{(0)}-(\lambda-1)^{-1}w^{(\infty)})
		=\\w^{(0)}+w^{(\infty)}-(\lambda\tau-1)^{-1}(\lambda\tau-1)w^{(0)}-(\lambda-1)^{-1}(\lambda-1)w^{(\infty)}=0.
	\end{multline*}
	Thus $\bar{w}-(\lambda\tau-1)^{-1}w^{(0)}-(\lambda-1)^{-1}w^{(\infty)}\in \mathcal{LM}_{\tau}(V)^{\lambda x_{1}\cdots x_{n}}$
	and  $\theta_{V}^{\tau}(\bar{w}-(\lambda\tau-1)^{-1}w^{(0)}-(\lambda-1)^{-1}w^{(\infty)})=\theta_{V}^{\tau}(\bar{w})=w$.
	Therefore in any cases we obtain the claim.

	This claim assures that 
	\[
		\mathcal{LM}_{\lambda}(\theta_{V}^{\tau})\colon \mathcal{LM}_{\lambda}(\mathcal{LM}_{\tau}(V))^{F_{n}}
		\longrightarrow \mathcal{LM}_{\lambda}(\mathcal{KLM}_{\tau}(V))^{F_{n}}	
	\]
	is surjective
	as above.

	Therefore we conclude that 
	\[
	\mathcal{LM}_{\lambda}(\theta_{V}^{\tau})\colon \mathrm{Ker\,}\theta_{\mathcal{LM}_{\tau}(V)}^{\lambda}\longrightarrow \mathrm{Ker\,}\theta_{\mathcal{KLM}_{\tau}(V)}^{\lambda}
\]
is surjective.
\end{proof}

\begin{lmm}\label{lem:prethm1}
	Let $V$ be a left $k[F_{n}\rtimes_{\alpha}G]$-module satisfying $(P2)$.
	Let us take $\lambda,\tau\in k^{\times}$.
	Then the multiplication map 
	\[
		\nabla\colon \mathcal{LM}_{\lambda}(\mathcal{LM}_{\tau}(V))\longrightarrow \mathcal{LM}_{\lambda\tau}(V)
	\]
	is surjective.
	Moreover if $\lambda\neq 1$, we obtain the short exact sequence
	\[
		0\rightarrow \bigoplus_{i=1}^{n}\mathcal{LM}_{\lambda}(\mathcal{LM}_{\tau}(V))^{[i]}
		\rightarrow \mathcal{LM}_{\lambda}(\mathcal{LM}_{\tau}(V))
		\xrightarrow{\nabla}
		\mathcal{LM}_{\lambda\tau}(V)\rightarrow 0.
	\]
\end{lmm}
\begin{proof}
	Let us write $V^{\oplus n}$ by $W$ for simplicity and 
	we regard $\mathcal{LM}_{\lambda}(\mathcal{LM}_{\tau}(V))$ as $W^{\oplus n}$ 
	and also $\mathcal{LM}_{\lambda\tau}(V)$ as $W$
	by the isomorphism $(\ref{eq:indentfy})$.
	Then the multiplication map is written in the following form,
	\[
		\begin{array}{cccc}
			\nabla\colon &W^{\oplus n}&\longrightarrow &W\\
			&(w_{1},\ldots,w_{n})&\longmapsto &\sum_{i=1}^{n}(\rho_{\tau}^{\mathrm{LM}}(x_{i})-1)w_{i}
		\end{array}.
	\] 
	For $(w_{1},\ldots,w_{n})\in W^{\oplus n}$, we write 
	$w_{i}=(v^{(i)}_{1},\ldots,v^{(i)}_{n})\in V^{\oplus n}=W$.
	Then the matrix representation of $\rho_{\tau}^{\mathrm{LM}}(x)$
	in Remark \ref{rem:DR} tells us that 
	\begin{align*}
		\sum_{i=1}^{n}(\rho_{\tau}^{\mathrm{LM}}(x_{i})-1)w_{i}&=
		\big(
			(\tau x_{1}-1)v^{(1)}_{1}+(x_{2}-1)v^{(1)}_{2}+\cdots (x_{n}-1)v^{(1)}_{n},\\
			&\quad\quad  \tau(x_{1}-1)v^{(2)}_{1}+(\tau x_{2}-1)v^{(2)}_{2}+\cdots (x_{n}-1)v^{(2)}_{n},\\
			&\quad\quad  \cdots,\\
			&\quad\quad  \tau(x_{1}-1)v^{(n)}_{1}+\tau (x_{2}-1)v^{(n)}_{2}+\cdots (\tau x_{n}-1)v^{(n)}_{n}
		\big).
	\end{align*}
	Therefore the property $(P2)$ implies that $\nabla$ is surjective.

	The second assertion follows from Propositions \ref{prop:multiplication} and \ref{prop:finite}. 
\end{proof}
\begin{lmm}\label{lem:exactup}
	Let $V$ be a left $k[F_{n}\rtimes_{\alpha}G]$-module.
	Let us take $\lambda,\tau\in k^{\times}$ so that 
	$\lambda\neq 1$.
	We consider the multiplication map 
	$
		\nabla\colon \mathcal{LM}_{\lambda}(\mathcal{LM}_{\tau}(V))\longrightarrow \mathcal{LM}_{\lambda\tau}(V).
	$
	Then we have the inclusion
	\[
		\nabla(\mathrm{Ker\,}\mathcal{LM}_{\lambda}(\theta_{V}^{\tau}))	
		\subset \bigoplus_{i=1}^{n}\mathcal{LM}_{\lambda\tau}(V)^{[i]},
	\]
	which is the equation 
	especially when $\lambda\tau\neq 1$.
\end{lmm}
\begin{proof}
	Note that we have 
	\begin{align*}
		\mathrm{Ker\,}\mathcal{LM}_{\lambda}(\theta_{V}^{\tau})&=
		\mathcal{LM}_{\lambda}(\mathrm{Ker\,}\theta_{V}^{\tau})\\
		&=
		\bigoplus_{i=1}^{n}\mathcal{LM}_{\lambda}(\mathcal{LM}_{\tau}(V)^{[i]})
		\oplus 	\mathcal{LM}_{\lambda}(\mathcal{LM}_{\tau}(V)^{F_{n}})
	\end{align*}
	since the twisted Long-Moody functor is additive and exact.
	Put $W:=V^{\oplus n}$ and  regard 
	$\mathcal{LM}_{\lambda}(\mathcal{LM}_{\tau}(V))$ as $W^{\oplus n}$
	and moreover regard $\mathcal{LM}_{\lambda\tau}(V)$ as $W=V^{\oplus n}$
	by the isomorphism $(\ref{eq:indentfy})$.
	Then Proposition \ref{prop:finite} tells us that 
	\begin{align*}
		&\bigoplus_{i=1}^{n}\mathcal{LM}_{\lambda}(\mathcal{LM}_{\tau}(V)^{[i]})=
		\\&\ 
		\{(w_{1},\ldots,w_{n})\in W^{\oplus n}\mid 
		w_{i}=(v_{1}^{(i)},\ldots,v_{n}^{(i)})
		\text{ with }
		v_{j}\in \mathrm{Ker\,}(\rho_{V}(x_{j})-1),
		\,i,j=1,\ldots,n
		\},\\\\
		&\bigoplus_{i=1}^{n}\mathcal{LM}_{\lambda\tau}(V)^{[i]}
		=
		\{(v_{1},\ldots,v_{n})\in V^{\oplus n}\mid 
		v_{i}\in \mathrm{Ker\,}(\rho_{V}(x_{i})-1),\,
		i=1,\ldots,n\},
	\end{align*}
	where $\rho_{V}\colon F_{n}\rtimes_{\alpha}G\rightarrow \mathrm{Aut}_{k}(V)$
	is the induced group homomorphism.
	Then we have that
	\[
		\begin{array}{cccc}
		\nabla\colon &\bigoplus_{i=1}^{n}\mathcal{LM}_{\lambda}(\mathcal{LM}_{\tau}(V)^{[i]})&
		\longrightarrow&
		\bigoplus_{i=1}^{n}\mathcal{LM}_{\lambda\tau}(V)^{[i]}\\
		&(w_{1},\ldots,w_{n})&\longmapsto 
		&(\lambda-1)\sum_{i=1}^{n}w_{i}
		\end{array}	
	\]
	is well-defined, and moreover 
	surjective especially when $\lambda\tau\neq 1$.
	Thus we obtain the result by Proposition \ref{prop:finite}.
\end{proof}
\begin{crl}\label{cor:inverse}
	Let $V$ be a left $k[F_{n}\rtimes_{\alpha}G]$-module.
	Let us take $\lambda,\tau\in k^{\times}$ so that 
	$\lambda\neq 1$.
	We consider the multiplication map 
	\[
		\nabla\colon \mathcal{LM}_{\lambda}(\mathcal{LM}_{\tau}(V))\longrightarrow \mathcal{LM}_{\lambda\tau}(V).
	\]
	Then we have 
	\[
		\nabla^{-1}(\mathrm{Ker\,}\theta_{V}^{\lambda\tau})=
		\mathrm{Ker\,}\theta_{\mathcal{LM}_{\tau}(V)}^{\lambda}+
		\mathrm{Ker\,}\mathcal{LM}_{\lambda}(\theta_{V}^{\tau}).
	\]
\end{crl}
\begin{proof}
	First we show that 
	\begin{equation}\label{eq:eqeq1}
		\nabla^{-1}(\mathcal{LM}_{\lambda\tau}(V)^{F_{n}})
		=\mathrm{Ker\,}\theta_{\mathcal{LM}_{\tau}(V)}^{\lambda}.
	\end{equation}
	Let us take  $w\in \mathrm{Ker\,}\theta_{\mathcal{LM}_{\tau}(V)}^{\lambda}$
	and write $w=
	w^{(0)}+w^{(\infty)}$
	according to the decomposition  
	$\mathrm{Ker\,}\theta_{\mathcal{LM}_{\tau}(V)}^{\lambda}=\bigoplus_{i=1}^{n}\mathcal{LM}_{\lambda}(\mathcal{LM}_{\tau}(V))^{[i]}
	\oplus \mathcal{LM}_{\lambda}(\mathcal{LM}_{\tau}(V))^{F_{n}}$.
	Then recalling that $\mathrm{Ker\,}\nabla=
	\bigoplus_{i=1}^{n}\mathcal{LM}_{\lambda}(\mathcal{LM}_{\tau}(V))^{[i]}$
	from Lemma \ref{lem:prethm1},
	we have 
	\[x\nabla(w)=x\nabla(w^{(\infty)})
	=\nabla(xw^{(\infty)})=\nabla(w^{(\infty)})=
	\nabla(w)
	\]
	for any $x\in F_{n}$, i.e.
	$\nabla(w)\in \mathcal{LM}_{\lambda\tau}(V)^{F_{n}}$.
	Thus  
	$\nabla^{-1}(\mathcal{LM}_{\lambda\tau}(V)^{F_{n}})
	\supset \mathrm{Ker\,}\theta_{\mathcal{LM}_{\tau}(V)}^{\lambda}.$
	Conversely take 
	$w\in \nabla^{-1}(\mathcal{LM}_{\lambda\tau}(V)^{F_{n}})$.
	Then we have 
	$$(x-1)w\in \mathrm{Ker\,}\nabla=\bigoplus_{i=1}^{n}\mathcal{LM}_{\lambda}(\mathcal{LM}_{\tau}(V))^{[i]},$$
	for any $x\in F_{n}$.
	Thus Lemma \ref{lem:prep} shows that $w\in \mathrm{Ker\,}\theta_{\mathcal{LM}_{\tau}(V)}^{\lambda}$.
	Therefore we obtain the equation $(\ref{eq:eqeq1})$.

	If $\lambda\tau=1$, then $\mathrm{Ker\,}\theta_{V}^{\lambda\tau}=
	\mathcal{LM}_{\lambda\tau}(V)^{F_{n}}$. 
	Also Lemma \ref{lem:exactup}
	tells that  $$\nabla(\mathrm{Ker\,}\mathcal{LM}_{\lambda}(\theta_{V}^{\tau}))	
	\subset \bigoplus_{i=1}^{n}\mathcal{LM}_{\lambda\tau}(V)^{[i]}
	\subset \mathrm{Ker\,}\theta_{V}^{\lambda\tau}.$$
	Thus the equation $(\ref{eq:eqeq1})$
	gives the desired equation.

	Next we suppose $\lambda\tau\neq 1$ and we show 
	\begin{equation}\label{eq:eqeq2}
		\nabla^{-1}(\bigoplus_{i=1}^{n}\mathcal{LM}_{\lambda\tau}(V)^{[i]})
		=\mathrm{Ker\,}\mathcal{LM}_{\lambda}(\theta_{V}^{\tau})
		+\mathrm{Ker\,}\nabla.	
	\end{equation}
	Since $\mathrm{Ker\,}\mathcal{LM}_{\lambda}(\theta_{V}^{\tau})=
	\bigoplus_{i=1}^{n}\mathcal{LM}_{\lambda}(\mathcal{LM}_{\tau}(V)^{[i]})
	\oplus \mathcal{LM}_{\lambda}(\mathcal{LM}_{\tau}(V)^{F_{n}})$,
	the inclusion $\supset$ follows from Lemma \ref{lem:exactup}.
	Let us take $w\in \nabla^{-1}(\bigoplus_{i=1}^{n}\mathcal{LM}_{\lambda\tau}(V)^{[i]})$.
	Then Lemma \ref{lem:exactup} again
	implies that there exits $\bar{w}\in \mathrm{Ker\,}\mathcal{LM}_{\lambda}(\theta_{V}^{\tau})$
	such that $\nabla(\bar{w})=\nabla(w)$,
	which implies that 
	$w-\bar{w}\in \mathrm{Ker\,}\nabla$,
	i.e., $w\in \mathrm{Ker\,}\mathcal{LM}_{\lambda}(\theta_{V}^{\tau})+\mathrm{Ker\,}\nabla$.
	Thus we obtain the converse inclusion $\subset$ and 
	therefore the equation $(\ref{eq:eqeq2})$. 

	In conclusion, equations $(\ref{eq:eqeq1})$ and $(\ref{eq:eqeq2})$
	give the desired equation.
\end{proof}
Let us consider the unique $k[F_{n}\rtimes_{\alpha}G]$-module homomorphism 
\[
	\widetilde{\nabla}\colon 
	\mathcal{LM}_{\lambda}(\mathcal{KLM}_{\tau}(V))
	\longrightarrow \mathcal{KLM}_{\lambda\tau}(V)	
\]
which makes 
the following diagram commutative.
\[
	\begin{tikzcd}
		0\arrow[r]&\mathcal{LM}_{\lambda}(\mathrm{Ker\,}\theta_{V}^{\tau})
		\arrow[r]\arrow[d,"\nabla"]
		&\mathcal{LM}_{\lambda}(\mathcal{LM}_{\tau}(V))
		\arrow[r,"\mathcal{LM}_{\lambda}(\theta_{V}^{\tau})"]
		\arrow[d,"\nabla"]
		&\mathcal{LM}_{\lambda}(\mathcal{KLM}_{\tau}(V))
		\arrow[r]\arrow[d,"\widetilde{\nabla}"]
		&0\\
		0\arrow[r]&\mathrm{Ker\,}\theta^{\lambda\tau}_{V}
		\arrow[r]
		&\mathcal{LM}_{\lambda\tau}(V)\arrow[r,"\theta_{V}^{\lambda\tau}"]
		&\mathcal{KLM}_{\lambda\tau}(V)\arrow[r]
		&0
	\end{tikzcd}
\]
\begin{crl}\label{cor:injective}
	Let $V$ be a left $k[F_{n}\rtimes_{\alpha}G]$-module satisfying $(P1)$.
	Let us take $\lambda,\tau\in k^{\times}\backslash\{1\}$.
	Then there uniquely exists 
	the $k[F_{n}\rtimes_{\alpha}G]$-module injective homomorphism 
	\[
		\widehat{\nabla}\colon 
		\mathcal{KLM}_{\lambda}(\mathcal{KLM}_{\tau}(V))
		\longrightarrow \mathcal{KLM}_{\lambda\tau}(V)
	\]
	which makes the diagram 
	\[
		\begin{tikzcd}
		\mathcal{LM}_{\lambda}(\mathcal{KLM}_{\tau}(V))
		\arrow[r,"\widetilde{\nabla}"]
		\arrow[d,"\theta_{\mathcal{KLM}_{\tau}(V)}^{\lambda}"']
		&\mathcal{KLM}_{\lambda\mu}(V)\\
		\mathcal{KLM}_{\lambda}(\mathcal{KLM}_{\tau}(V))
		\arrow[ur,"\widehat{\nabla}"']&
		\end{tikzcd}	
	\]
	commutative.
\end{crl}
\begin{proof}
	Let recall the following commutative diagram.
	\[
		\begin{tikzcd}
			0\arrow[d]&0\arrow[d]\\
			\mathcal{LM}_{\lambda}(\mathrm{Ker\,}\theta_{V}^{\tau})
			\arrow[r,"\nabla"]
			\arrow[d]
			&
			\mathrm{Ker\,}\theta^{\lambda\tau}_{V}
			\arrow[d]
			\\		
			\mathcal{LM}_{\lambda}(\mathcal{LM}_{\tau}(V))
			\arrow[r,"\nabla"]
			\arrow[d,"\mathcal{LM}_{\lambda}(\theta_{V}^{\tau})"]	
			&\mathcal{LM}_{\lambda\tau}(V)
			\arrow[d,"\theta_{V}^{\lambda\tau}"]\\
			\mathcal{LM}_{\lambda}(\mathcal{KLM}_{\tau}(V))
			\arrow[r,"\widetilde{\nabla}"]
			\arrow[d]	
			&\mathcal{KLM}_{\lambda\tau}(V)
			\arrow[d]\\
			0&0
		\end{tikzcd}	
	\]
	By Corollary \ref{cor:inverse}, we obtain 
	\[
		\mathcal{LM}_{\lambda}(\theta_{V}^{\tau})^{-1}(\mathrm{Ker\,}\theta_{\mathcal{KLM}_{\tau}(V)}^{\lambda})
		=\mathrm{Ker\,}\theta_{\mathcal{LM}_{\tau}(V)}^{\lambda} +\mathrm{Ker\,}\mathcal{LM}_{\lambda}(\theta_{V}^{\tau})
		=\nabla^{-1}(\mathrm{Ker\,}\theta_{V}^{\lambda\tau}).	
	\]
	This equation,  Lemma \ref{lem:surj}, and the above commutative diagram show 
	that 
	\begin{equation}\label{eq:eq2}
	\mathrm{Ker\,}\theta_{\mathcal{KLM}_{\tau}(V)}^{\lambda}= 
	\mathrm{Ker\,}\widetilde{\nabla}.
	\end{equation}
	Therefore $\widetilde{\nabla}$ factors through 
	the map
	$\theta_{\mathcal{KLM}_{\tau}(V)}^{\lambda}$, namely, 
	there uniquely exits 
	the injective $k[F_{n}\rtimes_{\alpha}G]$-module homomorphism 
	\[
		\widehat{\nabla}\colon 
		\mathcal{KLM}_{\lambda}(\mathcal{KLM}_{\tau}(V))
		\longrightarrow \mathcal{KLM}_{\lambda\tau}(V)
	\]
	and the following  diagram is commutative.
	\[
		\begin{tikzcd}
		\mathcal{LM}_{\lambda}(\mathcal{KLM}_{\tau}(V))
		\arrow[r,"\widetilde{\nabla}"]
		\arrow[d,"\theta_{\mathcal{KLM}_{\tau}(V)}^{\lambda}"']
		&\mathcal{KLM}_{\lambda\mu}(V)\\
		\mathcal{KLM}_{\lambda}(\mathcal{KLM}_{\tau}(V))
		\arrow[ur,"\widehat{\nabla}"']&
		\end{tikzcd}	
	\]
\end{proof}
Now we are ready for finishing the proof 
of Theorem \ref{thm:KLM}.
Let $V$ be an object in $\mathbf{Mod}_{k[F_{n}\rtimes_{\alpha}G]}^{\mathrm{NT}}$.
Let us take $\lambda,\tau\in k^{\times}$.
If $\lambda=1$ or $\tau=1$, Proposition \ref{prop:1middle}
shows that the multiplication map gives an isomorphism 
of functors 
$\mathcal{KLM}_{\lambda}\circ \mathcal{KLM}_{\tau}$ and 
$\mathcal{KLM}_{\lambda\tau}$.
Thus we may assume that $\lambda,\tau\in k^{\times}\backslash\{1\}$.
Then Corollary \ref{cor:injective} 
shows that the multiplication map induces the 
injective $k[F_{n}\rtimes G]$-module homomorphism 
\[
		\widehat{\nabla}\colon 
		\mathcal{KLM}_{\lambda}(\mathcal{KLM}_{\tau}(V))
		\longrightarrow \mathcal{KLM}_{\lambda\tau}(V)
\]
and this is surjective by Lemma \ref{lem:prethm1},
namely $\widehat{\nabla}$ gives
an isomorphism of functors 
$\mathcal{KLM}_{\lambda}\circ \mathcal{KLM}_{\tau}$ and 
$\mathcal{KLM}_{\lambda\tau}$.
\subsection{$F_{n}$-irreducibility}
Let us introduce the $F_{n}$-irreducibility of 
$k[F_{n}\rtimes_{\alpha}G]$-modules.
\begin{dfn}[$F_{n}$-irreducibility]
	For a left $k[F_{n}\rtimes G]$-module $V$,
	we say that $V$ is $F_{n}$-irreducible if 
	$V$ has no nontrivial $k[F_{n}]$-submodule.
\end{dfn}
Obviously, if $V$ is $F_{n}$-irreducible, then $V$ is irreducible 
as a $k[F_{n}\rtimes_{\alpha}G]$-module.
Moreover we can see that 
if $V$ is  $F_{n}$-irreducible, 
then $V$ is an object of the category $\mathbf{Mod}_{k[F_{n}\rtimes_{\alpha}G]}^{\mathrm{NT}}$.

Then it follows that the $F_{n}$-irreducibility is preserved by the KLM-functor as below.
The following is a direct corollary of Theorem\ref{thm:KLM}.
\begin{crl}[cf. Theorem 2.9.8 in \cite{Katz}, Corollary 2.5 in \cite{Vol},
	and Corollary 3.6 in \cite{DR}]
	If a left $k[F_{n}\rtimes_{\alpha}G]$-module $V$ is $F_{n}$-irreducible,
	then $\mathcal{KLM}_{\lambda}(V)$ is $F_{n}$-irreducible for any $\lambda\in k^{\times}$.
\end{crl}
\begin{proof}
	Suppose that $\mathcal{KLM}_{\lambda}(V)$ is not $F_{n}$-irreducible.
	Then it contains an irreducible $k[F_{n}]$-submodule $\widetilde{W}$.
	Now let us regard $\mathcal{KLM}_{\lambda}(V)$ and $\widetilde{W}$
	as objects of $\mathbf{Mod}_{k[F_{n}]}^{\mathrm{NT}}$,
	and apply $\mathcal{KLM}_{\lambda^{-1}}$ to them.
	Since the Katz-Long-Moody functor is an auto-equivalence of 
	the category $\mathbf{Mod}_{k[F_{n}]}^{\mathrm{NT}}$,
	inclusion relations are preserved, namely,
	we have $\mathcal{KLM}_{\lambda^{-1}}(\widetilde{W})
	\subset\mathcal{KLM}_{\lambda^{-1}}\circ \mathcal{KLM}_{\lambda}(V)\cong V$.
	Moreover since $\mathcal{KLM}_{\lambda}\circ \mathcal{KLM}_{\lambda^{-1}}(\widetilde{W})\cong \widetilde{W}$,
	$\mathcal{KLM}_{\lambda^{-1}}(\widetilde{W})\neq \{0\}$.
	Therefore it conclude that 
	$V$ contains the nontrivial $k[F_{n}]$-submodule $\mathcal{KLM}_{\lambda^{-1}}(\widetilde{W})$.
\end{proof}

\section{Constructing local systems via Katz-Long-Moody functor}\label{sec:locsys}
The Katz middle convolution 
is known to be a quite  efficient machinery to construct local systems 
on $\mathbb{C}\backslash\{n\text{-points}\}$ 
and to investigate their properties.
This machinery is called by a special name, the Katz algorithm.
On the other hand,
we have seen that 
many good properties of Katz middle convolution 
and 
Long-Moody functors for representations of braid groups 
are naturally unified in our framework of
KLM-functor. Thus
it is possible 
to 
extend the Katz algorithm to many other local systems through our KLM-functor.
Namely,
if $F_{n}\rtimes_{\alpha} G$ is the fundamental group of a topological space $X$,
the category $\mathbf{Mod}_{k[F_{n}\rtimes_{\alpha}G]}$ is equivalent to that of local systems of $k$-vector spaces over $X$.
Therefore if we repeatedly apply the KLM-functor and the multiplication functor to a local system over $X$,
then infinitely many non-isomorphic local systems over $X$ are obtained.

\subsection{$F_{n}\rtimes B_{n}$ as a fundamental group}
Let us  see that $F_{n}\rtimes B_{n}$ can be regarded as 
a fundamental group of a topological space.
Let $B^{n}(\mathbb{C})=\{\{c_{1},\ldots,c_{n}\}\subset \mathbb{C}\mid c_{i}\neq c_{j},\,i\neq j\}$ 
be the configuration space of unordered $n$ points in $\mathbb{C}$.
Here we regard $B^{n}(\mathbb{C})$ as a topological space 
via the bijection $B^{n}(\mathbb{C})\ni \{c_{1},\ldots,c_{n}\}\mapsto [(c_{1},\ldots,c_{n})]
\in \{ (a_{1},\ldots,a_{n})\in \mathbb{C}^{n}\mid a_{i}\neq a_{j}\text{ for }i\neq j \}/ \mathfrak{S}_{n}$.
Let us take $a_{i}=(i,0)\in \mathbb{R}^{2}= \mathbb{C}$, $i=1,\ldots,n$, as before 
and put $\mathbf{a}:=\{a_{1},\ldots,a_{n}\}\in B^{n}(\mathbb{C})$.
Then for any closed path $\gamma$ in $B^{n}(\mathbb{C})$ with the base point $\mathbf{a}$
defines the geometric braid
\[
	b_{\gamma}\colon [0,1]\ni t\longmapsto (\gamma(t),t)\in B^{n}(\mathbb{C})\times [0,1]	
\]
and this correspondence defines 
the isomorphism $\pi_{1}(B_{n}(\mathbb{C}),\mathbf{a})\cong B_{n}$.

Let us consider another topological space 
\[
B^{1,n}(\mathbb{C}):=\{
	(z,\{c_{1},\ldots,c_{n}\})\in \mathbb{C}\times B^{n}(\mathbb{C})\mid z\notin \{c_{1},\ldots,c_{n}\}\}.
\]
Then the projection map 
\[
b_{n}\colon B^{1,n}(\mathbb{C})\ni (z,\{c_{1},\ldots,c_{n}\})\longmapsto \{c_{1},\ldots,c_{n}\}
\in B^{n}(\mathbb{C})
\]
is a locally trivial fibration with the fiber $\mathbb{C}\backslash \{n\text{-points}\}$,
and the cross section 
\[
s_{b_{n}}\colon B^{n}(\mathbb{C})\ni \{c_{1},\ldots,c_{n}\}\longmapsto \left( 1+\sum_{i=1}^{n}|c_{i}|,\{c_{1},\ldots,c_{n}\}\right)
\in B^{1,n}(\mathbb{C}) 	
\]
defines the split short exact sequence 
\[
	1\rightarrow \pi_{1}(\mathbb{C}\backslash \{a_{1},\ldots,a_{n}\},
	s_{b_{n}}(\mathbf{a}))\rightarrow 
	\pi_{1}(B^{1,n}(\mathbb{C}),s_{b_{n}}(\mathbf{a}))
	\rightarrow
	\pi_{1}(B^{n}(\mathbb{C}),\mathbf{a})\rightarrow 1.
\]
Here we notice that the $n$-points set $\mathbf{a}$ was denoted by $Q_{n}$ in previous sections
and $D\backslash Q_{n}=D\backslash\{a_{1},\ldots,a_{n}\}$
is a deformation retract of $\mathbb{C}\backslash\{a_{1},\ldots,a_{n}\}$. 
Therefore the line segment $\overline{d\,s_{b_{n}}(\mathbf{a})}$
defines the isomorphism
\[
	\pi_{1}(\mathbb{C}\backslash \{a_{1},\ldots,a_{n}\},
s_{b_{n}}(\mathbf{a}))
\cong \pi_{1}(D\backslash Q_{n},d)=\langle x_{1},\ldots,x_{n}\rangle=F_{n}.
\]
The above split exact sequence gives the decomposition 
\[
\pi_{1}(B^{1,n}(\mathbb{C}),s_{b_{n}}(\mathbf{a}))\cong \pi_{1}(\mathbb{C}\backslash \{a_{1},\ldots,a_{n}\},s_{b_{n}}(\mathbf{a}))\rtimes \pi_{1}(B^{n}(\mathbb{C}),\mathbf{a})
\]
where the action of $\pi_{1}(B^{n}(\mathbb{C}),\mathbf{a})$ on 
$\pi_{1}(\mathbb{C}\backslash \{a_{1},\ldots,a_{n}\},s_{b_{n}}(\mathbf{a}))$ 
is given by the Artin representation under the
identifications $\pi_{1}(\mathbb{C}\backslash \{a_{1},\ldots,a_{n}\},s_{b_{n}}(\mathbf{a}))\cong F_{n}$
and 
$\pi_{1}(B^{n}(\mathbb{C}),\mathbf{a})\cong B_{n}$. Thus we obtain the isomorphism
\[
	\pi_{1}(B^{1,n}(\mathbb{C}),s_{b_{n}}(\mathbf{a}))\cong F_{n}\rtimes_{\theta_{\mathrm{Artin}}}B_{n},
\]
which gives a realization of $F_{n}\rtimes_{\theta_{\mathrm{Artin}}}B_{n}$ as 
the fundamental group of $B^{1,n}(\mathbb{C})$.

\subsection{Complements of polynomial coverings}
Let us recall the polynomial covering introduced by Hansen, see \cite{Han}.
Let $X$ be a $0$-connected topological space. A \emph{simple Weierstrass polynomial}
is a function $f\colon X\times \mathbb{C}\rightarrow \mathbb{C}$ of the form
\[
	f(x,z)=z^{n}+\sum_{i=1}^{n}a_{i}(x)z^{i},\quad (x,z)\in X\times \mathbb{C},
\]
where $a_{i}\colon X\rightarrow \mathbb{C}$, $i=1,\ldots,n$, are continuous functions
such that for any $x\in X$, the polynomial $f(x,z)$  of $z$ has no multiple roots.
For a simple Weierstrass polynomial $f$, the map
\[
	z_{f}\colon X\ni x\longmapsto \{z\in \mathbb{C}\mid f(x,z)=0\}\in B^{n}(\mathbb{C})	
\]
is called the \emph{root map} associated with $f$.
Then it is known that 
\[
	e_{f}\colon E_{f}:=\{(x,z)\in X\times \mathbb{C}\mid f(x,z)=0\}\ni (x,z)\longmapsto x\in X	
\]
becomes an $n$-fold covering, see \cite{Han}.  
It is moreover known that 
the complement $C_{f}:=(X\times \mathbb{C})\backslash E_{f}$ defines 
a fiber bundle as follows.
\begin{thm}[Hansen \cite{Han}, \text{M\o ller} \cite{Mol}, Cohen-Suciu \cite{CS}]\label{thm:moller}
	The projection map 
	\[
		c_{f}\colon C_{f}=\{(x,z)\in X\times \mathbb{C}\mid f(x,z)\neq 0\}\ni (x,z)\longmapsto x\in X
	\]
	is a locally trivial bundle, with the structure group $B_{n}$, and fiber $\mathbb{C}\backslash \{n\text{-points}\}$.
	Moreover this fibration has a cross section $s_{f}\colon X\rightarrow C_{f}$
	defined by 
	\[
		s_{f}(x):=\left(x, 1+\sum_{z\in z_{f}(x)}|z|\right)\in C_{f},
	\] 
	which induces the split short exact sequence 
	\[
		1\rightarrow \pi_{1}(\mathbb{C}\backslash z_{f}(x),s_{f}(x))\rightarrow \pi_{1}(C_{f},s_{f}(x))\rightarrow \pi_{1}(X,x)\rightarrow 1
	\]
	for an arbitrarily fixed $x\in X$ and 
	the decomposition 
	\[
		\pi_{1}(C_{f},s_{f}(x))\cong \pi_{1}(\mathbb{C}\backslash z_{f}(x),s_{f}(x))\rtimes \pi_{1}(X,x).
	\]
	Here in the semidirect product, 
	$\pi_{1}(X,x)$ acts on $\pi_{1}(\mathbb{C}\backslash z_{f}(x),s_{f}(x))$
	through the Artin representation 
	\[
		\pi_{1}(X,x)\xrightarrow[]{z_{f\,*}}\pi_{1}(B^{n}(\mathbb{C}),z_{f}(x)) \xrightarrow[]{\theta_{\mathrm{Artin}}} \mathrm{Aut}(F_{n})
	\]
	under the natural identifications $\pi_{1}(\mathbb{C}\backslash z_{f}(x),s_{f}(x))\cong F_{n}$
	and $\pi_{1}(B^{n}(\mathbb{C}),z_{f}(x))\cong B_{n}$.
\end{thm}
\begin{dfn}[Polynomial complement fibration]\normalfont
	Let $X$ be a $0$-connected topological space and $f\colon X\times \mathbb{C}\rightarrow \mathbb{C}$
	a simple Weierstrass polynomial. Then 
	we call the fibration $c_{f}\colon C_{f}\rightarrow X$ defined above the 
	\emph{polynomial complement fibration} with respect to $f$.
\end{dfn}
We note that  
the previously defined fibration $b_{n}\colon B^{1,n}(\mathbb{C})\rightarrow B^{n}(\mathbb{C})$
is a polynomial complement fibration 
with respect to the simple Weierstrass polynomial 
\[
B^{n}(\mathbb{C})\times \mathbb{C}\ni (\{c_{1},\ldots,c_{n}\},z)\longmapsto \prod_{i=1}^{n}(z-c_{i})\in \mathbb{C},
\]
and 
any polynomial complement fibration $c_{f}\colon C_{f}\rightarrow X$
can be obtained by the pull-back of $b_{n}$
with respect to the root map $z_{f}\colon X\rightarrow B^{n}(\mathbb{C})$,
i.e., 
we have $c_{f}=z_{f}^{*}(b_{n})$ and $C_{f}=z_{f}^{*}(B^{1,n}(\mathbb{C}))$, see \cite{Mol}.

Further we note that complements of polynomial coverings are closed under the pull-back. Namely, 
let $\phi \colon Y\rightarrow X$ be a continuous map between $0$-connected topological spaces and 
$c_{f}\colon C_{f}\rightarrow  X$ the polynomial complement fibration  with respect to a
simple Weierstrass polynomial $f$.
Then $g:=f\circ (\phi\times \mathrm{id}_{\mathbb{C}})\colon Y\times \mathbb{C}\rightarrow \mathbb{C}$
becomes a simple Weierstrass polynomial again and 
we can see that $c_{g}=\phi^{*}(c_{f})$ and $C_{g}=\phi^{*}(C_{f})$.

Therefore one can conclude that for an Artin representation $(G,\alpha)$ on $F_{n}$,
if 
$G$ is a fundamental group of a topological space $Y$ which has the homotopy type of a 
CW-complex,
then there exists a simple Weierstrass polynomial $f\colon Y\times \mathbb{C}
\rightarrow \mathbb{C}$ and the corresponding fibration 
$c_{f}\colon C_{f}\rightarrow Y$
such that 
\begin{equation}\label{eq:fundreal}
	\pi_{1}(C_{f})\cong F_{n}\rtimes_{\alpha}G.	
\end{equation}
This is verified as follows.
Recall that it is well-known that 
for every group homomorphism $\alpha \colon \pi_{1}(Y,y)\rightarrow \pi_{1}(B^{n}(\mathbb{C}),\alpha(y))$,
there is a
continuous map $\phi \colon Y\rightarrow B^{n}(\mathbb{C})$
up to homotopy equivalence 
such that $\alpha=\phi_{*}$, see \cite{Hat} for instance.
Then the pull-back of $b_{n}\colon B^{n,1}(\mathbb{C})\rightarrow B^{n}(\mathbb{C})$
by $\phi$ defines the polynomial complement fibration 
$c_{f}\colon C_{f}\rightarrow Y$ with $\phi\colon Y\rightarrow B^{n}(\mathbb{C})$
as the root map.
Thus Theorem \ref{thm:moller} gives us the desired isomorphism 
$(\ref{eq:fundreal})$.
\subsection{Complements of fiber-type arrangements}
We recall that complements of fiber-type arrangements have 
fundamental group of the forms $F_{n}\rtimes_{\alpha} G$ with some Artin representations $(G,\alpha)$.
\begin{dfn}[Strictly linearly fibered arrangements]\label{dfn:slf}\normalfont
	Let $\mathcal{A}$ and $\mathcal{B}$ be hyperplane arrangements
	in $\mathbb{C}^{l+1}$ and $\mathbb{C}^{l}$ respectively.
	Then we say that $\mathcal{A}$ is \emph{strictly linearly fibered}
	over $\mathcal{B}$ 
	if there exists 
	an isomorphism $\mathbb{C}^{l+1}\cong \mathbb{C}^{l}\times \mathbb{C}$ and 
	the following holds. Namely,
	under the identification $\mathbb{C}^{l+1}\cong \mathbb{C}^{l}\times \mathbb{C}$ chosen above,
		the projection $p\colon \mathbb{C}^{l}\times \mathbb{C}\ni (x,z)\mapsto x\in \mathbb{C}^{l}$
		satisfies that $p(\mathbb{C}^{l+1}\backslash \mathcal{A})=\mathbb{C}^{l}\backslash \mathcal{B}$ 
		and $p\colon \mathbb{C}^{l+1}\backslash \mathcal{A}\rightarrow\mathbb{C}^{l}\backslash \mathcal{B}$
		is a locally trivial fibration with the fiber $\mathbb{C}\backslash\{n\text{-points}\}$
		for some positive integer $n$.	
\end{dfn}
Let $\mathcal{F}_{n}(\mathbb{C}):=\{(c_{1},\ldots,c_{n})\in \mathbb{C}^{n}\mid c_{i}\neq c_{j},\,i\neq j\}$
be the configuration space of $n$ points in $\mathbb{C}$.
Then we can see that 
the projection $f_{n}\colon \mathcal{F}_{n+1}(\mathbb{C})\ni (c_{1},\ldots,c_{n},c_{n+1})
\mapsto (c_{1},\ldots,c_{n})\in \mathcal{F}_{n}(\mathbb{C})$
is a polynomial complement fibration
with respect to the simple Weierstrass polynomial 
$\mathcal{F}_{n}(\mathbb{C})\times \mathbb{C}\ni((c_{1},\ldots,c_{n}),z)
\mapsto \prod_{i=1}^{n}(z-c_{i})\in \mathbb{C}$.
Then the  following result by Cohen assures that the above 
fibration $p\colon \mathbb{C}^{l+1}\backslash \mathcal{A}\rightarrow\mathbb{C}^{l}\backslash \mathcal{B}$
is a polynomial complement fibration.
\begin{thm}[Cohen \cite{Coh}]
Let $\mathcal{A}$ be an arrangement of $m+n$ hyperplanes 
in $\mathbb{C}^{l+1}$ and $\mathcal{B}$ an arrangement of $m$ hyperplanes 
in $\mathbb{C}^{l}$. Suppose that $\mathcal{A}$ is strictly linearly fibered 
over $\mathcal{B}$. Then the fibration $p\colon \mathbb{C}^{l+1}\backslash \mathcal{A}\rightarrow\mathbb{C}^{l}\backslash \mathcal{B}$
in Definition \ref{dfn:slf} is a pull-back of the polynomial complement fibration
$f_{n}\colon \mathcal{F}_{n+1}(\mathbb{C})\rightarrow \mathcal{F}_{n}(\mathbb{C})$.
\end{thm}
Then the result in previous section 
shows that there exists an Artin representation $\alpha\colon \pi_{1}(\mathbb{C}^{l}\backslash \mathcal{B})
\rightarrow \mathrm{Aut}(F_{n})$ and 
the fundamental group $\pi_{1}(\mathbb{C}^{l+1}\backslash \mathcal{A})$
decomposes as the semi-direct product
\[
	\pi_{1}(\mathbb{C}^{l+1}\backslash \mathcal{A})\cong F_{n}\rtimes_{\alpha}\pi_{1}(\mathbb{C}^{l}\backslash \mathcal{B}).
\]

Now we introduces fiber-type arrangements.
\begin{dfn}[Fiber-type arrangements]\normalfont
	Let $\mathcal{A}$ be a hyperplane arrangement in $\mathbb{C}^{l}$.
	Then we say that $\mathcal{A}$ is \emph{fiber-type},
	if there exists a sequence $(\mathcal{A}_{k})_{k=1,\ldots,l}$ of 
	hyperplane arrangements in $\mathbb{C}^{k}$, $k=1,\ldots,l$, such that 
	$\mathcal{A}_{l}=\mathcal{A}$ and 
	$\mathcal{A}_{k+1}$ are strictly linearly fibered 
	over $\mathcal{A}_{k}$ for $k=1,\ldots,l-1$.
\end{dfn}
The above discussion shows that 
if $\mathcal{A}$ is a fiber-type arrangement in $\mathbb{C}^{l}$, $l\ge 2$, then 
there exists a sequence of Artin representations $\alpha_{k}\colon \pi_{1}(\mathbb{C}^{k}\backslash\mathcal{A}_{k})
\rightarrow \mathrm{Aut}(F_{d_{k}})$ with some positive integers $d_{k}$, $k=1,\ldots,l$,
and we have 
\begin{align*}
	\pi_{1}(\mathbb{C}^{l}\backslash \mathcal{A})&\cong F_{d_{l-1}}\rtimes_{\alpha_{l-1}}\pi_{1}(\mathbb{C}^{l-1}\backslash \mathcal{A}_{l-1})\\
	&\cong 	F_{d_{l-1}}\rtimes_{\alpha_{l-1}}F_{d_{l-2}}\rtimes_{\alpha_{l-2}}\pi_{1}(\mathbb{C}^{l-2}\backslash \mathcal{A}_{l-2})\\
	&\cdots\\
	&\cong	F_{d_{l-1}}\rtimes_{\alpha_{l-1}}F_{d_{l-2}}\rtimes_{\alpha_{l-2}}\cdots \rtimes_{\alpha_{1}}\pi_{1}(\mathbb{C}\backslash\mathcal{A}_{1})\\
	&\cong	F_{d_{l-1}}\rtimes_{\alpha_{l-1}}F_{d_{l-2}}\rtimes_{\alpha_{l-2}}\cdots \rtimes_{\alpha_{1}}F_{d_{0}}.
\end{align*}
Here we put $d_{0}$ to be the cardinality of the set $\mathcal{A}_{1}$ of points in $\mathbb{C}$.
\begin{rem}\normalfont
The braid arrangement $\mathcal{A}^{\mathrm{br}}_{\mathbb{C}}:=\{(x_{1},\ldots,x_{l})\in \mathbb{C}^{l}
\mid \prod_{i\le i<j\le l}(x_{i}-x_{j})=0\}$
is a typical example of fiber-type arrangements,
and it is well-known that the fundamental group of the complement of $\mathcal{A}^{\mathrm{br}}_{\mathbb{C}}$
is isomorphic to the pure braid group of degree $l$.
In the paper \cite{Har2},
Haraoka defines an endfunctor of the
category of $k$-local systems on $\mathbb{C}^{l}\backslash \mathcal{A}^{\mathrm{br}}_{\mathbb{C}}$
as a generalization of the Katz middle convolution.
Then we can check that this Haraoka's middle convolution is isomorphic 
to our KLM-functor in this case. The detail will be discussed in the forthcoming paper \cite{Neg}
by the second author.
\end{rem}

\subsection{Complements of closed braids in the solid torus}
We recall that fundamental groups of complements of closed braids in the solid torus
are of the forms $F_{n}\rtimes_{\alpha} G$ with Artin representations $(G,\alpha)$.

For $\beta\in B_{n}$, we write the closure of $\beta$ in the solid torus $T:=D\times S^{1}$ by
$\hat{\beta}$. 
Then we can see that the projection $T\backslash \hat{\beta}\rightarrow S^{1}$
becomes a fiber bundle with the fiber $D\backslash \{n\text{-points}\}$.
For a point $d\in \partial D$, 
we can take a cross section 
$s_{d}\colon S^{1}\ni x\mapsto (d,x)\in T\backslash \hat{\beta}$
of the bundle, which corresponds to a longitude of $T$.
Therefore we obtain the split short exact sequence 
\[
	1\rightarrow \pi_{1}(D\backslash \{n\text{-points}\},d)\rightarrow \pi_{1}(T\backslash \hat{\beta},s_{d}(s))
	\rightarrow\pi_{1}(S^{1},s)\rightarrow 1	
\]
for a fixed $s\in S^{1}$,
and the following holds. 
\begin{thm}[see Birman \cite{Bir} for example]
Let us take $\beta\in B_{n}$ and write the corresponding closed braid in the solid torus $T:=D\times S^{1}$ by
 $\hat{\beta}$. Then 
the fundamental group of the complement has the following decomposition,
\[
	\pi_{1}(T\backslash \hat{\beta})\cong F_{n}\rtimes_{\phi_{\beta}}\mathbb{Z},
\]
where the Artin representation
$\phi_{\beta}\colon \mathbb{Z}\rightarrow \mathrm{Aut}(F_{n})$
is defined by 
$1\mapsto \theta_{\mathrm{Artin}}(\beta)$.
\end{thm}

\subsection{A construction of integral representations of non-rigid local systems
via rigid local systems}
It is well-known that 
rank $2$ irreducible local systems on $\mathbb{C}\backslash\{0,1\}$ are 
realized as 
solution sheaves of Gauss hypergeometric differential equations and 
they are typical examples of 
rigid local systems in the sense of Katz \cite{Katz}.
On the other hand, rank $2$
local systems on $\mathbb{C}\backslash\{0,1,\lambda\}$
for some fixed $\lambda\in \mathbb{C}\backslash\{0,1\}$
are known 
to be the first example of non-rigid local systems,
we 
call these local systems \emph{Heun local systems},
because they correspond to solution sheaves of Heun differential equations.

It might not to be able to expect that 
solutions of linear differential equations 
corresponding to non-rigid local systems 
have integral representations 
like Gauss hypergeometric functions 
whose integrands are 
expressed by some elementary functions.
However our KLM-functor 
gives us some special examples of 
non-rigid local systems whose 
corresponding differential equations may have 
integral representations of their solutions.

Let us take an group homomorphisms 
$\alpha\in \mathrm{Hom}(F_{3},B_{4})$
and choose a continuous map $\phi_{\alpha}\colon \mathbb{C}\backslash \{0,1,\lambda\}
\rightarrow B^{4}(\mathbb{C})$
as before so that 
$(\phi_{\alpha})_{*}=\alpha$.
Then $p_{\alpha}\colon X_{\alpha}\rightarrow \mathbb{C}\backslash \{0,1,\lambda\}$
be the pull-back of $b_{4}\colon B^{1,4}(\mathbb{C})\rightarrow B^{4}(\mathbb{C})$
with respect to $\phi_{\alpha}\colon \mathbb{C}\backslash \{0,1,\lambda\}
\rightarrow B^{4}(\mathbb{C})$.
Then the cross section $s_{\alpha}\colon \mathbb{C}\backslash \{0,1,\lambda\}\rightarrow X_{\alpha}$
in Theorem \ref{thm:moller}
gives us the decomposition 
\[
	\pi_{1}(X_{\alpha})\cong F_{4}\rtimes_{\alpha}\pi_{1}(\mathbb{C}\backslash\{0,1,\lambda\})
\]
as before.
Let us suppose that there exists a character 
$\chi\colon F_{4}\rtimes_{\alpha}\pi_{1}(\mathbb{C}\backslash\{0,1,\lambda\})\rightarrow \mathbb{C}^{\times}$
such that 
\[
	\chi(x_{i})=1\text{ for }i=1,4,\quad \chi(x_{i})\neq 1\text{ for }i=2,3.
\]
Then if we take $\tau\in \mathbb{C}^{\times}\backslash \{\chi(x_{1}x_{2}x_{3}x_{4})^{-1}\}$,
a direct  computation tells us that $\mathcal{KLM}_{\tau}(\chi)$
becomes a rank $2$ local system.
Therefore the pull-back local system 
$s_{\alpha}^{*}(\mathcal{KLM}_{\tau}(\chi))$ 
with respect to the cross section
$s_{\alpha}\colon \mathbb{C}\backslash \{0,1,\lambda\}\rightarrow X_{\alpha}$
gives a Heun local system.
Since $\mathcal{KLM}_{\tau}(\chi)$ can be seen as a subspace 
of the Borel-Moore homology group
with rank $1$ local system as the coefficients,
we may conclude that 
$\mathcal{KLM}_{\tau}(\chi)$ gives an 
integral representation 
of this special Heun local system
whose integrand is the rank $1$ local system.

\end{document}